\documentclass[]{amsart}

\usepackage{amsmath}
\usepackage{amssymb}
\usepackage{amsthm}
\usepackage{amsfonts}
\usepackage[T1]{fontenc}
\usepackage[utf8]{inputenc}
\usepackage{lmodern}
\usepackage{textcomp}

\usepackage{xcolor}

\usepackage{tikz}
\usetikzlibrary{babel}			% babel-Kompatibilität auch in tikz-Grafiken

\usepackage{enumitem}

\usepackage{icomma}

\swapnumbers
\theoremstyle{plain}

\newtheorem{lemma}{Lemma}[section]
\newtheorem{thm}[lemma]{Theorem}
\newtheorem{prop}[lemma]{Proposition}
\newtheorem{cor}[lemma]{Corollary}
\newtheorem{rem}[lemma]{Remark}
\newtheorem*{mainthm}{Main Theorem}

\theoremstyle{definition}
\newtheorem{definition}[lemma]{Definition}

\newtheorem{ex}[lemma]{Example}

\newcommand{\N}{\mathbb{N}}

\newcommand{\coloneqq}{:=}
\newcommand{\nequiv}{\not \equiv}
\newcommand{\Sym}{\textrm{Sym}}

\newcommand{\erz}[1]{\langle {#1} \rangle}
\newcommand{\nt}{\trianglelefteq}
\newcommand{\syl}{\operatorname{Syl}}
\newcommand{\aut}{\operatorname{Aut}}
\newcommand{\inv}[1]{{#1}^{-1}}
\newcommand{\isom}{\cong}
\setlist[enumerate,1]{label = (\alph*), ref = (\alph*)}
\setlist[enumerate,2]{label = (\alph*), ref = (\labelenumi \alph*)}

\newenvironment{subproof}
{\begin{proof}}
	{\end{proof}}

	\usepackage{ulem}
	\usepackage{newclude}
	\usepackage{yfonts}
\begin{document}
	%\pagenumbering{gobble}
\begin{center}
\Large{\textbf{$L_9$-free groups}}

\vspace{0.2cm} \small{Imke Toborg, Rebecca Waldecker, Clemens B. Tietze\footnote{This article is partly based on the MSc thesis of the third author.}}
\end{center}
\normalsize
\newcommand{\FL}{\textfrak{L}}
\textbf{Abstract:}
In this article we classify all $L_9$-free finite groups.

\section*{intro}

There are some algebraic laws that hold in a lattice $L$ if and only if $L$ does not have a sublattice of a specific shape. For example, a lattice is modular if and only if it does not have a sublattice isomorphic to the so-called pentagon $L_5$.

If $L$ is a lattice, then we call a group $L$-free if and only if its subgroup lattice does not contain a lattice isomorphic to $L$.
For example, the finite $L_5$-free groups are exactly the modular groups, and these have been classified by Iwasawa in 1941, see \cite{Iwasawa}.
The subgroup lattice
of the dihedral group of order 8, often denoted by $L_{10}$, and some of its sublattices are of particular interest. One reason is that, if $p$ is a prime number, then a finite $p$-group is $L_5$-free if and only if it is $L_{10}$-free.

There are several sublattices of $L_{10}$ containing $L_5$:\\

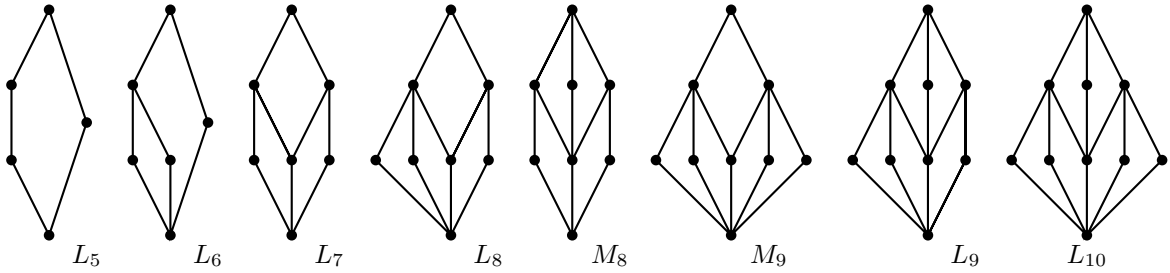
\begin{minipage}{1.5cm}
\begin{tikzpicture}
					\coordinate		(E) at (0.5,0);
					\coordinate		(F) at (0.5,3);
					\coordinate	(A) at (0,2);
					\coordinate			(T) at (0,1);
					\coordinate		(V) at (1,1.5);
					\coordinate[label=below:{$L_{5}$}]		(X) at (1,0);
					\foreach \x in {E,F,A,T,V} \fill (\x) circle (2pt);
						\draw [thick]	(E)  -- (T) -- (A)  -- (F) -- (V) -- (E);
				\end{tikzpicture}
\end{minipage}
\begin{minipage}{1.5cm}
\begin{tikzpicture}
					\coordinate		(E) at (0.5,0);
					\coordinate		(F) at (0.5,3);
					\coordinate	(A) at (0,2);
					\coordinate	(C) at (1,1.5);
					\coordinate			(T) at (0,1);
					\coordinate	(D) at (0.5,1);
					\coordinate[label=below:{$L_{6}$}]		(X) at (1,0);
					\foreach \x in {E,F,A,C,D,T} \fill (\x) circle (2pt);
						\draw [thick]	(E) -- (T) -- (A) -- (F) -- (C) -- (E) -- (D) -- (A);
				\end{tikzpicture}
\end{minipage}
\begin{minipage}{1.5cm}
\begin{tikzpicture}
					\coordinate		(E) at (0.5,0);
					\coordinate		(F) at (0.5,3);
					\coordinate	(A) at (0,2);
					\coordinate	(C) at (1,2);
					\coordinate			(T) at (0,1);
					\coordinate	(D) at (0.5,1);
					\coordinate		(U) at (1,1);
					\coordinate[label=below:{$L_{7}$}]		(X) at (1,0);
					\foreach \x in {E,F,A,C,D,T,U} \fill (\x) circle (2pt);
						\draw [thick]	(E) -- (T) -- (A) -- (F) -- (C) -- (U) -- (E) -- (D) -- (A) -- (D) -- (C);
				\end{tikzpicture}
\end{minipage}
\begin{minipage}{2cm}
\begin{tikzpicture}
					\coordinate		(E) at (0.5,0);
					\coordinate		(F) at (0.5,3);
					\coordinate	(A) at (0,2);
					\coordinate	(C) at (1,2);
					\coordinate			(S) at (-0.5,1);
					\coordinate			(T) at (0.,1);
					\coordinate	(D) at (0.5,1);
					\coordinate		(U) at (1,1);
					\coordinate[label=below:{$L_{8}$}]		(X) at (1,0);
					\foreach \x in {E,F,A,C,D,T,U,S} \fill (\x) circle (2pt);
						\draw [thick]	(E) -- (T) -- (A) -- (F) -- (C) -- (U) -- (E) -- (S) -- (A) -- (D) -- (C) -- (D) -- (E) ;
				\end{tikzpicture}
\end{minipage}
\begin{minipage}{1.5cm}
\begin{tikzpicture}
					\coordinate		(E) at (0.5,0);
					\coordinate		(F) at (0.5,3);
					\coordinate	(A) at (0,2);
					\coordinate		(B) at (0.5,2);
					\coordinate	(C) at (1,2);
					\coordinate			(T) at (0,1);
					\coordinate	(D) at (0.5,1);
					\coordinate		(U) at (1,1);
					\coordinate[label=below:{$M_8$}]		(X) at (1,0);
					\foreach \x in {E,F,A,B,C,D,T,U} \fill (\x) circle (2pt);
						\draw [thick]	(E) -- (T) -- (A) -- (F) -- (C) -- (U) -- (E) -- (D) -- (B) -- (F) -- (A) -- (D) -- (C);
				\end{tikzpicture}
\end{minipage}
\begin{minipage}{2.5cm}
\begin{tikzpicture}
					\coordinate		(E) at (0.5,0);
					\coordinate		(F) at (0.5,3);
					\coordinate	(A) at (0,2);
					\coordinate	(C) at (1,2);
					\coordinate			(S) at (-0.5,1);
					\coordinate			(T) at (0,1);
					\coordinate	(D) at (0.5,1);
					\coordinate		(U) at (1.5,1);
					\coordinate		(V) at (1,1);
					\coordinate[label=below:{$M_{9}$}]		(X) at (1,0);
					\foreach \x in {E,F,A,C,D,S,T,U,V} \fill (\x) circle (2pt);
						\draw [thick]	(E) -- (S) -- (A) -- (F) -- (C) -- (U) -- (E) -- (D) -- (A) -- (T) -- (E) -- (V) -- (C) -- (D);
				\end{tikzpicture}
\end{minipage}
\begin{minipage}{2cm}
\begin{tikzpicture}
					\coordinate		(E) at (0.5,0);
					\coordinate		(F) at (0.5,3);
					\coordinate	(A) at (0,2);
					\coordinate		(B) at (0.5,2);
					\coordinate	(C) at (1,2);
					\coordinate			(T) at (0,1);
					\coordinate	(D) at (0.5,1);
					\coordinate		(U) at (1,1);
					\coordinate		(S) at (-0.5,1);
					\coordinate[label=below:{$L_{9}$}]		(X) at (1,0);
					\foreach \x in {E,F,A,B,C,D,S,T,U} \fill (\x) circle (2pt);
						\draw [thick]	(E) -- (S) -- (A) -- (F) -- (B) -- (D) -- (A) -- (T) -- (E) -- (D) -- (C) -- (U) -- (E) -- (U) -- (C) -- (F);
				\end{tikzpicture}
\end{minipage}
\begin{minipage}{2.5cm}
\begin{tikzpicture}
					\coordinate		(E) at (1,0);
					\coordinate		(F) at (1,3);
					\coordinate	(A) at (0.5,2);
					\coordinate		(B) at (1,2);
					\coordinate	(C) at (1.5,2);
					\coordinate			(S) at (0.5,1);
					\coordinate			(T) at (0,1);
					\coordinate	(D) at (1,1);
					\coordinate		(U) at (1.5,1);
					\coordinate		(V) at (2,1);
					\coordinate[label=below:{$L_{10}$}]		(X) at (1,0);
					\foreach \x in {E,F,A,B,C,D,S,T,U,V} \fill (\x) circle (2pt);
						\draw [thick]	(E) -- (S) -- (A) -- (F) -- (B) -- (D) -- (A) -- (T) -- (E) -- (D) -- (C) -- (U) -- (E) -- (V) -- (C) -- (F);
				\end{tikzpicture}
\end{minipage}

\vspace{0.5cm}
In 1999 Baginski and Sakowicz \cite{BS} studied finite groups that are $L_6$-free and $L_7$-free at the same time, and later Schmidt \cite{Lfree} classified the finite groups that are $L_6$- or $L_7$-free. Together with Andreeva and the first author he also characterised, in \cite{L8M8}, all finite groups that are $L_8$-free or $M_8$-free. Finally, the finite $M_9$-free groups have been classified by P\"olzing and the second author in \cite{M9}.
Furthermore, there is a general discussion of $L_{10}$-free groups by Schmidt, which can be found in \cite{L10}
and \cite{L10pq}.

In this paper we investigate finite $L_9$-free groups.
Since $L_9$ is a sublattice of $L_{10}$, the groups that we consider are $L_{10}$-free and therefore
we can use Corollary C in \cite{L10}
as a starting point for our analysis:
 Every finite $L_{10}$-free group $G$ has normal Hall subgroups $N_1\leq N_2$ such that $N_1=\erz{P\in \mathrm{Syl}(G)\mid P\unlhd G}$, $N_2/N_1$ is a $2$-group and $G/N_2$ is meta-cyclic.

\smallskip Our strategy is to choose $N:=N_1$ maximal with respect to the above constraints, and then we show that $N$ has a complement $K$ that is a direct and coprime product of
groups of the following structure: cyclic groups, groups isomorphic to $Q_8$ or semi-direct products $Q\rtimes R$, where $Q$ has prime order and $R$ is cyclic of prime-power order such that $1\neq \Phi(R)=C_Q(R)$.
Furthermore, $[N,K]\cap C_N(K)$ is a $2$-group and $C_{O_2(N)}(K)$ is cyclic or elementary abelian of order $4$ and
every Sylow subgroup of $[N,K]$ is elementary abelian or isomorphic to $Q_8$.
If the action of $K$ on $N$ satisfies some more conditions, then we say that $NK$ is in class $\FL$.
\\The aim of our article is to prove the following theorem:

\begin{mainthm}
A finite group is in class $\FL$ if and only if it is $L_9$-free.
\end{mainthm}

\section{Notation and preliminary results}

In this article we mostly follow the notation from Schmidt's book \cite{Schmidt} and from
\cite{KurzStell}.
All groups considered are finite and $G$ always denotes a finite group, moreover $p$ and $q$ always denote prime numbers.
We quickly recall some standard concepts:

$L(G)$ denotes the \textbf{subgroup lattice of $G$}, consisting of the set of subgroups of $G$ with inclusion as the partial ordering.
The \textbf{infimum} of two elements $A,B \in L(G)$
is $A \cap B$ (their intersection) and the \textbf{supremum} is $A\cup B=\langle A,B\rangle$ (the subgroup generated by $A$ and $B$).

If $L$ is any lattice, then $G$ is said to be \textbf{$L$-free} if and only if $L(G)$ does not have any sub-lattice that
is isomorphic to $L$.

A lattice $L$ is said to be \textbf{modular} if and only if for all $X,Y,Z \in L$ such that $X \le Z$, the
following (also called the \textbf{modular law}) is true: $(X \cup Y ) \cap Z = X \cup (Y \cap Z).$
We say that a group $G$ is \textbf{modular} if and only if $L(G)$ is modular.

\medskip The modular law is similar to Dedekind's law (see 1.1.11 of \cite{KurzStell}). For all $X,Y,Z \leq G$ such that $X \le Z$ it says that
$X  Y  \cap Z = X  (Y \cap Z).$ We will use Dedekind's law frequently throughout this article without giving an explicit reference each time.

If $N\leq G$, then we say that an element $g\in G$ induces \textbf{power automorphisms} on $N$ if and only if $U^g=U$ for all subgroups $U$ of $N$.
Furthermore $\mathrm{Pot}_G(N):=\{g\in G\mid U^g=U \text{~f.a.~} U\leq N \}$ is a subgroup of $G$ because $\mathrm{Pot}_G(N)=\bigcap\limits_{U \le N} N_G(U)$.

\begin{lemma}\label{coprime} Let $K$ be a finite group that acts coprimely on the $p$-group $P$. Then
$P=[P,K]C_P(K)$. If $P$ is abelian, then this product is direct.
If $[P,K]\le \Phi(P)$, then $[P,K]=1$.
Furthermore $[P,K]=[P,K,K]$ and for all $K$-invariant normal subgroups $N$ of $P$ we have that $C_{P/N}(K)=C_P(K)N/N$.
\end{lemma}
\begin{proof}These statements are a collection from 8.2.2,  8.2.7, 8.2.9 and 8.4.2 of \cite{KurzStell}.
\end{proof}

\begin{lemma}\label{lem:MQvar}
Let $p \in \pi(G)$ and suppose that $G=PK$, where $P$ is a normal Sylow $p$-subgroup of $G$, $K \le G$ is a $p'$-group and $P_0:=[P,K] \neq 1$.
Suppose further that $\Phi(P_0)\leq C_P(K)$ and that $K$  acts irreducibly on $P_0/\Phi(P_0)$.
If $g \in P \setminus C_P(K)$, then  $P_0\leq \erz{[g,K],K}\leq \erz{g,K}$.
	\end{lemma}

\begin{proof}
Let $g \in P \setminus C_P(K)$ and $R:=\erz{[g,K],K}$. Then
we first remark that $R\leq \erz{g,K}$.
Lemma \ref{coprime} shows that $P=C_P(K)P_0$ and hence
we have elements $c\in C_P(K)$ and $h\in P_0$ such that $g=ch$.
We note that $P_0 \unlhd G$ and therefore $\Phi(P_0)$ is a normal subgroup of $G$, moreover $P_0$ is a $p$-group and hence $P_0/\Phi(P_0)$ is elementary abelian.
Let $-:G\to G/\Phi(P_0)$ denote the natural homomorphism.
As $\overline{P_0}=[\bar P,\bar K]$ is abelian
and $\bar K$ acts coprimely on it, we see that
$C_{[\bar P,\bar K]}(\bar K) \cap [\bar P,\bar K,\bar K]=1$, again by Lemma \ref{coprime}. Therefore
$C_{[\bar P,\bar K]}(\bar K) =C_{[\bar P,\bar K,\bar K]}(\bar K)=1$.
We recall that $ch=g\notin C_P(K)$ and thus $h\notin C_P(K)$, and then by hypothesis
$h \notin \Phi(P_0)$ and in particular $1 \neq \bar h \in \overline{P_0}$.
It follows that $[\bar h,\bar K]\neq 1$ because
$C_{[\bar P,\bar K]}(\bar K)=1$, see above.
We conclude  that $1\neq [\bar h,\bar K]=[\bar g,\bar K]=\overline{[g,K]}\le\overline{P_0}\cap \bar R$.
By hypothesis $K$ acts irreducibly on $\overline{P_0}$, hence $\bar K$ does as well and we see that
 $\overline{P_0}=\overline{P_0\cap R\Phi(P_0)}=\overline{(P_0\cap R)\Phi(P_0)}$.
The main property of the Frattini subgroup (see for example 5.2.3 in \cite{KurzStell}) finally gives that $P_0=P_0 \cap R$. \end{proof}

\begin{lemma}\label{power} Let $Q$ be a $2$-group that is elementary abelian, cyclic or isomorphic to $Q_8$.
Then $Q$ does not admit power automorphisms of odd order.
\end{lemma}
\begin{proof}If $Q$ is abelian, then the assertion follows from 2.2.5 of \cite{KurzStell}. If $Q\cong Q_8$, then $Aut(Q)\cong \Sym_4$, and any automorphism of order $3$ interchanges the maximal subgroups of $Q$.
\end{proof}

\begin{lemma}\label{lem:dedekind}\label{lem:Hilfslemma}
Suppose that $G=NK$, where $N$ is a normal Hall subgroup of $G$ and $K$ is a complement.
Let $N_1,N_2\leq N$, $Q,R\leq K$ and $x\in N$.
Then the following hold:
\begin{enumerate}
\item[(a)] If $N_1Q$ and $N_2R^x$ are subgroups of $G$, then $N_1Q\cap N_2R^x=N_1(Q\cap R)\cap N_2(Q\cap R)^x$.
\item[(b)] If $N_2R^x\leq G$ and $\langle x^{Q\cap R}\rangle \cap N_2=1$, then $Q\cap N_2R^x=Q\cap R^x\leq C_{Q\cap R}(x)$.
\item[(c)] If $K$ is abelian and acts irreducibly on the abelian group $N/\Phi(N)$ or if $N$ is abelian and $K$ induces power automorphisms on it, then $C_K(N)=C_K(x)$ or $x\in \Phi(N)$.
\item[(d)] If $Q\leq R$, then $\erz{Q,R^x}=\erz{[x,Q]^{R^x}}R^x$. %for all $x\in N$.% such that $[x,L_1]$ is normalised by $L_2^x$.
\end{enumerate}

\end{lemma}
\begin{proof}
Suppose that $N_1Q$ and $N_2R^x$ are subgroups of $G$.

For (a) we do the following calculation:

$\begin{array}{rcl}
N_1Q\cap N_2R^x&=& (N_1Q\cap NQ)\cap  N_2R^x
=N_1Q\cap (NQ\cap N_2R^x)
=N_1Q\cap N_2(NQ^x\cap R^x)\\
%\end{array}$$\begin{array}{rcl}
&=&N_1Q\cap N_2(NQ^x\cap (K^x\cap R^x))=N_1Q\cap N_2((NQ^x\cap K^x)\cap R^x)\\
%\end{array}$ $\begin{array}{rcl}
&=&N_1Q\cap N_2(Q^x\cap R^x)=N_1Q\cap (N(Q\cap R)^x\cap N_2(Q\cap R)^x)
\\ &=&(N_1Q\cap N(Q\cap R))\cap N_2(Q\cap R)^x=N_1(Q\cap N(Q\cap R))\cap N_2(Q\cap R)^x
\\ &=&N_1((Q\cap K)\cap N(Q\cap R))\cap N_2(Q\cap R)^x
\\&=&N_1(Q\cap (K\cap N(Q\cap R)))\cap N_2(Q\cap R)^x
\\ &=&N_1(Q\cap (Q\cap R))\cap N_2(Q\cap R)^x
=N_1(Q\cap R)\cap N_2(Q\cap R)^x.\end{array}$

\smallskip
Then (a) yields that
$Q\cap N_2R^x=(Q\cap R)\cap N_2(Q\cap R)^x$.
Therefore, if $\erz{x^{Q\cap R}}\cap N_2=1$ and if $a,b\in Q\cap R$ and $y\in N_2$ are such that $a=yb^x=yx^{-1}x^{b^{-1}}b$, then the fact that $ab^{-1}=yx^{-1}x^{b^{-1}}\in K\cap N=1$ implies that $a=b$ and that $y^{-1}=[x,b^{-1}]\in N_2\cap \erz{x^{Q\cap R}}=1$.
We deduce that $a=b^x=a^x\in Q\cap R^x$ and that $[x,a]=1$. Hence $Q\cap N_2R^x=(Q\cap R)\cap N_2(Q\cap R)^x=Q\cap R^x\leq C_{Q\cap R}(x)$.
This is (b).

If $K$ is abelian, then $x\in C_N(C_K(x))$ and $C_N(C_K(x))$ is $K$-invariant. Thus, if $K$ acts irreducibly on $N/\Phi(N)$ and $x\notin \Phi(N)$, then $C_N(C_K(x))\Phi(N)=N$, and the fact that $[C_K(x),N]=1$ implies that $C_K(x)=C_K(N)$.
If $K$ induces power automorphisms on the abelian group $N$, then 1.5.4 of \cite{Schmidt} implies that these are universal.
If $x=1$, then $x \in \Phi(N)$, and otherwise $x\neq 1$ and every element of $K$ that centralises $x$ also centralises $N$. Altogether (c) holds.

Suppose finally  that $Q\leq R$.
Then, for all $g\in Q$, we have that $g=g^x \cdot (g^{-1})^x \cdot g \in R^x [x,Q]\le \erz{[x,Q]^{R^x}}R^x$.
It follows that $\erz{Q,R^x}\leq \erz{[x,Q]^{R^x}}R^x$.
\\On the other hand $(g^{-1})^x\in Q^x\leq R^x$ for all $g\in Q$ and therefore $[x,g]=(g^{-1})^x g\in \erz{Q, R^x}$.

This implies that $[x,Q]\leq \erz{Q,R^x}$. Since $R^x\leq \erz{Q,R^x}$, we deduce that $\erz{[x,Q]^{R^x}}R^x\leq \erz{Q,R^x}$, which is (d).
\end{proof}

\section{Battens and batten groups}

\begin{definition}\label{defi:batten}
\begin{enumerate}
\item We say that $G$ is a \textbf{batten} if and only if $G$ is a cyclic $p$-group, or isomorphic to $Q_8$, or $G=QR$, where $Q$ is a normal subgroup of prime order and $R$ is a cyclic $p$-group of order coprime to $|Q|$ and such that $C_R(Q)=\Phi(R)\neq 1$.

\item We say that $G$ is a \textbf{batten group} if and only if $G$ is a direct product of battens of pairwise coprime order.

\item If $G$ is a batten group, then we say that $B \le G$ is a \textbf{batten of $G$} if and only if $B$ is a batten that is one of the direct factors of $G$.
\end{enumerate}
\end{definition}

\textbf{Warning:} It is possible for a subgroup of a batten group $G$ to be a batten, abstractly, but not to be a batten of $G$. This can happen when it is a $p$-subgroup for some prime $p$ of a batten as in the third case of Definition \ref{defi:batten}.

\begin{ex}\begin{enumerate}
    \item
Suppose that $Q:=\langle x\rangle$ is a group of order $19$ and that $R=\langle y\rangle$ is a subgroup of $\aut(X)$ of order $27$. Further suppose that $x^y:=x^7$. Then $B:=QR$ is a non-nilpotent batten. For this we calculate that $x^{y^3}=(x^7)^{y^2}=(x^{49})^y=(x^{11})^y=x^{77}=x$. Then the fact that $x^y=x^7\neq x$ implies that $C_R(Q)=\langle y^3\rangle=\Phi(R)$.
We note that $B$ is a batten group and that $Q$ is a subgroup of $B$ that is a batten, but not a batten of $B$ because $[Q,R] \neq 1$.

\item Let $B=QR$ be as in (a),  let $T\cong Q_8$ and let $S$ be a cyclic group of order $625$. Then $B\times T\times S$ is a batten group.
\end{enumerate}
\end{ex}

\begin{rem}
Let $G$ be a batten group and let $B$ be a batten of $G$ such that $|\pi(B)|= 2$.
\\(a) $B$ is not nilpotent, but $B$ has a unique normal Sylow subgroup.
\\(b) A Sylow subgroup $Q$ of $B$ is cyclic, and therefore $Q$ is batten. But $Q$ is not a direct factor of $G$ and hence $Q$ is not a batten of $G$.

\end{rem}

\begin{definition}\label{defi:B(G)}

Suppose that $G$ is a non-nilpotent batten.
Then there is a unique prime $q \in \pi(G)$ such that
$G$ has a normal Sylow $q$-subgroup $Q$, and $Q$ is cyclic of order $q$. In this case we set $\mathcal B(G):=Q$.
\end{definition}

From the definition we can immediately see that $\mathcal B(G)$
is a characteristic subgroup of a non-nilpotent batten $G$ and that it has prime order.

\begin{lemma}\label{lem:batten}
Suppose that $G$ is a non-nilpotent batten, that $r \in \pi(G)$ and that $R \in \syl_r(G)$ has order at least $r^2$. Then $Z(G)=C_R(\mathcal B(G)) =\Phi(R)=O_r(G)$.
%If $N\unlhd G$ such that $R\leq N$, then $N=G$.
\end{lemma}
\begin{proof}
Since $|R| \ge r^2$, we see that $R\neq \mathcal B(G)$.
Then Definition \ref{defi:batten} implies that $R$ is cyclic and that
there is a prime $q \in \pi(G)\setminus\{r\}$ such that  $Q:=\mathcal B(G)\in \syl_q(G)$.

Now $G=Q \rtimes R$ and $C_R(\mathcal B(G))=\Phi(R)$.
We recall that $R$ is cyclic, and then this implies that $\Phi(R)\le Z(G)$.
Since $G$ is not nilpotent, we see that $\mathcal B(G)$ is not contained in $Z(G)$. Thus $Z(G)$ is an $r$-group, because $\mathcal B(G)$ has prime order.
In addition $R\nleq Z(G)$, because $G$ is not nilpotent. Since $\Phi(R)$ is a maximal subgroup of the cyclic group $R$, it follows that $\Phi(R)=Z(G)$.
Moreover we have that $[Q,O_r(G)] \le Q \cap R=1$, whence $O_r(G) \le C_R(Q)=\Phi(R)= Z(G) \le C_R(Q)$.
This proves all statements.
\end{proof}

\begin{lemma}\label{battenremark}
If $G$ is a batten group and $P \le G$ is a Sylow $p$-subgroup of $G$ for some prime $p$, then $G$ has a subgroup of order $p$.
In addition,  $\Omega_1(P)\leq Z(G)$ or there is some non-nilpotent batten $B$ of $G$ such that $\Omega_1(P)=P=\mathcal B(B)$.
\end{lemma}

\begin{proof}
Since $P\leq G$, it follows that $P$ is cyclic or isomorphic to $Q_8$. Therefore $\Omega_1(P)$ has order $p$.
\\If $P$ is a batten, then $\Omega_1(P)\leq Z(P)\leq Z(G)$.
In particular $\Omega_1(P)$ is the unique subgroup of $G$ of its order.
Otherwise there is a non-nilpotent batten $B$ of $G$ such that $P\leq B$.
If $P$ has order $p$, then $P=\Omega_1(P)=\mathcal B(B)$ is a normal subgroup of $B$ and so of $G$. Again $\Omega_1(P)$ is the unique subgroup of $K$ of order $q$.
Otherwise we have that $\Omega_1(P)\leq \Phi(P)=Z(B)\leq Z(G)$ by Lemma \ref{lem:batten}.
In particular $\Omega_1(P)$ is the unique subgroup of $K$ of its order.
\end{proof}

\begin{lemma}\label{lem:battensub}
Suppose that $H$ is a batten and that $U \lneq H$. Then $U$ is a cyclic batten group.
Furthermore, all subgroups of batten groups are batten groups.
\end{lemma}
\begin{proof}
Assume for a contradiction that $U$ is not a cyclic batten group.
Then $U$ is not a cyclic batten, and therefore $H$ is neither cyclic of prime power order nor isomorphic to $Q_8$.
Thus $H$ is not nilpotent, in particular $|\pi(H)|=2$ and all Sylow subgroups of $H$ are cyclic. This implies that $U$ is not a $p$-group.
Let $\pi(H)=\{q,r\}$, let $Q:=\mathcal B(H)$ and $R \in \syl_r(H)$ be such that $H=QR$ and $C_R(Q)=\Phi(R) \neq 1$.
Now $\pi(U)=\{q,r\}$ as well and therefore $\mathcal B(H)\leq U$.
Then Dedekind's law gives that $U=\mathcal B(H) \cdot (U \cap R)$ is a proper subgroup of $H=\mathcal B(H) \cdot R$, and it follows that $U\cap R$ is a proper subgroup of $R$. In particular, since $R$ is cyclic, we have that $1\neq U\cap R\leq \Phi(R)$. %and so $|R| \ge r^2$.
Then Lemma \ref{lem:batten} gives that $\Phi(R)=Z(H)$.
Altogether $U\leq \mathcal B(H)\Phi(R)=\mathcal B(H)\times \Phi(R)$.
But then $U$ is a direct product of cyclic groups of prime power order, i.e. a cyclic batten group, and this is a contradiction.

Next suppose that $G$ is a batten group and that $U$ is a subgroup of $G$. Then $U$ is a direct product of subgroups of the battens of $G$ whose orders are pairwise coprime.
Consequently $U$ is a batten group as well, by the arguments above.  \end{proof}

We remark that sections of battens, or batten groups, are not necessarily batten groups. For example,  $Q_8/Z(Q_8)$ is not a batten group.

\begin{lemma}\label{lem:battenheart}\label{lem:battensyl}
Suppose that $K$ is a batten group and that $Q \leq K$ is a $q$-group.

If $Q$ is not normal in $K$, then there is a non-nilpotent batten $B$ of $K$ such that $B=\mathcal B(B) Q$ and $N_K(Q)=C_K(Q)$.
\\If $Q\unlhd K$, then $|K:C_K(Q)| \in \{1,4\}$ or this index is a prime number.\end{lemma}
\begin{proof} Since $K$ is a batten group, there is a batten $B$ of $K$ such that $Q\leq B$. Moreover there is a subgroup $L$ of $K$ such that $K=L\times B$. Then $L\leq C_K(B)\leq C_K(Q)$ ($\ast$).

We first suppose that $Q$ is not a normal subgroup of $K$.
Since $K$ is a direct product of battens, it follows that $Q$ is not normal in $B$.  Thus $B$ is neither abelian nor hamiltonian (otherwise all subgroups of $B$ would be normal), and
it follows that $B$ is not nilpotent. Now $Q$ is a proper subgroup of $B$ because $Q \not \unlhd B$.
We conclude from Lemma \ref{lem:batten} that neither $Q\leq \mathcal B (B)$ nor $Q\leq Z(B)$, whence $B=\mathcal B(B) Q$ and therefore $N_B(Q)=Q=C_B(Q)$. Consequently ($\ast$) and Dedekind's modular law yield that $N_K(Q)=LN_B(Q)=LC_B(Q)=C_K(Q)$.

Suppose now that $Q\unlhd K$. Using ($\ast$) we see that $ |K:C_K(Q)|=|B:C_B(Q)|$.
Hence we may suppose that $B$ is not abelian. If $B$ is not nilpotent, then $Z(B)$ and a Sylow $q$-subgroup of $B$ centralise $Q$. Thus Lemma \ref{lem:batten} yields that $ |K:C_K(Q)|=|B:C_B(Q)|$ equals the prime in $\pi(B)\setminus\{q\}$.
Let $B\cong Q_8$ and suppose that $Q\nleq Z(B)$. Then $Q$ has order $4$ or $8$.
In the first case $|K:C_K(Q)|=|B:C_B(Q)|=|B:Q|=2$ and in the second case $|K:C_K(Q)|=|B:C_B(Q)|=|B:Z(B)|=4$, which completes the proof.
\end{proof}

 \section{$L_{10}$ and its sublattices}

Throughout this article we will use the notation
from the next definition whenever we refer to
$L_{10}$ and its sublattices:

		\begin{definition}
			\label{defi:L10}
	The lattice $L_{10}$ is defined to be isomorphic to $L(D_8)$, with notation as indicated in the picture.		

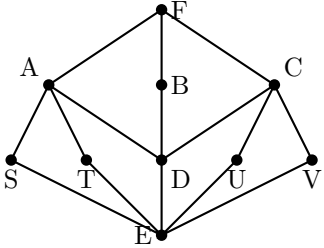
\begin{minipage}{5.2cm}	
%			\begin{figure}[h]
%				\centering
				\begin{tikzpicture}
					\coordinate[label=left:{E}]		(E) at (0,0);
					\coordinate[label=right:{F}]		(F) at (0,3);
					\coordinate[label=above left:{A}]	(A) at (-1.5,2);
					\coordinate[label=right:{B}]		(B) at (0,2);
					\coordinate[label=above right:{C}]	(C) at (1.5,2);
					\coordinate[label=below:{S}]			(S) at (-2,1);
					\coordinate[label=below:{T}]			(T) at (-1,1);
					\coordinate[label=below right:{D}]	(D) at (0,1);
					\coordinate[label=below:{U}]		(U) at (1,1);
					\coordinate[label=below:{V}]		(V) at (2,1);
					
					\foreach \x in {E,F,A,B,C,D,S,T,U,V} \fill (\x) circle (2.15pt);
					
					\draw [thick]	(E) -- (S) -- (A) -- (F) -- (B) -- (D) -- (A) -- (T) -- (E) -- (D) -- (C) -- (U) -- (E) -- (V) -- (C) -- (F);
				\end{tikzpicture}
			%	\caption{The lattice $L_{10}$}
			%	\label{pic:L10}
			%\end{figure}
	\end{minipage}\hfill\begin{minipage}{10cm}			
 Now we define

\smallskip\begin{minipage}{5cm}			\begin{enumerate}
				\item $L_5 := \{E, S, U, A, F\}$,
				\item[(c)] $L_7 := L_5 \cup \{D, C\}$,
				\item[(e)] $M_8 := L_{10} \setminus \{T, V\}$,
				\item[(g)] $M_9 := L_{10} \setminus \{B\}$,
			\end{enumerate}
   \end{minipage}\hfill
   \begin{minipage}{4.5cm}			\begin{enumerate}
				\item[(b)] $L_6 := L_5 \cup \{T\}$,
				\item[(d)] $L_8 := L_{10} \setminus \{B, V\}$,
				\item[(e)] $L_9 := L_{10} \setminus \{V\}$,
				\end{enumerate}
   \end{minipage}

\smallskip with the corresponding inclusion relations induced from the lattice $L_{10}$.	\end{minipage}
		\end{definition}

\begin{definition}Let $L$ be a lattice. An equivalence relation $\equiv$ on $L$ is called a congruence relation if and only if, for all $A,\ B,\ C,\ D\in L$ such that $A\equiv B$ and $C\equiv D$, we have that $\erz{A,C}\equiv\erz{B,D}$ and $A\cap C\equiv B\cap D$.
\end{definition}

	\begin{lemma}\label{lem:L9SubdirIrred}
Let $\equiv$ be a congruence relation on $L_9=\{A,\ B,\ C,\ D,\ E,\ F,\ U,\ T,\ S\}$  as in Definition \ref{defi:L10} and suppose
that $\equiv$ is not equality. Then $E \equiv D$.
	\end{lemma}
	\begin{proof}Let $\equiv$ be a congruence relation on $L_9$ and suppose that $E \nequiv D$.

If $X\in\{A,\ B,\ C,\ F\}$, then $E\cap D=E\nequiv D=X\cap D$ and therefore $E\nequiv X$.

First we assume that $X_0\in L_9\setminus\{F\}$ is such that $F\equiv X_0$.
Then there is some  $X \in \{A,\  B,\  C\}$ such that $X_0\leq X$, and then $X=\erz{X_0, X}\equiv \erz{F, X}=F$. We choose $Y\in \{T,U\}$ such that $X\cap Y=E$.
Then $E=X\cap Y\equiv F\cap Y=Y$ and therefore, if $Z\in\{T,U\}\setminus\{Y\}$, then $Z=\erz{Z,E}\equiv \erz{Z,Y}=F$.
Now it follows that $E=Z\cap D\equiv F\cap D=D$, which gives a contradiction.

We have seen that $E$ is not congruent %equivalent
to any of the elements $A,B,C,D,F$.
Next we assume that $X \in \{S, T, U\}$ is such that $E \equiv X$ and we choose $Y\in\{A,C\}$ such that $X\nleq Y$.
Then $Y=\erz{Y,E}\equiv\erz{Y,X}=F$, which gives another contradiction.
We conclude that $\{E\}$ and $\{F\}$ are singleton classes with respect to $\equiv$.

If there are
$X, Y \in L \setminus \{E, F\}$ such that $X \neq Y$ and $X \equiv Y$, $X \cap Y = E$ or $\langle X, Y\rangle = F$, then this implies that $X = X \cap X \equiv X \cap Y = E$ or $X = \langle X, X\rangle \equiv \langle X, Y\rangle = F$. As this is impossible, we conclude that such elements $X,\ Y$ do not exist.

In particular $A,\ B,\ C$ are pairwise non-congruent %equivalent
and $D,\  S,\ T,\ U$ are also pairwise non-congruent. %equivalent

We assume that $D\equiv X\in\{A,\ B,\ C\}$. Then we choose $Y\in \{T,\ U\}$ such that $Y\nleq X$, and this gives the contradiction $F\neq \langle Y, D\rangle \equiv \langle Y, X\rangle = F$. Hence $\{D\}$ is a singleton as well.

Finally, we assume that there are $X\in\{A,\ B,\ C\}$ and  $Y\in\{T,\ S,\ U\}$ such that $X\equiv Y$.
We have seen that $X\cap Y\neq E$ and then it follows that $Y\leq X$ by the structure of $L_9$. Now $D=X\cap D\equiv Y\cap D=E$, which is impossible.
In conclusion, for all $X,\ Y\in L_9$, we have that $X\equiv Y$ if and only if $X=Y$. This means that $\equiv$ is equality.	
\end{proof}
	
In the following lemma we argue similarly to Lemma 2.2 in \cite{Lfree}.

\begin{lemma}\label{lem:L9GdirProdTfOrd}
	Suppose that $n \in \N$ and that $G_1,...,G_n$ are normal subgroups of $G$ of
pair-wise coprime order such that $G=G_1\times ...\times G_n$.
	Then $G$ is $L_9$-free if and only if, for every $i\in\{1,...,n\}$, the group $G_i$ is $L_9$-free.
		\end{lemma}
	\begin{proof}
Since subgroups of $L_9$-free groups are $L_9$-free we just need to verify the ``if'' part.

Suppose that $G_1$,...,$G_n$ are $L_9$-free. Then Lemma 1.6.4 of \cite{Schmidt} implies that $L(G)\cong L(G_1)\times \cdots \times L(G_n)$.
By induction we may suppose that $n=2$. Assume that $L=\{E,\ T,\ S,\ U,\ D,\ A,\ B,\ C,\ F\}$ is a sublattice of
$L(G)\cong L(G_1)\times L(G_2)$ that is isomorphic to $L_9$ as in Definition \ref{defi:L10}.
Then the projections $\varphi_1$ and $\varphi_2$ of $L$ into
$L(G_1)$ and $L(G_2)$, respectively, are not injective, because $L(G_1)$ and $L(G_2)$ are $L_9$-free.
Let $i\in\{1,2\}$ and define, for all $X,\ Y\in L$:
$$X\equiv_i Y :\Leftrightarrow \varphi_i(X)=\varphi_i(Y).$$ Then $\equiv_i$ is a congruence relation on $L$, because $\varphi_i$ is a lattice homomorphism, but it is not equality because $\phi_i$ is not injective.
Then Lemma \ref{lem:L9SubdirIrred} implies that $\varphi_1(D)=\varphi_1(E)$
and $\varphi_2(D)=\varphi_2(E)$, and
hence $D=E$. This
is a contradiction.
\end{proof}

%Finding a sublattice that is isomorphic to $L_9$ will become easier with the next lemma. 		

\begin{lemma}\label{lem:L9Char}
		 $L_9=\{A,\ B,\ C,\ D,\ E,\ F,\ U,\ T,\ D\}$  as in Definition \ref{defi:L10} is completely characterised by the following:
			\begin{enumerate}
				\item[L9 (i)] ~~$D\neq E$.
		\item[L9 (ii)] \label{lem:p:L9Char:1} ~~$\langle S,T\rangle =
\langle S,D \rangle = \langle T,D\rangle = A$ and $S \cap T = S \cap D = T \cap D = E$.
				\item[L9 (iii)] \label{lem:p:L9Char:2} ~~$\langle D, U\rangle = C$ and $D \cap U = E$.
				\item[L9 (iv)] \label{lem:p:L9Char:4} ~~$\langle S,U \rangle= \langle T,U\rangle = F$.
		\item[L9 (v)] \label{lem:p:L9Char:3} ~~$\langle A, B\rangle  = \langle B, C\rangle = F$ and $A \cap B = A \cap C = B \cap C = D$.
		%Imke $\langle A,C\rangle$ ist ueberflussig $\langle A,C\rangle \geq \langle S,U\rangle=F$.
	
					\end{enumerate}
		\end{lemma}

\begin{proof}
 We first remark that $L_9$ satisfies the relations given in L9 (i) -- L9 (v).

Suppose conversely that a lattice $L=\{A,\ B,\ C,\ D,\ E,\ F,\ S,\ T,\ U\}$  satisfies the relations given in L9 (i) -- L9 (v).
Then we see that $E\leq A,\ B,\ C,\ D,\ S,\ T,\ U\leq F$ and $D\leq A,\ B,\ C$ as well as $S,\ T\leq A$ and $U\leq C$.  If these are the unique inclusions and $|L|=9$, then $L\cong L_9$.

For all $X,\ Y\in L$ such that $X\leq Y$ we have that $X\cap Y=X$ and $\langle X,\ Y\rangle=Y$.
Thus L9 (ii) shows that $S$, $T$ and $D$ are pair-wise not subgroups of each other.

Using L9 (iii) we obtain that $D\nleq U$ and $U\nleq D$, and L9 (iv) gives that also $S$, $T$ and $U$ are pair-wise not subgroups of each other.
In addition, by L9 (v), we have that $A$, $B$ and $C$ are pair-wise not subgroups of each other.
Together with the fact that $S,\ T\leq A$ and $U\leq C$, this implies that $A\nleq U$, $C\nleq S,T$ and $B\nleq S,\ T,\ U$.
Moreover we have that $D\leq A,\ B,\ C$ and $S,\ T\leq A$ and $U\leq C$ and $A\neq F\neq B$ and $C\neq F$.
Together with L9 (ii) and L9 (iii), this information yields that $A\ngeq U$ and $C\ngeq S,T$ as well as $B\ngeq S,\ T,\ U$.

We conclude that there is a lattice homomorphism $\varphi$ from $L_9$ to $L$.
Hence we obtain a congruence relation $\equiv$ on $L_9$ by defining that $X\equiv Y$ if and only if $\varphi(X)=\varphi(Y)$, for all $X,Y\in L_9$.

If $\varphi$ is not injective, then Lemma \ref{lem:L9SubdirIrred} implies that $E=D$. This contradicts L9 (i).
Consequently $\varphi$ is injective and $L\cong L_9$.
\end{proof}

	The next lemma gives an example of  a group that is not $L_9$-free.

	\begin{lemma}\label{lem:D12}
			$D_{12}$ is not $L_9$-free.
		\end{lemma}

		\begin{proof}
			Let $G$ be isomorphic to $D_{12}$ and let $a, b \in G$ be such that $o(a) = 6$, $o(b) = 2$ and $G = \erz{a, b}$.
			Then we find a sublattice in $L(G)$ isomorphic to $L_9$ by checking the equations from Lemma \ref{lem:L9Char}.

		We let
$L \coloneqq \{1, \erz{b}, \erz{a^2b}, \erz{a^2}, \erz{ab}, \erz{a^2, b}, \erz{a}, \erz{a^2, ab}, G\}$
and we define $A := \erz{a^2, b}$ and $C := \erz{a^2, ab}$.

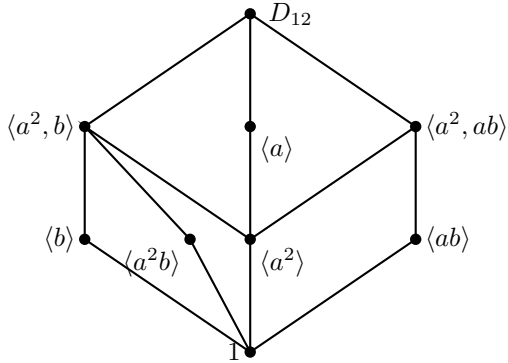
\begin{minipage}{7cm}
				\begin{tikzpicture}
				\coordinate[label=left:{$1$}]					(E) at (0,0);
				\coordinate[label=right:{~$D_{12}$}]				(F) at (0,4.5);
				\coordinate[label=left:{$\erz{a^2, b}$}]	(A) at (-2.2,3);
				\coordinate[label=below right:{$\erz{a}$}]			(B) at (0,3);
				\coordinate[label=right:{$\erz{a^2, ab}$}]		(C) at (2.2,3);
				\coordinate[label=left:{$\erz{b}$}]				(S) at (-2.2,1.5);
				\coordinate[label=below left:{$\erz{a^2b}$}]			(T) at (-0.8,1.5);
				\coordinate[label=below right:{$\erz{a^2}$}]	(D) at (0,1.5);
				\coordinate[label=right:{$\erz{ab}$}]			(U) at (2.2,1.5);
				
				\foreach \x in {E,F,A,B,C,D,S,T,U} \fill (\x) circle (2.15pt);
				
				\draw [thick]	(E) -- (S) -- (A) -- (F) -- (C)
				(E) -- (T) -- (A) -- (D) -- (C) -- (U) -- (E)
				(E) -- (D) -- (B) -- (F);
				\end{tikzpicture}
\end{minipage}\hfill
\begin{minipage}{8.5cm}
			\begin{enumerate}
			\item[L9 (i):] We see that $a^2\neq 1$ and hence $\erz{a}\neq 1$.			

				\item[L9 (ii):] We notice that $A \le G$ is isomorphic to $\Sym_3$ with
cyclic normal subgroup $\erz{a^2}$ of order $3$ and distinct subgroups $\erz{b}$, $\erz{a^2b}$ of order $2$.
Then $\langle \erz{b}, \erz{a^2b}\rangle = \langle \erz{b},\erz{a^2}\rangle = \langle\erz{a^2b}, \erz{a^2}\rangle = A$ and $\erz{b} \cap \erz{a^2b} = \erz{b} \cap \erz{a^2} = \erz{a^2b} \cap \erz{a^2} = 1$.
				\item[L9 (iii):] The subgroup $C$ is also isomorphic to $\Sym_3$, the subgroup $\erz{ab}\leq C$ has order $2$, and moreover $\langle \erz{a^2}, \erz{ab}\rangle = C$ and $\erz{a^2} \cap \erz{ab} = 1$.
							\end{enumerate}
\end{minipage}
\begin{enumerate}
\item[L9 (iv):] We first see that $\langle \erz{b}, \erz{ab}\rangle = \erz{a, b} = G$ and then $\langle \erz{a^2b}, \erz{ab}\rangle = \erz{a^2b\inv{(ab)}, ab} = \erz{a, b} = G$.
\item[L9 (v):] $\erz{a}$ is a cyclic normal subgroup of $G$ of order $6$, and the subgroups $A$ and $C$ also have order $6$. These three subgroups are maximal in $G$. Hence $\erz{A,\erz{a}}=\erz{C,\erz{a}}=G$.
Assume for a contradiction that $A = C$. Then $b \in C = \{1, a^2, a^4, ab, a^3b, a^5b\}$, which is false.
Since $\erz{a^2}$ is the unique subgroup of order $3$ of $G$, we conclude that $A \cap \erz{a} = A \cap C = \erz{a} \cap C = \erz{a^2}$. \end{enumerate}
				
Altogether it follows, with  Lemma \ref{lem:L9Char}, that $L$ is isomorphic to $L_9$, and then $G$ is not $L_9$-free.
		\end{proof}

The next lemma shows how we can construct an entire class of groups that are not $L_9$-free.

\begin{lemma}\label{lem:L9example}
Suppose that $p \neq q$ and that $G=PQ$,
where $P$ is an elementary abelian normal Sylow $p$-subgroup of $G$ and $Q$ is a cyclic Sylow $q$-subgroup of $G$.
Suppose that $Q$ acts irreducibly on $[P,Q]\neq 1$ and that $|C_P(Q)|\geq 3$. Then $G$ is not $L_9$-free.	\end{lemma}

\begin{proof} Since $Q$ is abelian, we see that $E:=C_Q(P)\unlhd G$. We claim that $G/C_Q(P)$ is not $L_9$-free.
Therefore we may suppose that $E=1$.

Our hypotheses imply that $[P,Q]$ is not centralised by $Q$, and in particular $|[P,Q]|\geq 3$. Moreover $|C_P(Q)|\geq 3$ by hypothesis.
Since $P$ is elementary abelian,
Lemma \ref{coprime} gives that $P=[P,Q]\times C_P(Q)$.

We let $V\leq [P,Q]$ and $D\leq C_P(Q)$ be
subgroups of minimal order such that $|V|\geq 3$ and $|D|\geq 3$ and we set $A:=V\times D$.
If $p$ is odd, then $A$ has order $p^2$, and if $p=2$, then $|A|=2^4=16$.
In the first case $A$ has $\frac{p^2-1}{p-1}=p+1\geq 4$ subgroups isomorphic to $V$.
In the second case $A$  has $\frac{15\cdot 14}{3}=70$ subgroups isomorphic to $V$, where $3\cdot \frac{14}{2}-2=19$ of these subgroups intersect $V$ non-trivially and $19$ of them intersect $D$ non-trivially.
In both cases, we find  subgroups $T$ and $S$  of $A$ isomorphic to $V$ such that $|\{D,T,S,V\}|=4$ and $T\cap V=T\cap D=T\cap S=S\cap D=S\cap V=1=E$.

We recall that $A$ is elementary abelian, and then it follows that L9 (i) and L9 (ii) hold and that $A=\erz{S,V}=\erz{T,V}$ ($\ast$).

We further set $U:=Q$. Then $U\cap D\leq Q\cap P=E=1$ and $\erz{U,D}=UD=:C$, which implies L9 (iii).

\begin{minipage}{8cm}		
If $X\in\{T,S\}$, then $X\nleq C_P(Q)$ and then the irreducible action of $Q$ on $[P,Q]$ and Lemma \ref{lem:MQvar} yield that $V\leq [P,Q]\leq \erz{X,U}$.
Using ($\ast$) it follows that  $D\leq A\leq \erz{V,X,U}=\erz{X,U}$.
Combining all this information gives that $\erz{X,U}=[P,Q]D Q$. Now if we set  $F:=[P,Q]D Q$, then we have L9 (iv).

To prove our claim, it remains to show that property L9 (v) of Lemma \ref{lem:L9Char} is satisfied.

We set $B:=DQ^x$ for some $x\in [P,Q]^\#$.

If $y\in\{1,\ x\}$, then $D\leq A\cap DQ^y=D(A\cap Q^y)=D$ and hence $A\cap B=A\cap C=D$.
We further have that $D\leq B\cap C=D(Q\cap DQ^x)=DC_Q(x)=D$, by Lemma \ref{lem:dedekind}~(b) and (c), because $Q$ acts irreducibly on $[P,Q]$.

\end{minipage}\hfill
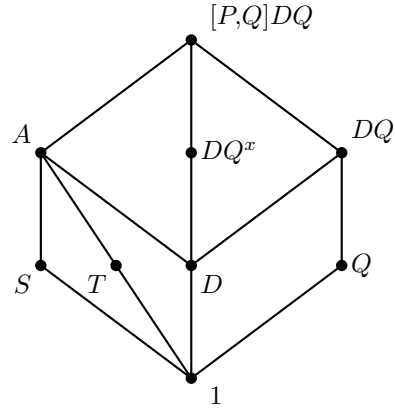
\begin{minipage}{7cm}			\begin{tikzpicture}
				\coordinate[label=below right:{~~$1$}]					(E) at (0,0);
				\coordinate[label=above right:{~~$[P,Q]D Q$}]				(F) at (0,4.5);
				\coordinate[label=above left:{$A$}]	(A) at (-2,3);
				\coordinate[label=right:{$DQ^x$}]			(B) at (0,3);
				\coordinate[label=above right:{$DQ$}]		(C) at (2,3);
				\coordinate[label=below left:{$S$}]				(S) at (-2,1.5);
				\coordinate[label=below left:{$T$}]			(T) at (-1,1.5);
				\coordinate[label=below right:{$D$}]	(D) at (0,1.5);
				\coordinate[label=right:{$Q$}]			(U) at (2,1.5);
				
				\foreach \x in {E,F,A,B,C,D,S,T,U} \fill (\x) circle (2.15pt);
				
				\draw [thick]	(E) -- (S) -- (A) -- (F) -- (C)
				(E) -- (T) -- (A) -- (D) -- (C) -- (U) -- (E)
				(E) -- (D) -- (B) -- (F);
				\end{tikzpicture}
\end{minipage}

\vspace{0.5cm}
In addition, the irreducible action of $Q^x$ on $[P,Q]$ and Lemma \ref{lem:MQvar} yield that $[P,Q]\leq \erz{A,Q^x}$.
It follows that  $\erz{A,B}=[P,Q]DQ^x=F$.
Finally we deduce from Part (d) of Lemma \ref{lem:dedekind} that
$$\erz{B,C}=D\erz{Q,Q^x}=D\erz{[x,Q]^{Q^x}}Q^x=D[P,Q]Q^x=F.$$
\end{proof}

\section{Group orders with few prime divisors}

%\textcolor{red}{Imke: Das folgende Lemma passt hier icht so wirklich hin, ich weiss aber nicht wo es besser waere}

Much of our analysis will focus on non-nilpotent groups with a small number of primes dividing their orders.
The next lemma %\ref{lem:mind1NT}
sheds some light on why this situation naturally occurs.
	
\begin{lemma}\label{lem:mind1NT}
			Suppose that $G$ is $L_9$-free. Then $G$ possesses a normal Sylow  subgroup.		\end{lemma}
		\begin{proof}
			Assume that this is false. Since $G$ is $L_9$-free and hence $L_{10}$-free,
[\cite{L10}, Corollary C] is applicable. Then $G$ is metacyclic because it does not have any normal Sylow subgroup, and it follows that $G$ is supersoluble. Then Satz VI.9.1(c) in \cite{Huppert} gives a contradiction.
		\end{proof}

	\begin{lemma}\label{lem:pGrModularÄqu}
			Suppose that $G$ is a $p$-group. Then the following statements are equivalent:
			\begin{itemize}
			\item $L(G)$ is modular.
			\item $G$ is $L_5$-free.
			\item $G$ is $L_9$-free.
			\item $G$ is $L_{10}$-free.
	 \end{itemize}
		\end{lemma}

\begin{proof}This lemma follows from Theorem 2.1.2 in \cite{Schmidt} and Lemma 2.1 in \cite{L10}, since $L_9$ is a sublattice of $L_{10}$ containing $L_5$.
\end{proof}

\begin{lemma}\label{lem:rank1}
Suppose that $p \neq q$ and that $G$ is an $L_9$-free $\{p,q\}$-group.
Let $P$ be a normal Sylow $p$-subgroup of $G$ and let $Q \in \syl_q(G)$.
If $G$ is not nilpotent, then $Q$ is cyclic or $Q\cong Q_8$.
	\end{lemma}
		
	\begin{proof}
First we note that all subgroups and sections of $G$ are $L_9$-free and that $G$ is $L_{10}$-free.

Let $G$ be non-nilpotent and assume for a contradiction that $Q$ is neither cyclic nor isomorphic to $Q_8$.
Given that $G$ is $L_{10}$-free, we may apply Theorem B of \cite{L10} and we see that neither (a), (b) nor (c) hold.
Therefore $p = 3$ and $q = 2$.
Now there are $a,b\in Q$ such that $\erz{a,b}$ is not cyclic and $b$ is an involution.
If, for all choices of $b$, we have that $C_P(b)=C_P(Q)$, then $\Omega_1(Q)$ acts element-wise fixed-point-freely on $P/C_P(Q)$, contradicting 8.3.4~(b) of \cite{KurzStell}.
Therefore we may choose $b$ such that $a$ does not centralise $C_P(b)$, and we also choose
$a$ of minimal order under these constraints.
Then $a^2$ centralises $P$ and $a$ inverts an element $x\in C_P(b)$ by a result of Baer (e.g. 6.7.7 of \cite{KurzStell}).
If follows that $a$ inverts $\Omega_1(\erz{x})$ and we may suppose that $x$ has order $3$.
Now $\erz{x,a,b}/C_{\erz{a}}(x)$ is isomorphic to $D_{12}$, contrary to
Lemma \ref{lem:D12}.
			\end{proof}

\begin{lemma}\label{lem:p2Qirred}
Suppose that $p \neq q$ and that $G$ is an $L_9$-free $\{p,q\}$-group.
Furthermore, let $P$ be a normal Sylow $p$-subgroup of $G$ and let $Q \in \syl_q(G)$ be cyclic such that $1\neq [P,Q]$ is elementary abelian.

Then every subgroup of $Q$ acts irreducibly or by inducing (possibly trivial) power  automorphism on $[P,Q]$. Moreover, $C_P(Q)$ is a cyclic $2$-group and $P$ is abelian.
\end{lemma}

\begin{proof}
First we note that all subgroups and sections of $G$ are $L_9$-free and that all subgroups and sections of $[P,Q]$ are elementary abelian, by hypothesis.  In addition Lemma 2.2 of \cite{L10} yields that $P=C_P(Q)\times [P,Q]$, as $G$ is also $L_{10}$-free.
In particular $P$ is abelian if $C_P(Q)$ is.

Assume that the lemma is false and let $G$ be a minimal counterexample.

Since $[P,Q]$ is elementary abelian, we introduce the following notation with Maschke's Theorem:
\\Let $n \in \N$ and let
$M_1, \dots, M_n \le [P,Q]$ be $Q$-invariant and such that $[P,Q] = M_1 \times \cdots \times M_n$ and that $Q$ acts irreducibly on $M_1,...,M_n$, respectively.
Lemmas 2.3.5 of \cite{Schmidt} and \ref{lem:pGrModularÄqu} yield that $\Omega_1(C_P(Q))$ is elementary abelian. Now there are $r \in \N$ and cyclic subgroups $M_{n+1},....,M_{n+r}$ of $\Omega_1(C_P(Q))$ such that $\Omega_1(C_P(Q))=M_{n+1}\times \cdots\times M_{n+r}$.
\\We set
$H_1:=(M_1\times \cdots\times M_{n+r-1})Q$ and $H_2:=(M_2\times \cdots\times M_{n+r})Q$.

Then for every $i\in\{1,2\}$ the group $O_p(H_i)$ is elementary abelian.
Moreover $H_i$ is a proper subgroup of $G$ and then the minimal choice of $G$ implies that every subgroup of $Q$ either induces (possibly trivial) power automorphisms on $[O_p(H_i),Q]$ or acts irreducibly on it.

\medskip (1)   $C_P(Q)=1$ and $n\leq 2$.

\begin{subproof}				
We assume for a contradiction that $n+r \geq 3$.

Then $Q$ does not act irreducibly on both $O_{p}(H_1)$ and $O_p(H_2)$, and it follows that $Q$ induces (possibly trivial) power automorphisms on $[O_p(H_i),Q]$.

We suppose first that $C_P(Q)=1$. Then $[O_p(H_i),Q]=O_p(H_i)$ for both $i\in\{1,2\}$. Therefore Lemma 1.5.4 of \cite{Schmidt}, together with the fact that $1\neq M_2\leq O_p(H_1)\cap O_p(H_2)$, provides some $k\in \mathbb N$ such that $a^y=a^k$ for every $a\in O_p(H_1) O_p(H_2)=[P,Q]\Omega_1(C_P(Q))$.
But this means that $Q$, and hence every subgroup of $Q$, induces (possibly trivial) power automorphism on $O_p(H_1)O_p(H_2)=[P,Q]$ in this case. Thus $G$ is not a counterexample, which is a contradiction.

We conclude that $C_P(Q)\neq 1$ and now there is some $i\in\{1,2\}$ such that  $1\neq [O_p(H_i),Q]$ and $C_{O_p(H_i)}(Q)\neq 1$. Then $H_i$ satisfies the hypotheses of our lemma and it follows that $C_{O_p(H_i)}(Q)$ is an non-trivial $2$-group. In particular $p=2$. But then Lemma \ref{power} provides the contradiction that $[O_p(H_i),Q]=1$.

For the proof of (1), we assume for a further contradiction that $r=n=1$.
Then $Q$ acts irreducibly on the elementary  abelian group $[P,Q]$ and
Lemma \ref{lem:L9example}, applied to $([P,Q]\times \Omega_1(C_P(Q)))Q$, gives that $|\Omega_1(C_P(Q))|= 2$.
In particular we have that $p=2$.
Thus the minimal choice of $G$ and Lemma \ref{power} yield, for every proper subgroup $U$ of $Q$, that $U$ centralises $P$ or acts irreducibly on $[P,Q]=[P,U]$.
\\Since $G$ is a counterexample, it follows that $C_P(Q)$ is not cyclic. But $|\Omega_1(C_P(Q))|= 2$ and therefore $C_P(Q)$ is a generalised quaternion group. It follows that $C_P(Q)\cong Q_8$ by Lemma \ref{lem:pGrModularÄqu}.
In this case $1\neq Z:=Z(C_P(Q))\unlhd G$ and $G/Z$ satisfies the hypotheses of our lemma, but not the conclusion. Thus $G$ is not a minimal counterexample, contrary to our choice.
\end{subproof}

(2)  $Q$ acts irreducibly on $P=[P,Q]$.

\begin{subproof}			
Assume for a contradiction that $n =2$.
By hypothesis $Q$ is cyclic, and then we may suppose that  $C_Q(M_1) \leq C_Q(M_2)=:Q_0$.
If $C_Q(M_1)=C_Q(M_2)$, then Lemma 2.8 of \cite{L10} implies that $Q$ induces power automorphisms on $P$. Thus $G$ is not a counterexample, which is a contradiction.

Therefore $C_Q(M_1) \lneq Q_0$ and $1\neq [M_1,Q_0]\leq [P,Q_0]$. The minimal choice of $G$ yields that $Q_0$ acts irreducibly or by inducing power automorphisms on $[P,Q_0]$ and that $C_P(Q_0)$ is a cyclic $2$-group.
Now we notice that $M_2\leq C_P(Q_0)$, but $M_2\nleq 1=C_P(Q)$, whence we deduce a contradiction from 2.2.5 of \cite{KurzStell}.
\end{subproof}

Since $G=PQ$ is a counterexample to the lemma, $Q$ has a proper subgroup $U$ that does not act irreducibly on $[P,Q]=P$ and it also does not induce power automorphisms on $[P,Q]$. In particular it does not act trivially. Since $PU$ is a proper subgroup of our minimal counterexample $G$, it follows that $1\neq C_P(U)\neq P$. But $C_P(U)$ is $Q$-invariant, because $Q$ is abelian.
This is a final contradiction with regard to (2).
\end{proof}

\begin{lemma}\label{lem:L9Ham}
Suppose that $q$ is odd and that $G$ is an $L_9$-free $\{2,q\}$-group.
Suppose further that $P$ is a normal Sylow $2$-subgroup of $G$ such that $[P,Q]$ is hamiltonian and let $Q \in \syl_q(G)$.
Then one of the following holds:

(a) $G$ is nilpotent or

(b) $[P,Q]\cong Q_8$ and
there exists a group $I$ of order at most $2$ such that
$P=[P,Q]\times I$ and $Q$ is a cyclic $3$-group. Moreover $[P,Q]Q/Z([P,Q]Q)\cong  \mathrm{Alt}_4$.
\end{lemma}

\begin{proof}We suppose that $G$ is not nilpotent.

Then $Q$ is not normal in $G$ and Lemma \ref{lem:rank1} implies that $Q$ is cyclic.
Furthermore, $P$ is $L_9$-free and hence it is modular by Lemma \ref{lem:pGrModularÄqu}. Since  $[P,Q]$ is hamiltonian, Theorem 2.3.1 of \cite{Schmidt} provides subgroups $P_0, I \leq P$ such that $P_0 \isom Q_8$, $I$ is elementary abelian and $P = P_0 \times I$.

We recall that the automorphism group of $Q_8$ is isomorphic to $\Sym_4$.
Thus, if $Q_8\cong P_1\leq P$ is $Q$-invariant, but not centralised by $Q$, then $Q_8\cong P_1\leq [P,Q]$ and $1\neq |Q/C_Q(P_1)|=3$.
It follows that $Q$ is a cyclic $3$-group and  $[P_1,Q]Q/Z([P_1,Q]Q)\cong  \mathrm{Alt}_4$.

We conclude that our assertion holds if $I=1$.
Now suppose that $I \neq 1$.
We recall that $P \unlhd G$ and therefore $\Phi(P_0) = \Phi(P) \nt G$. Since $|\Phi(P_0)|=2$, it follows that $\Phi(P_0) \le Z(G)$, and then $\bar G:=G/\Phi(P_0)$ is not nilpotent because $G$ is not.
Furthermore, $\bar G$ is $L_9$-free, $\bar I \isom I \neq 1$, $\bar Q \isom Q$ and $\bar P_0$ is elementary abelian of order $4$. In particular $\bar P$ is elementary abelian, hence it is a non-hamiltonian $2$-group of order at least $8$.
Lemma \ref{lem:p2Qirred} states that $\bar Q$ acts irreducibly on $[\bar P, \bar Q]\neq 1$ or induces power automorphisms on it. The second case is not possible by Lemma \ref{power}.

Hence $Q\cong \bar Q$ acts irreducibly on $[\bar P,\bar Q]=\overline{[P,Q]}$ and, by Lemma \ref{lem:p2Qirred}, we see that $C_{\bar P}(\bar Q)$ is a cyclic $2$-group.
Since $I\neq 1$ and $\Omega_1(P)=\Phi(P_0)\times I$,  we have that $1\neq \bar I=\overline{\Omega_1(P)}$ is $\bar Q$-invariant and $\bar P_0$ is a non-cyclic complement of $\bar I$ in $\bar P$.
This implies that $\bar I=C_{\bar P}(\bar Q)$ is cyclic and elementary abelian at the same time.
Thus $I\cong \bar I$ has order $2$ and with Lemma \ref{coprime} we deduce that $C_P(Q)\Phi(P_0)=I\Phi(P_0)=\Omega_1(P)$ is elementary abelian of order $4$. Then we deduce that $C_P(Q)=\Omega_1(P)$ and then $[P,Q]\cap C_P(Q)\neq 1$. Moreover, since $[P,Q]C_P(Q)=P=P_1\times I\leq P_1C_Q(P)$, it follows that $|[P,Q]|\leq |P_1|\leq 8$.
In conclusion, $[P,Q]$ is a subgroup of order at most $8$ admitting an automorphism of odd order that centralises $\Omega_1([P,Q])$. It follows that $[P,Q]\cong Q_8$ and then, together with the fact that $I\leq Z(G)$, our assertions follow.
\end{proof}

\begin{definition}\label{defi:avoidcyclic}

Suppose that $Q$ is a cyclic $q$-group that acts coprimely on the $p$-group $P$.
We say that the action of $Q$ on $P$ \textbf{avoids $L_9$}  (and we indicate more technical details by writing \textbf{"of type ($\cdot$)"})
if and only if one of the following is true:

\begin{enumerate}
\item[(std)] Every subgroup of $Q$ acts irreducibly or by inducing (possibly trivial) power automorphisms on the elementary abelian group $[P,Q]=P$.
\item[(cent)] Every subgroup of $Q$ acts irreducibly or trivially
%Imke: Elementarabelsche 2-Gruppen erlauben keine nichttrivialen Potenzautomorphismen.
on the elementary abelian group $[P,Q]$, $P$ is abelian and $C_P(Q)$ is a non-trivial cyclic $2$-group.
\item[(hamil)] $[P,Q]\cong Q_8$, $P=[P,Q]\times I$, where $I$ is a group of order at most $2$, and $Q$ is a cyclic $3$-group such that $[P,Q]Q/Z([P,Q]Q)\cong \textrm{Alt}_4$.
\end{enumerate}
\end{definition}

\begin{lemma}\label{lem:p2QQ8}
Suppose that $p$ is an odd prime and that $G$ is an $L_9$-free $\{2,p\}$-group.
Let further $P$ be a normal Sylow $p$-subgroup of $G$ and let $Q \in \syl_2(G)$ be isomorphic to $Q_8$ and such that $1\neq [P,Q]$ is elementary abelian.

Then $p\equiv 3\mod 4$, $|P|=p^2$ and $Q$ acts faithfully on $P$.
\end{lemma}

\begin{proof}
We set $Z:=\Omega_1(Q)$.
If $Z \nt G$, then $Z \le Z(G)$ and we consider $\bar G:=G/Z$.
Then $\bar G$ is an $L_9$-free $\{2,p\}$-group, $\bar P$ is a normal Sylow $p$-subgroup of $\bar G$ and  $\bar Q \in \syl_2(\bar G)$.
Since $\bar Q$ is neither cyclic nor isomorphic to $Q_8$, Lemma \ref{lem:rank1} is applicable and we see that $\bar G$ is nilpotent.
But then $G$ is also nilpotent, contrary to our hypothesis that $[P,Q]\neq 1$.
Thus $\Omega_1(Q)$ is not normal in $G$ and $Q$ acts faithfully on $P$.
Now, for all $y\in Q$ of order $4$, we apply Lemma \ref{lem:p2Qirred} on $[P,Q]\erz{y}$  to deduce that $\erz y$ either induces power automorphisms on $[P,Q]$ or acts irreducibly on it.
Theorem 1.5.1 of \cite{Schmidt} states that  $\mathrm{Pot}_G(P)$ is abelian, but $Q \cong Q_8$ is not, which means that we may choose $y$ such that $y$ does not induce power automorphisms on $P$.
In particular $P$ is not cyclic of prime order.
Moreover $p$ is odd and therefore $4$ divides $(p + 1)(p - 1)= p^2 -1$, and Satz II 3.10 of \cite{Huppert}  yields that  $|P|\le p^2$. It follows that $|P|=p^2$.
More precisely, as $|P|\neq p$, the result implies that $p \equiv 3 \pmod{4}$, and then the proof is complete.
\end{proof}

\begin{definition}\label{defi:avoidQ8}
Suppose that $Q\cong Q_8$ acts coprimely on the $p$-group $P$.
We say that the action of $Q$ on $P$ \textbf{avoids $L_9$}
if and only if $p\equiv 3\mod 4$, $|P|=p^2$ and $Q$ acts faithfully on $P$.
\end{definition}

\begin{lemma}\label{lem:MOP4}	
Suppose $Q\cong Q_8$ and that $P$ is a $p$-group on which $Q$ acts avoiding $L_9$.
Then $P$ is elementary abelian, $\Omega_1(Q)$ inverts $P$, and every subgroup of $Q$ of order at least $4$ acts irreducibly on $P$.
\end{lemma}
\begin{proof} Since a cyclic group of order $p^2$ has an abelian automorphism group by 2.2.3 of  \cite{KurzStell}, it follows that $P$ is elementary abelian.
If $1\neq R$ is a cyclic subgroup of $P$, then $|\mathrm{Aut}(R)|=p-1$ and therefore $R$ does not admit an automorphism of order $4$.
Additionally, $[P,\Omega_1(Q)]$ is $Q$-invariant and, since $Q$ acts faithfully on $P$, we see that $[P,\Omega_1(Q)]\neq 1$.
Furthermore, Lemma \ref{coprime} gives that $[P,\Omega_1(Q)]\cap C_P(\Omega_1(Q))=1$ because $P$ is abelian. Moreover, $Q$ has rank $1$, and then it follows that $Q$ acts faithfully on $[P,\Omega_1(Q)]$.
This implies that $|[P,\Omega_1(Q)]|\neq p$ and consequently $[P,\Omega_1(Q)]=P$.
Hence 8.1.8 of \cite{KurzStell} states that $\Omega_1(Q)$ inverts $P$.
In particular $\Omega_1(Q)$ inverts every cyclic subgroup $R$ of $P$.

In addition, these arguments show that every subgroup $U$ of order $4$ of $Q$ does not normalise any non-trivial proper subgroup of the elementary abelian group $P$.
This means that $U$ acts irreducibly on $P$.
\end{proof}

\begin{cor}\label{kor:2prim}
Suppose that $p \neq q$ and that $G$ is an $L_9$-free $\{p,q\}$-group such that $P\in\syl_p(G)$ is normal in $G$ and $Q\in\syl_q(G)$.
\\Then either $G$ is nilpotent and $P$ and $Q$ are modular or $Q$ is a batten and it acts on $P$ avoiding $L_9$.

In particular, if $G$ is not nilpotent, then $Q$ is isomorphic to $Q_8$ or cyclic and $[P,Q]$ is elementary abelian or isomorphic to $Q_8$, where in the second case $q=3$.
\end{cor}

\begin{proof}
By hypothesis $G$ is $L_9$-free, hence $P$ and $Q$ are, too.
Then Lemma \ref{lem:pGrModularÄqu} implies that $P$ and $Q$ are modular.
\\Suppose that $G$ is not nilpotent. Then Lemma \ref{lem:rank1} applies: $Q$ is cyclic or isomorphic to $Q_8$ and hence it is a batten.
Moreover Lemma 2.2 of \cite{L10} states that $[P,Q]$ is a hamiltonian $2$-group or  elementary abelian.
In the first case Lemma \ref{lem:L9Ham} gives the assertion.
In the second case our statement follows from Lemmas \ref{lem:p2Qirred}  and \ref{lem:p2QQ8}.
\end{proof}

\begin{lemma}\label{lem:lactirreorpower}
Let $Q$ be a nilpotent batten that acts on the $p$-group $P$ avoiding $L_9$, and suppose that $U$ is a subgroup of $Q$.

Then $U$ induces power automorphisms on $P$ or it acts irreducibly on $[P,Q]/\Phi([P,Q])$.
\end{lemma}	

\begin{proof}
First suppose that $Q\cong Q_8$. Then Lemma \ref{lem:MOP4} implies that every subgroup of order at least $4$ of $Q$, and in particular $Q$ itself, acts irreducibly on $P=[P,Q]/\Phi([P,Q])$. Moreover, the involution of $Q$ inverts $P$ by Lemma \ref{lem:MOP4}, and then the statement holds.

Next we suppose that $Q$ is cyclic. Then Definition \ref{defi:avoidcyclic} gives the assertion unless the action of $Q$ on $P$ avoids $L_9$ of type (hamil). In this case every proper subgroup of $Q$ centralises $P$, while $Q$ acts irreducibly on $[P,Q]/Z([P,Q])=[P,Q]/\Phi([P,Q])$.
\end{proof}

Next we investigate groups of order divisible by more than two primes. This needs some preparation.

\begin{lemma}\label{lem:exam}
Suppose that $P$ and $R$ are distinct Sylow subgroups of $G$, that $Q \in \syl_q(G)$ is  cyclic and that it normalises $P$ and $R$, but does not centralise them.
Suppose further that $R$ normalises every $Q$-invariant subgroup of $P$.
If $C_Q(P)= C_Q(R)$, then $G$ is not $L_9$-free.
\end{lemma}

\begin{proof}
We suppose that $C_Q(P)= C_Q(R)=:E$.
Then $E$ is a normal subgroup of $G$ because $Q$ is abelian.

We claim that $G/E$ is not $L_9$-free, and for this we may suppose that $E=1$.
Then $Q\neq 1$ acts faithfully on $P$ and $R$.
Now we need a technical step before we move on:

There are a $Q$-invariant subgroup $D$ of $P$ and elements $x,y\in D$ such that
$C_Q(x)=C_Q(y)=C_Q(xy^{-1})=1$ and $D=[x,Q]=[y,Q]=[xy^{-1},Q]$. ~~($\ast$)

 Since $Q$ is cyclic, there is some $u\in Q$ such that $\erz{u}=Q$.
Let $x_0\in [P,Q]$ be such that $\Omega_1(Q)$ does not centralise $x_0$.
\\Then $[x_0,Q]=[x_0,\erz{u},Q]=[\erz{[x_0,u]},Q]=[[x_0,u],Q]$ by Lemma \ref{coprime}, and  for all integers $n$ we have the following:
$(x_0^{-1}x_0^u)^{u^n}=1$ iff $x_0^{u^n}=x_0^{u^{n+1}}$ iff $x_0^u=1$.
It follows that  $[x_0,u]\in [[x_0,u],Q]=[x,Q]$ and $C_Q([x_0,u])=C_Q(x_0)=1$.

Now we set $x:=[x_0,u]$, $y:=x^u$ and $D:=[x,Q]$.
Then $D$ is $Q$-invariant and we have that $x,y\in D$, $C_Q(x)=C_Q(y)=1$ and $D=[x,Q]=[y,Q]$.
If we set $z:=xy^{-1}$, then $z=[x^{-1},u]$ and we can use the information from the end of the previous paragraph:\\
$[z,Q]=[x^{-1},Q]=D$ and $C_Q(z)=C_Q(x^{-1})=C_Q(x)=1$. This concludes the proof of ($\ast$).

\vspace{0.2cm}
We use ($\ast$) and its notation and, similarly, we find a $Q$-invariant subgroup $R_0$ of $R$ and an element $h\in R_0$ such that $C_Q(h)=1$ and $R_0=[h,Q]$.

\smallskip \begin{minipage}{11.3cm}
We set  $S:=Q^x$, $T=Q^y$, $A:=DQ$, $B:=DR_0$, $U:=Q^h$, $C:=DQ^h$
and $F:=DR_0Q$, and we claim that $\{A,\ B,\ C,\ D,\ E,\ F,\ S,\ T,\ U\}$ is isomorphic to $L_9$.
The properties L9 (i) and  L9 (iii) of Lemma \ref{lem:L9Char} follow from the choice of $D$, since $h$ and $Q$ normalise $D$.	

\smallskip For L9 (ii) we first note that $D\cap Q^x=D\cap Q^y=1=E$ and $\erz{D,Q^x}=\erz{D,Q^y}=DQ$, since $x\in D$ and hence $y=x^u\in D$.
Next, Lemma \ref{lem:dedekind}~(b) yields that
$T\cap S=Q^x\cap Q^y\leq C_Q(xy^{-1})^y=1$.
Part (d) of the same lemma shows that
$\erz{T,S}=\erz{Q^x,Q^y}=\erz{[xy^{-1},Q]^{Q^{xy^{-1}}}}Q^{xy^{-1}}=DQ^{xy^{-1}}=DQ=A$, as $xy^{-1}\in D$.

 For all $z\in\{x,y\}$ we calculate that
$\erz{Q^z,Q^h}=\erz{[zh^{-1},Q]^{Q^{zh^{-1}}}}Q^{zh^{-1}}=$
$\erz{([z,Q]^{h^{-1}}[h^{-1},Q])^{Q^{zh^{-1}}}}Q^{zh^{-1}}$
$=\erz{D,R_0}Q^{zh^{-1}}=F$ by Lemma \ref{lem:dedekind}~(d).  			
Thus L9 (iv) of Lemma \ref{lem:L9Char} is true.
\end{minipage}\hfill
\begin{minipage}{4cm}
				\begin{tikzpicture}
				\coordinate[label=left:{$1$}]					(E) at (0,0);
				\coordinate[label=right:{$DR_0Q$}]				(F) at (0,4.5);
				\coordinate[label=above left:{$DQ$}]	(A) at (-1.5,3);
				\coordinate[label=below right:{$DR_0$}]			(B) at (0,3);
				\coordinate[label=above:{$DQ^h$}]		(C) at (1.5,3);
				\coordinate[label=below:{$Q^x$}]				(S) at (-2,1.5);
				\coordinate[label=below:{$Q^y$}]			(T) at (-1,1.5);
				\coordinate[label=below right:{$P$}]	(D) at (0,1.5);
				\coordinate[label=below:{$Q^h$}]			(U) at (1.5,1.5);
				
				\foreach \x in {E,F,A,B,C,D,S,T,U} \fill (\x) circle (2.15pt);
				
				\draw [thick]	(E) -- (S) -- (A) -- (F) -- (C)
				(E) -- (T) -- (A) -- (D) -- (C) -- (U) -- (E)
				(E) -- (D) -- (B) -- (F);
				\end{tikzpicture}
\end{minipage}

\smallskip We moreover have that $\erz{A,B}=\erz{D,Q,R_0}=F=\erz{D,Q^h,R_0}$
and  $A\cap B=DQ\cap DR_0=D(Q\cap DR_0)=D=D(Q^h\cap DR_0)=C\cap B$.
Finally $A\cap C= DQ\cap DQ^h=D(Q\cap DQ^h)\leq DC_Q(h)=D$ by Lemma\ref{lem:dedekind}~(b).

Using  Lemma \ref{lem:L9Char} we conclude that $G/E$ is not $L_9$-free, and hence $G$ is not $L_9$-free.
		\end{proof}

\begin{cor}\label{kor:3primuntenil}
Suppose that $p,~q$ and $r$ are pairwise distinct primes and that $G$ is a  directly indecomposable $L_9$-free $\{p,q,r\}$-group.
Suppose further that $P \in \syl_p(G)$ and  $R \in \syl_r(G)$ are normal in $G$ and let $Q \in \syl_q(G)$.

Then $Q$ is cyclic and $C_Q(P)\neq C_Q(R)$.
\end{cor}

\begin{proof}
Since $G$ is directly indecomposable, we see that $Q$ acts non-trivially on both $P$ and $R$.
Moreover, $PQ$ and $RQ$ are $L_9$-free by hypothesis, and then we conclude that $Q$ is cyclic or isomorphic to $Q_8$.
\\In the first case, our assertion follows from Lemma \ref{lem:exam}, and in the second case, we choose a maximal subgroup $Q_1$ of $Q$. Then $Q_1$ acts irreducibly on $P$ and $R$ by Lemma \ref{lem:MOP4}, and the same Lemma shows that $\Phi(Q)$ inverts $P$ and $R$.
Thus $Q_1$ acts on $P$ and on $R$ avoiding $L_9$, respectively, and it acts faithfully. This contradicts Lemma \ref{lem:exam}.
\end{proof}

We explain another example where a subgroup lattice contains $L_9$.

\begin{lemma}\label{example}
Suppose that $p$, $q$ and $r$ are pairwise distinct primes and that $G$ is a $\{p,q,r\}$-group.
Suppose further that $P \in \syl_p(G)$ is normal in $G$ and that $Q \in \syl_q(G)$ and $R\in\syl_r(G)$ are cyclic groups such that $R\unlhd RQ$.
Suppose that $|R|=r$ and $C_Q(R)=1$.

If $R$ acts irreducibly on $P$, but non-trivially, and if $1\neq [P,Q]$ is elementary abelian, then $G$ is not $L_9$-free.
\end{lemma}

\begin{proof} We first remark that $G$ is soluble, because $P\unlhd PR\unlhd PRQ=G$.
We will construct the lattice $L_9$ in $L(G)$ using Lemma \ref{lem:L9Char}.
For this we set $E:=1$ and $D:=P$.
Then $D\neq E$ and we see that L9 (i) is true.

\smallskip Next, we recall that $1\neq [P,Q]$ by hypothesis.
Assume that $|[P,Q]|= 2$. Then $Q$, which normalises $[P,Q]$, must centralise it, and then Lemma \ref{coprime} gives a contradiction.

Therefore $|[P,Q]|\gneq 2$.
As a consequence, we find $a,b\in [P,Q]^\#$ such that $a\neq b$, and then we set
$S:=R^a$ and $T:=R^b$. Now $D\cap S=1=E=D\cap T$ and
$\erz{D,T}=PR^b=PR=PR^a=\erz{D,S}$. We set $A:=PR$.
In addition, since $R$ acts irreducibly, but non-trivially on $P$, it follows that $C_R(ba^{-1})\lneq R$.  Then the fact that  $|R|=r$ gives that $C_R(ba^{-1})=1=E$.

Lemma \ref{lem:dedekind}~(b) shows that $S\cap T=(R\cap R^{ba^{-1}})^a\leq C_{R}(ba^{-1})^a=E$, and now we recall that $[ba^{-1},R]\neq 1$.
Moreover, $R$ acts irreducibly on $P$, and then Part (e) of the same lemma yields the following:

$\erz{T,S}=\erz{R,R^{ba^{-1}}}^a=(\erz{[ba^{-1},R]^{R^{ba^{-1}}}}R^{ba^{-1}})^a=PR=A$.
We conclude that  L9 (ii) holds.

\smallskip For L9 (iii)  we set $U:=Q$ and $C:=\erz{D,Q}=PQ$. Then we note that $U\cap D=Q\cap P=1=E$.

\medskip\begin{minipage}{5.5cm}
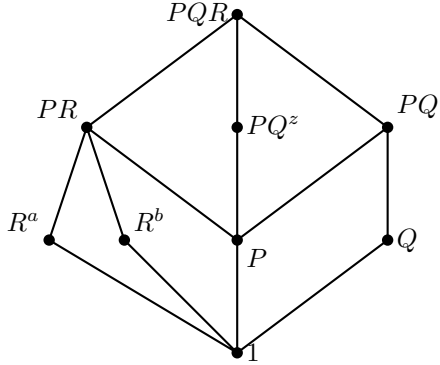
\begin{tikzpicture}
				\coordinate[label=right:{$1$}]					(E) at (0,0);
				\coordinate[label=left:{$PQR$}]				(F) at (0,4.5);
				\coordinate[label=above left:{$PR$}]	(A) at (-2,3);
				\coordinate[label=right:{$PQ^z$}]			(B) at (0,3);
				\coordinate[label=above right:{$PQ$}]		(C) at (2,3);
				\coordinate[label=above left:{$R^a$}]				(S) at (-2.5,1.5);
				\coordinate[label=above right:{$R^b$}]			(T) at (-1.5,1.5);
				\coordinate[label=below right:{$P$}]	(D) at (0,1.5);
				\coordinate[label=right:{$Q$}]			(U) at (2,1.5);
				
				\foreach \x in {E,F,A,B,C,D,S,T,U} \fill (\x) circle (2.15pt);
				
				\draw [thick]	(E) -- (S) -- (A) -- (F) -- (C)
				(E) -- (T) -- (A) -- (D) -- (C) -- (U) -- (E)
				(E) -- (D) -- (B) -- (F);
				\end{tikzpicture}
\end{minipage}
\hfill\begin{minipage}{9.5cm}
Assume for a contradiction that $X:=\erz{R^c,Q}$ has odd order for some $c\in\{a,b\}$.
%If $p=2$, then $G$ is soluble because $P \unlhd G$ and $G/P$ has odd order by assumption.
%If $p \neq 2$, then $G$ has odd order and is hence soluble.
%(Here we have used \cite{FT} twice.)
%Anmerkung Imke: Wir brauchen nicht das odd-order theorem. Es ist G/P isomorph zu RQ und da R ein Normalteiler von RQ ist, ist $P\unlhd PR\unlhd G$ eine Reihe mit aufloesbaren faktoren.
%Ich habe eine Bemerkung ganz am Anfang vom Beweis grmacht.

In both cases $X$ is a $p'$-Hall subgroup of the soluble group $G=PRQ$ and therefore $R^c=O_r(X)$.
It follows that $Q$ normalises $R^c$ and then that $Q^{c^{-1}}$ and $Q$ normalise $R$.

Since $N_P(R)$ is $R$-invariant and $R$ acts irreducibly, but non-trivially on $P$, we conclude that $N_G(R)=RQ$.
Thus Sylow's Theorem provides some $y\in R$ such that $Q^{c^{-1}}=Q^y$.
Now $[yc,Q]\leq PR\cap Q=1$. In addition $|R|=r$ and $[R,Q]\neq 1$ by hypothesis. Together this gives that $yc\in C_G(Q)\cap PR\leq PQ\cap PR=P(Q\cap PR)=P$, by Dedekind's modular law.
Altogether we have that
$yc \in C_P(Q)$. We recall that $c \in \{a,b\} \subseteq P$, and then $y=ycc^{-1} \in P$.
But we chose $y \in R$ and now $y \in R \cap P=1$, whence $Q^{c^{-1}}=Q$. In other words, $c\in N_P(Q)$, and this means that $[Q,c] \le Q \cap P=1$ and $c \in C_P(Q)$.
We recall that $c\in [P,Q]^\#$ and that $[P,Q]$ is elementary abelian by hypothesis. Then Lemma \ref{coprime} implies that
$[P,Q]=[P,Q,Q] \times C_{[P,Q]}(Q)$, and this contradicts the fact that
$c \in C_P(Q) \cap [P,Q]$.
\end{minipage}

\smallskip
It follows that $X$ has even order and since $R^c$ acts irreducibly on $P$, we conclude that $P\leq X$.
This implies that $\erz{S,U}=\erz{R^a,Q}=PRQ=G=\erz{R^b,Q}=\erz{T, Q}$ and then
 L9 (iv) holds for $F:=G$.

We finally set $B:=PQ^z$ for some $z\in R^\#$. Then $R=\erz{z}$ and Lemma \ref{lem:dedekind}~(b) and (c), together with our hypothesis, show that
$P\leq PQ\cap PQ^z=P(Q\cap PQ^z)\leq PC_Q(z)=PC_Q(R)=P$. Thus we have that $B\cap C=P=D$.
We further see that $A\cap C=PR\cap PQ=P(R\cap PQ)=P=D$ and $A\cap B=PR\cap PQ^z=P(R\cap PQ^z)=P=D$.
Since $\erz{A,B}=\erz{PR, PQ^z}=PQR=G$ and $\erz{B,C}=\erz{PQ, PQ^z}=P\erz{[Q,z]^{Q^z}}Q^z=P[Q,R]Q=PRQ=G$ by Lemma \ref{lem:dedekind}~(d), we finally obtain L9 (v).

Altogether Lemma \ref{lem:L9Char} gives the assertion.
\end{proof}

\begin{prop}\label{lem:3primPunique}
Suppose that $p,~q$ and $r$ are pairwise distinct primes and that $G$ is a non-nilpotent $L_9$-free $\{p,q,r\}$-group with normal Sylow $p$-subgroup $P$.
Suppose further that $R \in \syl_r(G)$ and $Q\in\syl_q(G)$ are not normal in $G$, that $R\unlhd RQ$ and $[R,Q]\neq 1$.

Then $RQ$ is a batten, $P$ is elementary abelian of order $p^r$, $R$ and $Q$ act irreducibly on $P$ and $\Phi(Q)$ induces non-trivial power automorphisms on $P$.
\end{prop}

\begin{proof}

We proceed in a series of steps.

\medskip(1) The groups $PQ$ and $PR$ are not nilpotent, $Q$ is cyclic and $R\cong Q_8$ or $|R|=r$. In addition $R=[R,Q]$ and $[R,C_Q(P)]\leq C_{R}(P)$.

\begin{subproof}
By hypothesis $R$ is not normal in $G$, but $Q$ normalises $R$. Hence $P \nleq N_G(R)$ and in particular
$PR$ is not nilpotent. But $PR$ is $L_9$-free, because $G$ is.  Moreover, $RQ$ is non-nilpotent and $L_9$-free, again by hypothesis.
Then Corollary \ref{kor:2prim} implies that $R$ and $Q$ are
battens, that $R$ acts on $P$ avoiding $L_9$ and that $Q$ acts on $R$ avoiding $L_9$.
More specifically, $R$ and $Q$ are cyclic or isomorphic to $Q_8$, and $[P,R]$ as well as $[R,Q]$ are elementary abelian or isomorphic to $Q_8$.

It follows that $R\cong Q_8$ or that $R$ is cyclic of order $r$.
In both cases Lemma \ref{coprime} yields $R=[R,Q]$ and the avoiding $L_9$ action of $Q$ on $R$ gives that $Q$ is cyclic.

In addition  $[P,C_Q(P),R]=1$, $[P,R,C_Q(P)]\leq [P,C_Q(P)]=1$ and then the Three Subgroups Lemma (see for example 1.5.6 of \cite{KurzStell}) implies that $1=[R,C_Q(P),P]$. Thus $[R,C_Q(P)]\leq C_{R}(P)$.

If it was true that $[P,Q]=1$, then  $R=[R,Q]=[R,C_Q(P)]$ would centralise $P$. But this is a contradiction.
\end{subproof}

\medskip(2) $C_R(P)=1$ and  $C_Q(P) \le Z(G)$.

\begin{subproof}
Since $[P,R]\neq 1$ by (1), we can apply Corollary \ref{kor:2prim} to $PR$, and this shows that $R$ acts on $P$ avoiding $L_9$.
Otherwise (1) implies that $R\cong Q_8$ and then  Definition \ref{defi:avoidQ8} gives that $R$ acts faithfully on $P$.
If $|R|=r$, then $R$ acts faithfully of $P$ because $[P,R]\neq 1$.
In both cases we see that $C_R(P)=1$, and then the last statement of (1) implies  that $[R,C_Q(P)] \le C_R(P)=1$.
Then $C_Q(P)$ centralises $P$ and $R$, and $Q$ is cyclic by (1), and therefore it follows that $C_Q(P) \le Z(G)$.
\end{subproof}

(3) $Z(G)=1$ or $p\neq 2$.

\begin{subproof}
We suppose that $p=2$.
Let $-:G\to G/Z(G)$ be the natural homomorphism.
We show that $\bar G$ satisfies the hypotheses of our lemma.

From (1) we see that none of the groups $P$, $Q$ or $R$ is contained in $Z(G)$. We even have that $R \cap Z(G)=1$ by (1).
In particular $p,q,r \in \pi(\bar G)$ and $\bar G$ is $L_9$-free.
We see that $\bar P$ is a normal Sylow $p$-subgroup of $G$ and
that $\bar Q \in \syl_q(G)$ and $\bar R\in\syl_r(G)$ are such that $\bar R\unlhd \bar R\bar Q\leq G$.
Let $X\in\{R,Q\}$. If $\bar X\unlhd \bar G$, then $XZ(G)\unlhd G$ and $X$ is a characteristic Sylow subgroup of $XZ(G)$ and hence normal in $G$. This is a contradiction.

We deduce that all hypotheses of the lemma hold for $\bar G$ and that $[\bar R,\bar Q]=\overline{[R,Q]}=\bar R \neq 1$.
Therefore, if $Z(G) \neq 1$, then the minimal choice of $G$ implies that  $\Phi(\bar Q)$ induces non-trivial power automorphisms on the elementary abelian group $\bar P$. Then Lemma \ref{power} yields that $p\neq 2$.
\end{subproof}

(4) $C_P(R)=1$,  and the groups $[P,Q]$ and $P=[P,R]$ are elementary abelian.
\begin{subproof}
As $PR$ and $PQ$ are not nilpotent by (1), Corollary \ref{kor:2prim} implies that $X$ acts on $P$ avoiding $L_9$  and that  $[P,X]$ is elementary abelian or isomorphic to $Q_8$  for  both $X\in\{Q,R\}$.

Assume for a contradiction that $[P,X]$ is isomorphic to $Q_8$ for some $X\in\{Q,R\}$.
Then $X$ acts on $P$ of type (hamil). Hence we obtain a  group $I$ of order $1$ or $2$ such that $P\cong Q_8\times I$.
It follows that $\mathrm{Aut}(P)$ is a $\{2,3\}$-group. But this is impossible because $Q$ and $R$ both act coprimely and non-trivially on $P$ by (1), and $p=2$, $q$ and $r$ are pairwise distinct.

We conclude that $[P,Q]$ and $[P,R]$ are elementary abelian, and then it follows that $P$ is abelian, by Lemma \ref{lem:p2Qirred}, applied to $PR$.

Assume for a  further   contradiction that $C_P(R)\neq 1$.
As $R$ avoids $L_9$ in its action on $P$, it follows from Lemma \ref{lem:MOP4} that $R$ is not isomorphic to $Q_8$. Now (1) yields that $R$ is cyclic and we may apply Lemma \ref{lem:p2Qirred} to $PR$, because $[P,R]$ is elementary abelian. It follows that  $C_P(R)$ is a cyclic $2$-group and thus $q$ and $r$ are odd.

Furthermore $C_P(R)$ is normalised by $Q$, because $Q$ normalises $P$ and $R$. Since $q$ is odd, it follows that $C_P(R)$ is centralised by $Q$.
We recall that $P$ is abelian, and then (3) yields that $C_P(R)\leq Z(G)=1$. This is a contradiction.

Altogether $C_P(R)=1$ and Lemma \ref{coprime} gives that $P=[P,R]$ is elementary abelian.
\end{subproof}

(5) $C_P(Q)=1$ and $Q$ acts on $P$ avoiding $L_9$ of type (std).

\begin{subproof}
Since $PQ$ is not nilpotent and $Q$ is cyclic by (1), Corollary \ref{kor:2prim} implies that the action of $Q$ on $P$ avoids $L_9$.  But $P$ is elementary abelian by (4), and therefore the action is not of type (hamil).

We assume for a contradiction that $C_P(Q)\neq 1$.
Then it follows that $PQ$ is not of type (std).
 We  consequently have type (cent) and we see that  $C_P(Q)$ is a cyclic $2$-group. In particular $p=2$, and thus $q$ and $r$ are odd. Then $|R|=r$ by (1).

Next we claim that $G$ satisfies the hypotheses of Lemma \ref{example}.

First, $q$ and $r$ are pairwise distinct odd primes and $G$ is a finite $\{2,q,r\}$-group. From above, (1) and our assumption we see that $P \in \syl_2(G)$ is normal in $G$ and that $Q \in \syl_q(G)$ and $R\in\syl_r(G)$ are cyclic groups such that $R\unlhd RQ$.
We have shown that $R$ has order $r$.

We recall that $C_P(Q)$ is a cyclic $2$-group (first paragraph). As $r$ is odd  and $C_P(R)=1$ by (4), we see that $C_P(Q)$ is not $R$-invariant.
But $C_P(C_Q(R))$ is $R$-invariant, and now the irreducible action of $R$ on $P$ and the fact that $1\neq C_P(Q)\leq  C_P(C_Q(R))$ show that $P= C_P(C_Q(R))$. Then (2) and (4) imply that $C_Q(R)\leq C_Q(P)\leq Z(G)=1$.

By (4) $P$ is an elementary abelian $2$-group, and $P=[P,R] \neq 1$ by (1) and (4).
In particular $R$ does not induce power automorphism on $P$ by Lemma \ref{power}.
As $C_P(R)=1$ by (4), we deduce that $R$ acts irreducibly on $P$ (using Lemma \ref{lem:p2Qirred}). Finally, (1) and (4) yield that
$1\neq [P,Q]$ is elementary abelian.

All hypotheses of Lemma \ref{example} are satisfied now, and we infer that $G$ is not $L_9$-free. This is a contradiction.

Thus $C_P(Q)=1$ and we deduce that the action of $Q$ on $P$ is not of type (cen). It remains that $Q$ acts on $P$ avoiding $L_9$ of type (std).
\end{subproof}

(6) $R$ and $Q$ act irreducibly on $P$. If $X\leq Q$ induces power automorphisms on $P$, then $X$ centralises $R$.

\begin{subproof}
We recall from (2) that $C_Q(P) \le Z(G)$ and $C_R(P)=1$. Consequently
$C_{RQ}(P)=C_Q(P) \le Z(RQ)$ and
$RQ/C_Q(P)$ is isomorphic to a subgroup of $\mathrm{Out}(P)$.
Since $RQ$ is not nilpotent, we see that $RQ/C_Q(P)$ is not nilpotent.

The group $P$ is an elementary abelian $p$-group by (4) and hence, if $X\leq RQ$ induces power automorphisms on it, then
it follows that
$XC_Q(P)/C_Q(P)\leq Z(RQ/C_Q(P))$, see page 177 of \cite{Huppert}. We denote this fact by ($\ast$).
Then we deduce that $[R,X]\le C_Q(P) \le Z(RQ)$ (see above) and therefore
$[R,X]=[X,R]=[X,R,R]=1$ by Lemma \ref{coprime}.

In addition, the fact ($\ast$) shows that neither $R$ nor $Q$ induces power automorphisms on $P$.
It follows from (5) and Definition \ref{defi:avoidcyclic}~(std) that $Q$ acts irreducibly on $P$.
Furthermore (1), together with Corollary  \ref{kor:2prim}, yields that $R$ acts on $P$ avoiding $L_9$. Since $C_P(R)=1$ by (4), this action has type (std) or (hamil). In the first case, the irreducible action follows from Definition  \ref{defi:avoidcyclic}~(std), and in the second case, it follows from Lemma \ref{lem:MOP4}.
\end{subproof}

(7) $C_Q(R)$ induces non-trivial power automorphisms on $P$.

\begin{subproof}
We set $Q_0:=C_Q(R)$ and we assume that $Q_0$ acts irreducibly on $P$.
Then II 3.11 of \cite{Huppert} implies that $RQ= C_{RQ}(Q_0)$ is isomorphic to a subgroup of the multiplicative group of some field of order $|P|$, and it follows that $RQ$ is cyclic. This is a contradiction.

Thus (6) and Lemma \ref{lem:lactirreorpower}  show that $Q_0$ induces power automorphisms on $P$.
Since $R$ normalises every $Q$-invariant subgroup of $P$ by (7), we see from Lemma \ref{lem:exam} and (2) that $C_Q(R) \gneq C_Q(P)$.
\end{subproof}

(8) $|P|=p^q$ and $\Phi(Q)=C_Q(R)$.
\begin{subproof} We recall that $Q$ is cyclic, by (1).
Moreover $[R,Q]\neq 1$, whence
$C_Q(R)<Q$ and therefore we
may choose $y \in Q$  such that $ C_Q(R)<\erz{y}$.

By (5) and Lemma \ref{lem:lactirreorpower}, it follows that every subgroup of $\erz{y}$ either acts irreducibly on $P$ or induces power automorphism on it (in particular normalising every
subgroup of $P$).
Then $P\erz{y}$ satisfies  (b) of Lemma 3.1 in \cite{Lfree}, which implies that it
satisfies one of the possibilities 3.1~(i)-3.1~(iii).
By (5) and the choice of $y$, we see that 3.1~(i) is not true.
Further (7) provides some $x\in C_Q(R)$ that induces a power automorphism of order $q$ on $P$.
This implies that $q$ divides $p-1$ and therefore
$P\erz{y}$ satisfies  (ii) of Lemma 3.1~(b) in \cite{Lfree}.
It follows that $|P|=p^q$ and that, if $k$ is the largest positive integer such that $q^k$ divides $p-1$, then
 $y$ induces an automorphism of order $q^{k+1}$ on $P$.
We conclude that $q^{k+1}=|\erz{y}:C_\erz{y}(P)|=|\erz{y}:C_R(P)|$, because $Q$ is cyclic.
Finally, we deduce that $o(y)$ is uniquely determined, that $Q=\erz{y}$ and that$C_Q(P)=\Phi(Q)$.
\end{subproof}

 By (6), we see that $R$ and $Q$ act irreducibly on $P$, and (1) gives that $Q$ is cyclic. Then (8)  and (7) say that $\Phi(Q)$ induces non-trivial power automorphisms on $P$.
 In addition  $P$ is elementary abelian by (4), and it has order $p^r$ by (8).
If $|R|=r$, then $R$ is a cyclic group of order $r$ and $RQ$ is a batten. In particular $G$ satisfies the assertion of our lemma.

But $G$ is a counterexample, and then it follows that $R\cong Q_8$ and $r=2$.
Then (4) yields that $PR$ fulfills the hypothesis of Lemma \ref{lem:p2QQ8}, and consequently $p^r=p^2=|P|=p^q$ by (8). This is our final contradiction, because $q\neq r$.
\end{proof}

\begin{definition}\label{defi:avoidbatten}
Suppose that $B$ is a non-nilpotent batten  that acts coprimely on the $p$-group $P$.
We say that the action of $B$ on $P$ \textbf{avoids $L_9$}
if and only if $[P,Z(B)]\neq 1$ and if one of the following occurs:

\begin{enumerate}
\item[(Cy)] $[P,\mathcal B(B)]=1$ and  $Q$ acts on $P$ avoiding $L_9$ for every Sylow subgroup $Q$ of $B$ different from  $\mathcal B(B)$ or
\item [(NN)] $P$ is elementary abelian of order $p^{|B:\mathcal B(B)Z(B)|}$ and the Sylow subgroups of $B$ act irreducibly on $P$, while $Z(B)$ induces power automorphisms on $P$.
\end{enumerate}
As in definition \ref{defi:avoidcyclic}, we specify the type of the $L_9$-avoiding action by writing
that \textbf{"$B$ acts on $P$ avoiding $L_9$ of type ($\cdot$)"}.
\end{definition}

\begin{lemma}\label{lem:avoid1prop}
Let $B$ be a batten that acts non-trivially and avoiding $L_9$ on the $p$-group $P$.
Then the following hold:
\begin{enumerate}
\item If $C_P(B)\neq 1$, then $p=2$.
\item Either $P=[P,B]\times C_P(B)$, where $[P,B]$ is elementary abelian and $C_P(B)$ is cyclic, or  $P=[P,B]\times I$, where $I$ is a group of order at most $2$ and $[P,B]\cong Q_8$.
\item $C_P(B)$ is centralised by every automorphism of $P$ of order coprime to $p$ that leaves $C_P(B)$ invariant.
%\item If $C_P(B)=1$, then $P=[P,B]$ is elementary abelian.
%\item $\Phi(P)$ is a cyclic $2$-group.
\end{enumerate}
\end{lemma}	
	
\begin{proof}
If $B \cong Q_8$, then Definition \ref{defi:avoidQ8} and Lemma \ref{lem:MOP4} imply that
$P$ is elementary abelian and that
$B$ acts irreducibly on it. We conclude that $P=[P,B]$ is elementary abelian and we deduce from Lemma \ref{coprime} that $C_P(B)=1$. Hence, in this case, all statements of our lemma hold.

Now suppose that $B$ is not nilpotent and that it acts of type (NN). Then Definition \ref{defi:avoidbatten} states that, once more, $P$ is elementary abelian and
$B$ acts irreducibly on it. Again we see that $P=[P,B]$ is elementary abelian, and as before all statements hold.

Next we suppose $B$ is not nilpotent and that it acts of type (Cy), or that $B$ is cyclic.
In the first case $B$ has a cyclic Sylow subgroup $Q$ that acts on $P$ avoiding $L_9$ such that $C_P(Q)=C_P(B)$ and $[P,B]=[P,Q]$ by  Definition \ref{defi:avoidbatten}.
In the second case we set $Q:=B$.
\\Then, in both cases, $Q$ is a cyclic group that acts on $P$ avoiding $L_9$  such that $C_P(Q)=C_P(B)$ and $[P,B]=[P,Q]$.
If $Q$ acts of type (std), then $P=[P,Q]$ is elementary abelian by Definition \ref{defi:avoidcyclic}. Again we deduce the statements of our lemma.

Suppose that $Q$ acts of type (cent). Then Definition \ref{defi:avoidcyclic} yields that $[P,Q]=[P,B]$ is elementary abelian, that $P$ is abelian and that $C_P(Q)=C_P(B)$ is a cyclic $2$-group.
In particular $P=[P,B]\times C_P(B)$ by Lemma \ref{coprime}.
It also follows that $C_P(B)$ is centralised by every automorphisms of $P$ of odd order that leaves $C_P(B)$ invariant. These are the statements of our lemma.

Finally, suppose that $Q$ acts of type (hamil).
Then Definition \ref{defi:avoidcyclic} yields that $[P,Q]\cong Q_8$ and $P=[P,Q]\times I$, where $I$ is a group of order at most $2$.  In particular statement (a) is true.
Moreover, we deduce that $C_P(Q)\leq \Omega_1(P)= \Phi([P,Q])\times I$, where $\Phi([P,Q])$ is cyclic of order $2$. In particular $\Phi([P,Q])$ is centralised by $B$ and by every automorphisms of $P$.
We conclude that $C_P(Q)$ is elementary abelian of order at most $4$ and that every automorphism of $P$ centralises a cyclic subgroup of order $2$.
This implies (b).
\end{proof}

\section{Avoiding $L_9$}

We now work towards a classification of arbitrary $L_9$-free groups, and therefore we need to understand in more detail the group structures that appear when "$L_9$ is avoided" in the sense of the previous section.

\begin{definition}\label{defi:avoidbattengroup}
Suppose that $K$ is a batten group that acts coprimely on the $p$-group $P$.
We say that the action of $K$ on $P$ \textbf{avoids $L_9$}
if and only if $[P,K]\neq 1$ and every batten of $K$ either centralises $P$ or avoids $L_9$ in its action on $P$.
\end{definition}

\begin{lemma}\label{avoidbattensubgroup}Let $K$ be a batten group that acts coprimely on the $p$-group $P$ avoiding $L_9$. Suppose further that $L\leq K$ and $L_0\unlhd L$ such that $[P,L]\neq 1=[P,L_0]$.
Then $L/L_0$ acts on $P$ avoiding $L_9$. In particular $L/L_0$ is a batten group.
\end{lemma}
\begin{proof}By induction we may suppose that $K$ is a batten and that either $L_0=1$ and $L$ is a maximal subgroup of $K$ or that $L_0$ is a minimal normal subgroup of $K=L$.
Thus either $|L_0|$ has order $q$ or $|K:L|=q$. Since $L_0\leq C_K(P)$, we first remark that $L/L_0$ induces automorphisms on $P$.

If $K\cong Q_8$, then $K$ acts faithfully on $P$ by Definition \ref{defi:avoidQ8}. Thus $L_0\leq C_K(P)=1$ and it follows that $L$ is a cyclic group of order $4$.
Thus Lemma \ref{lem:MOP4} yields  that $\Omega_1(L)$ inverts $P$ and that $L$ acts irreducibly on the elementary abelian group $P=[P,L]$. Then we see that $L\cong L/L_0$ acts on $P$ avoiding $L_9$ of type (std).

Next suppose that $K$ is cyclic. Then $L/L_0$ is cyclic.
If $K$ acts of type (std) on $P$, then $L$ and every subgroup of $L$ act irreducibly or via inducing power automorphisms on the elementary abelian group $P=[P,K]$. Since $[P,L]\neq 1$ and power automorphisms are universal, by Lemma 1.5.4 of \cite{Schmidt}, it follows that $P=[P,L]$.
Moreover, the action of $L$ on $P$ is equivalent to that of $L/L_0$, and then it follows that $L/L_0$ acts on $P$ avoiding $L_9$ of type (std).
\\If $K$ acts  on $P$ of type (cent), then $L$ and all its subgroups act irreducibly or trivially on the elementary abelian group $[P,K]$. Again the fact that $[P,L]\neq 1$ implies that $P=[P,L]$, and then $C_P(L)=C_P(K)$ by Lemma \ref{coprime}.
Since the action of $L$ on $P$ is equivalent to that of $L/L_0$, it follows that $L/L_0$ acts on $P$ avoiding $L_9$ of type (cent).
\\We suppose now that $K$ acts of type (hamil). Then $K$ is a cyclic $3$-group and $K/C_K(P)$ has order $3$.
It follows that $L=K$. But again,  the action of $L/L_0=K/L_0$ on $P$ is equivalent to the action of $K$ on $P$, which means that it has type (hamil).

We finally suppose that $K$ is a non-nilpotent batten. Let $R$ be a Sylow subgroup of $K$ such that $K=\mathcal B(K)\cdot R$.
Suppose first that $L/L_0$ is a $q$-group.
Then our choice of $L$ and $L_0$ implies that $L/L_0\cong R$.
If $K$ acts on $P$ of type (Cy) in this case, then it follows that $L/L_0\cong R$ is cyclic and that it acts on $P$ avoiding $L_9$, according to Definition \ref{defi:avoidbatten}. Otherwise, if $K$ acts of type (NN), then $\mathcal B(K)\nleq C_K(P)$ and then $L_0=1$. It follows that $L=R$ acts irreducibly on the elementary abelian group $P$, whence $P=[P,L]=[P,L/L_0]$. In addition $\Phi(L)$ and all of its subgroups induce power automorphism on $P$.
Altogether the cyclic group $L/L_0\cong L$ acts on $P$ of type (std).
\\Now we suppose that $L/L_0$ does not have prime power order.
Then $L_0\leq \Phi(R)=Z(K)$. Now if  $L/L_0$ is nilpotent, then $L\neq K$ and therefore $L_0=1$.
It follows from Lemma \ref{lem:batten} that $L=Z(K)\times \mathcal B(K)$.
We have already proven that $R$ acts on $P$ avoiding $L_9$, and then $Z(K)$ also acts on $P$ avoiding $L_9$.
In addition $\mathcal B(K)$ either centralises $P$ or it acts irreducibly on the elementary abelian group $P=[P,\mathcal B(K)]$. Since $\mathcal B(K)$ has prime order, it follows that the cyclic group $\mathcal B(K)$ acts on $P$ avoiding $L_9$ of type (std). Altogether $L/L_0\cong L=Z(K)\times \mathcal B(K)$ acts on $P$ avoiding $L_9$.
\\Finally, suppose that $L/L_0$ is not nilpotent. Then $L$ is not nilpotent and hence Lemma \ref{lem:battensub} implies that $L=K$. We conclude that $L_0\neq 1$. Since $[P,Z(K)]\neq 1$ by Definition \ref{defi:avoidbatten}, it follows that $L_0$ is a proper subgroup of $Z(K)$ and that $L/L_0=K/L_0\cong \mathcal B(K) \rtimes R/L_0$ is a non-nilpotent batten.
If $[P, \mathcal B(K)]=1$, then our investigation above imply that the cyclic group $R/L_0$ acts on $P$ avoiding $L_9$, and then $K/L_0$ acts on $P$ avoiding $L_9$. Otherwise $1\neq [P,\mathcal B(K)]$ is elementary abelian of order $p^{|K:\mathcal B(K)Z(K)|}=p^{|K/L_0:\mathcal B(K)Z(K)/L_0|}$, moreover  $\mathcal B(K)\cong \mathcal B(K)L_)/L_0$ and $R/L_0$ act irreducibly on $P$. At the same time  $Z(K/L_0)=Z(K)/L_0$ induces power automorphisms on $P$.
Altogether $K/L_0$ acts on $P$ avoiding $L_9$ of type (NN).
\end{proof}

\begin{lemma}\label{lem:Schmidt}
Let $K$ be a batten group that acts coprimely on the $p$-group $P$ avoiding $L_9$.
Then the following assertions are true:
\begin{enumerate}
\item If $L \unlhd K$, then $[P,K]=[P,L]$ or $[P,L]=1$.
\item $[P,K]$ is elementary abelian or isomorphic to $Q_8$.
%\item If $[P,K]$ is elementary abelian, then $[P,K]\times C_P(K)$.
\end{enumerate}
\end{lemma}

\begin{proof} Let $L\unlhd K$ be such that $[P,L]\neq 1$.
Then $L$ is a batten group by Lemma \ref{lem:battensub} and therefore there is a batten $B$ of $L$ such that $[P,B]\neq 1$.
Assume for a contradiction that $[P,B]\leq [P,L]\lneq [P,K]\leq P$.
Then the fact that $P\neq [P,B]$ implies that $C_P(B)\neq 1$ by Lemma \ref{coprime}.
In addition $B$  avoids $L_9$ in its action  on $P$, by Lemma \ref{avoidbattensubgroup}.
Since $B$ is a batten of $L$, it is characteristic in $L$, and therefore $B \unlhd K$.
In particular $C_P(B)$ is $K$-invariant and hence it is centralised by $K$ by Lemma \ref{lem:avoid1prop}~(c).
This implies that $C_P(B)=C_P(K)$.
Finally $[P,K]=[[P,B]C_{P}(B),K]=[[P,B]C_P(K),K]=[[P,B],K]\leq [P,B]\lneq [P,K]$, which is  a contradiction.
In particular (a) is true.

Together with Lemma \ref{lem:avoid1prop}~(b), the statement in (b) follows from (a).
\end{proof}

\begin{lemma}\label{lem:laction}
Let $K$ be a batten group that acts on the $p$-group $P$ avoiding $L_9$, and suppose that $H$ is a subgroup of $K$. Then $H$ centralises $P$ or $[P,H]=[P,K]$.
\\Moreover, $H$ induces power automorphisms on $P$ or it acts irreducibly on $[P,K]/\Phi([P,K])$.
\end{lemma}	

\begin{proof}Let $H\leq K$. If $H$ centralises $P$, then it induces power automorphisms on $P$.
We may suppose that $[P,H]\neq 1$. Then $H$ has a $q$-subgroup $Q$ such that $[P,Q]\neq 1$.
Therefore, if $Q\unlhd K$, then we have that $[P,K]=[P,Q]$  by Lemma  \ref{lem:Schmidt}~(a).
Then the fact that $[P,Q]\leq [P,H]\leq [P,K]$ yields that $[P,H]=[P,K]$.

Assume for a contradiction that $[P,K]\neq [P,H]$.
Then Lemma \ref{lem:battenheart} implies that $K$ has a non-nilpotent batten $B$ such that $B=\mathcal B(B) Q$.
We moreover deduce that $[P,Q]\lneq [P,K]$ and $[P,B]=[P,K]$ by Lemma \ref{lem:Schmidt}~(a), because $B\unlhd K$.
Since the action of $K$ on $P$ avoids $L_9$, the action of $B$ also does.
If $B$ acts of type (Cy), then we obtain the contradiction that $[P,Q]=[P,B]$.
Thus $B$ acts of type (NN) and in particular $Q$ acts irreducibly on $P$. But this is impossible as well.
It follows that $[P,H]=[P,K]$.

Assume for a further  contradiction  that $H\leq K$ neither induces power automorphisms on $P$ nor does it act irreducibly on $[P,K]/\Phi([P,K])=[P,H]/\Phi([P,H])$.
Then there is a batten $B$ of $H$ that neither induces power automorphisms on $P$ nor does it act irreducibly on $[P,H]/\Phi([P,H])$. Similarly to the arguments above, we deduce that $[P,K]=[P,H]=[P,B]$, and  Lemma \ref{avoidbattensubgroup} gives that
$B$ avoids $L_9$ in its action on $P$.
Therefore Lemma \ref{lem:lactirreorpower} yields that $B$ is not nilpotent.
From Definition \ref{defi:avoidbatten} we further see that $B$ does not act of type (NN), and thus $B$ acts of type (Cy) on $P$.
Consequently $[P,\mathcal B(B)]=1$  and $B$ has a cyclic  Sylow subgroup $Q$ such that $B=\mathcal B(B) Q$ and $Q$ acts on $P$ avoiding $L_9$.
Again we have $[P,K]=[P,B]=[P,Q]$ and Lemma \ref{lem:lactirreorpower} gives that $Q$ induces power automorphisms on $P$ or acts irreducibly on  $[P,Q]/\Phi([P,Q])=[P,K]/\Phi([P,K])$.
In the first case $B=C_B(P)Q$ induces power automorphism on $P$ and in the second case $B$ acts irreducibly on  $[P,K]/\Phi([P,K])$. This is a contradiction.
  \end{proof}

\begin{lemma}\label{lem:avoidL9Cen}
Let $K$ be a batten group that acts on the $p$-group $P$ avoiding $L_9$, and suppose that $H$ is a subgroup of $K$ that acts non-trivially on $R\leq P$.

Then $C_H(R)=C_H(P)$, $C_P(H)=C_P(K)$ and $[P,H]=[P,K]$.
\end{lemma}	

\begin{proof}From Lemma \ref{lem:laction} we see that $[P,H]=[P,K]$.
In addition $C_P(K)\leq C_P(H)\leq C_P(Q)$ for every $q$-subgroup $Q$ of $H$ and every prime $q$.
Let $Q$ be a $q$-subgroup of $H$ for some prime $q$ such that $C_P(Q)\neq P$.
Then $Q$ is a batten by Lemma \ref{lem:battensub}, and it acts on $P$ avoiding $L_9$ by Lemma \ref{avoidbattensubgroup}.
If $Q \unlhd K$, then $K$ centralises $C_P(Q)$ by Lemma \ref{lem:avoid1prop}~(c). Thus $C_P(K)\leq C_P(H)\leq C_P(Q)\leq C_P(K)$, and this gives that $C_P(K)=C_P(H)$.
If $Q$ is not a normal subgroup of $K$, then Lemma \ref{lem:battenheart} provides a non-nilpotent batten $B$ of $K$ such that $B=\mathcal B(B) Q$.
Since the action of $K$ on $P$ avoids $L_9$, the action of $B$ also does.
If  $B$ acts of type (NN), then  $Q$ acts irreducibly on $P$ and therefore $C_P(Q)=1\leq C_P(K)$. Again we deduce that $C_P(K)=C_P(H)$.
If $B$ acts of type (Cy), then $[P,\mathcal B(B)]=1$ and hence $C_P(B)=C_P(Q)$.
But now $B$ is a normal subgroup of $K$, and then Lemma \ref{lem:avoid1prop}~(c) gives that  $C_P(B)=C_P(Q)\leq C_P(K)$.
As above we deduce that $C_P(K)=C_P(H)$.

Finally, suppose that $R\leq P$ is $H$-invariant, but not centralised by $H$, and set $ H_0:=C_H(R)\geq C_H(P)$.
Assume for a contradiction that $H_0$ does not centralise $P$.
Then we deduce, as above, that $C_P(K)=C_P(H_0)\geq R$. This is a contradiction, because $H$ does not centralise $R$.
\end{proof}

\begin{cor}\label{cor:p=2nK}
Let $K$ be a batten group that acts non-trivially and avoiding $L_9$ on the $p$-group $P$.
Then the following hold:
\begin{enumerate}
\item If $C_P(K)\neq 1$, then $p=2$.
\item If $C_P(K)=1$, then $P=[P,K]$ is elementary abelian.
\item If $K$ induces power automorphisms on $P$, then $P=[P,K]$ is elementary abelian of odd order. In particular $C_P(K)=1$ in this case.
\end{enumerate}
\end{cor}	
	
\begin{proof} Let $B$ be a batten of $K$ that does not centralise $P$. Then Lemma \ref{lem:avoidL9Cen} implies that  $C_P(K)= C_P(B)$.
Thus Part (a) and (b) of Lemma \ref{lem:avoid1prop} yield the statements (a) and (b) of our lemma.
For Part (c) we suppose that $K$ induces power automorphisms on $P$. Then $P$ is not an elementary abelian $2$-group by Lemma \ref{power}.
If $C_P(K)=1$, then our assertion holds by (b). Otherwise $p=2$ by (a), and then Lemma \ref{power} implies that $[P,K]$ is neither elementary abelian nor isomorphic to $Q_8$, contradicting Part (b) of Lemma \ref{lem:Schmidt}.

For the final comment we just use that $p$ is odd and then apply (a).
\end{proof}

\begin{lemma}\label{lem:avoidsubP}Let $B$ be a batten that acts on the $p$-group $P$ avoiding $L_9$.
Let $R\leq P$ be $B$-invariant and $R_0\leq C_P(B)$.

Then  $B$ avoids $L_9$ in its action on $R/R_0$.
\end{lemma}

\begin{proof}Since $B$ centralises $R_0$, the action of $B$ on $R/R_0$ is well-defined.

We first suppose that $B\cong Q_8$.
Then Lemma \ref{lem:MOP4} yields that $B$ acts irreducibly on $P$, and then it follows that $R=P$ and $R_0=1$.
Thus our assertion is true in this case.

Next suppose that $B$ is cyclic.
If $B$ acts of type (std) on $P$, then $R_0\leq C_P(B)=1$. If $B$ acts irreducibly on $P$, then again $P=R$ and there is nothing left to prove. Otherwise $B$ and all of its subgroups induce power automorphisms on $P$, and hence on $R=[R,B]$ as well. It follows that $B$ also acts of type (std) on $R\cong R/R_0$.

Suppose now that $B$ acts of type (cent).
Then, since $B$ does not centralise $R$ and $B$ acts irreducibly on $[P,B]$, it follows that $[P,B]\leq R$.
Moreover $P$ is abelian and then we use the fact that $R_0\leq C_P(B)$. This gives that $R/R_0=[R/R_0,B]\times C_{R/R_0}(B)\cong [R,B]\times C_R(B)/R_0$, where $C_R(B)/R_0$ is a cyclic $2$-group.
Since every subgroup of $B$ that does not centralise $P$ acts irreducibly on $[P,B]\cong [R/R_0,B]$ in this case, there are two possibilities for the action of $B$ on $R\cong R/R_0$: If $R_0\neq C_R(B)$, then $B$ acts of type (cent), and otherwise it acts of type (std).

Suppose now that $B$ acts of type (hamil).
Then the cyclic $3$-group $B$ acts irreducibly on $[P,B]/\Phi([P,B])$,
by Lemma \ref{lem:laction}, and we see again that $[P,B]\leq R$.
It follows that $R\cong Q_8\times I$, where $I$ is a group of order at most $2$, and then $R_0\leq C_R(B)\leq \Phi([R,B])\times I$.
We remark that $[R/R_0,B]B/Z([R/R_0,B]B)\cong [R,B]B/Z([R,B]B)\cong [P,B]B/Z([P,B]B)\cong \mathrm{Alt}_4$.
\\If $R_0\cap [P,B]=1$, then $R/R_0\cong Q_8\times \tilde J$ for some group $J$ of order $\frac{|I|}{|R_0|}$. Thus $B$ acts on $R/R_0$ of type (hamil) in this case.
\\Otherwise we have that $R_0\geq \Phi([P,B])$ and therefore $R/R_0$ is elementary abelian of order $4$ or $8$. Moreover $B$ acts irreducibly on $[R/R_0, B]$, which is a group of order $4$.
In addition every proper subgroup of $B$ centralises $R/R_0$.
Consequently, if $R/R_0=[R/R_0,B]$, then $B$ acts on $R/R_0$ of type (std) or of type (cent).

The final case is that $B$ is not nilpotent, and we suppose that $B$ acts of type (NN) on $P$.
Then Definition \ref{defi:avoidbatten} yields that $B$ acts irreducibly on $P$.
Hence there is nothing left to prove.

Suppose that $B$ acts of type (Cy). Then we choose a Sylow subgroup $Q$ of $B$ such that $B=\mathcal B(B)Q$.
 Then $[R/R_0,\mathcal B(B)]=[R,\mathcal B(B)]=[P,\mathcal B(B)]=1$ and $Q$ acts on $P$ avoiding $L_9$ in such a way that $\Phi(Q)=Z(B)$ does not centralise $P$. Then Lemma \ref{lem:avoidL9Cen} yields that $\Phi(Q)$ does not centralise $R$. In particular, we have that $[R/R_0,Z(B)]\neq 1$. In addition $Q$ acts on $R/R_0$ avoiding $L_9$, by our arguments above.
Altogether  $B$ acts on $R/R_0$ avoiding $L_9$ of type (Cy) in this final case.
\end{proof}

 \section{The first implication}

We now investigate the general case.

\begin{prop}\label{prop:Hinein}
Let $G$ be a finite $L_9$-free group. Then $G=NK$, where $N$ is a nilpotent normal Hall-subgroup of $G$ with modular Sylow subgroups and $K$ is a batten group. Moreover, for all $p \in \pi(N)$, every batten of $K$ acts on $O_p(N)$ avoiding $L_9$ or it centralises $O_p(N)$.
\end{prop}

\begin{proof}We first remark that $G$ is $L_{10}$-free, whence Theorem $A$ of \cite{L10} implies that $G$ is soluble.
Furthermore, Corollary C of \cite{L10} provides normal Hall-subgroups $N$ and $M$ of $G$ such that $N\leq M$ and such that $N$ is nilpotent, $M/N$ is a $2$-group and $G/M$ is metacyclic.
We choose $N$ as large as possible with these constraints.
From Lemma \ref{lem:mind1NT} we see that $N\neq 1$.
We also have that every Sylow subgroup of $N$ is $L_9$-free, and hence it is modular by Lemma \ref{lem:pGrModularÄqu}.
In addition the Schur-Zassenhaus Theorem (see for example 3.3.1. of \cite{KurzStell}) provides a complement $K$ of $N$ in $G$.

\medskip (1) If $RQ$ is a non-nilpotent Hall $\{r,q\}$-subgroup of $K$, where $R$ is a normal Sylow $r$-subgroup of $RQ$ and $Q\in \mathrm{Syl}_q(RQ)$, then $RQ$ is a batten. For all $p\in \pi(N)$, the group $RQ$ centralises $O_p(N)$ or acts on it avoiding $L_9$.

\begin{subproof}
If there is some $p\in \pi(N)$ such that $[O_p(N),R]\neq 1$, then we set $P:=O_p(N)$.
Otherwise the maximal choice of $N$ implies that $R$ is not a normal subgroup of $K$.
Then, using the solubility of $G$, we find a prime $s\in\pi(K)\setminus \{r\}$ and a normal $s$-subgroup $T$ of $K$ such that $[T,R]\neq 1$  (see 5.2.2 of \cite{KurzStell}). Then $s\neq q$ because $1\neq [R,Q]\leq R$ and $1\neq [T,R]\leq T$. In this case we set $P:=T$.

In both cases $p,~r$ and $q$ are pairwise different primes and $PRQ$ is a non-nilpotent $\{p,r,q\}$-subgroup that satisfies the hypothesis of Proposition \ref{lem:3primPunique}. For this we note that $P\unlhd PRQ$, $P\nleq N_G(R)$ and $R\nleq N_G(Q)$. Since $[R,Q]\neq 1$, the assertion in (1) follows.
\end{subproof}

(2) Every Sylow subgroup $S$ of $K$ is a batten, and for all $p\in \pi(N)$ it is true that $S$ centralises $O_p(N)$ or acts on $O_p(N)$ avoiding $L_9$.

\begin{subproof}
Let $S$ be a Sylow subgroup of $K$. If $S$ centralises $N$, then the choice of $N$ provides some Sylow subgroup $R$ of $K$ such that $RS$ is not nilpotent. Then (1) implies that $RS$ is a batten, then that $S$ is cyclic and hence that $S$ is a batten.
\\Let $p\in \pi(N/C_N(S))$. Then  $O_p(N)S$ is an $L_9$-free $\{p,q\}$-group for some prime $q$.
Hence Corollary \ref{kor:2prim} implies the assertion.
\end{subproof}

(3) $K$ is a batten group.

\begin{subproof}
Let $1\neq B\leq K$ be such that there is some $K_1\leq K$ such that $K=K_1\times B$, where $(|K_1|,|B|)=1$ and $B$ is not a direct product of non-trivial subgroups of coprime order.
In particular $B$ is a Hall subgroup of $K$.
If $B$ is nilpotent, then $B$ is a Sylow $q$-subgroup of $K$ for some prime $q$.
In this case (2) implies that $B$ is a batten.

Assume for a contradiction that $B$ is not a batten of $K$.
Then $B$ is not nilpotent and therefore (1) yields that $|B|$ is divisible by at least three different primes.
Since $B\leq G$ is $L_9$-free, Lemma \ref{lem:mind1NT} provides a normal Sylow $r$-subgroup $R$ of $B$ for some prime $r\in \pi(B)$.
We remark that $R$ is a normal subgroup of $K=K_1\times B$.
In addition $B$ is  is not a direct product of non-trivial subgroups with coprime order and hence there are a prime $q\in \pi(B)$ and some $Q\in\mathrm{Syl}_q(B)$ such that $RQ$ is not nilpotent.
Now $B$ is a Hall subgroup of $K$ and thus $RQ$ is a Hall subgroup of $K$.
In particular (1)  implies that $RQ$ is a batten and it follows that  $|R|=r$ and $1\neq C_Q(R)=\Phi(Q)$. We further see, from Definition \ref{defi:avoidbatten} and (1), that for every $p\in \pi(N)$ with the property $[O_p(N),R]\neq 1$ we have that $|O_p(N)|=p^q$.
Since $R$ is a normal Sylow subgroup of $K$, the maximal choice of $N$ provides some $p\in\pi(N)$ such that $R$ does not centralise $P:=O_p(N)$. In particular we have that $|P|=p^q$.

Let $s\in\pi(B)\setminus\{q,r\}$ and let $S$ be a Sylow $s$-subgroup of $B$ such that $QS=SQ$.
 Such a subgroup exists by Satz VI. 2.3 in \cite{Huppert}.
If $S$ does not centralise $R$, then $RS$ is not nilpotent and therefore our arguments above show that $p^s=|P|=p^q$. This is impossible because $r\neq s$. Consequently $[R,S]=1$.

Since $B$ is directly indecomposable, we conclude that $SQ$ is not nilpotent. But $SQ$ is a Hall subgroup of $B$ and then it is a Hall subgroup of $K$. In particular (1) yields that $SQ$ is a batten. From the fact that $1\neq C_Q(R)=\Phi(Q)$ we conclude that $|Q|\neq q$, and then $|S|=s$ and $S\unlhd SQ$.
In addition $C_Q(S)=\Phi(Q)=C_Q(R)$. But $(R\times S)Q$ is an $L_9$-free group, and this situation contradicts Corollary \ref{kor:3primuntenil}.
\end{subproof}

We conclude that, by construction, $G=NK$, where $N$ is a nilpotent normal subgroup of $G$ with modular Sylow subgroups and $K$ is a batten group, by (3), such that for all $p \in \pi(N)$ it is true that every batten of $K$ centralises $O_p(N)$ or acts on $O_p(N)$ avoiding $L_9$, by (1) and (2).
\end{proof}

The converse of Proposition \ref{prop:Hinein} is false, as can be seen in the following example and subsequent lemma.

\begin{ex}Let $H=C_{19}\times C_{19}$, $J=C_5\times C_5$ and let $x,y\in GL(2,19)\times GL(2,5)$ be such that $$x=\left(\left(\begin{array}{cc}-1&0\\0&-1\end{array}\right), \left(\begin{array}{cc}0&3\\1&0\end{array}\right)\right) \text{ and }y=\left(\left(\begin{array}{cc}4&0\\0&4\end{array}\right), \left(\begin{array}{cc}2&3\\1&2\end{array}\right)\right).$$
Then $xy=\left(\left(\begin{array}{cc}-4&0\\0&-4\end{array}\right), \left(\begin{array}{cc}3&1\\2&3\end{array}\right)\right)=yx$.
Hence $G:= (H\times J)\rtimes (\erz{x}\times \erz{y})$ is a group.

Moreover, $N:=H\times J$ is a nilpotent normal subgroup of $G$ with modular Sylow subgroups.
Since
$$x^8= (x^2)^4=\left(\left(\begin{array}{cc}1&0\\0&1\end{array}\right), \left(\begin{array}{cc}3&0\\0&3\end{array}\right)\right)^4=\left(\left(\begin{array}{cc}1&0\\0&1\end{array}\right), \left(\begin{array}{cc}-1&0\\0&-1\end{array}\right)\right)^2=1$$ and
\\$y^9=\left(\left(\begin{array}{cc}7&0\\0&7\end{array}\right), \left(\begin{array}{cc}1&0\\0&1\end{array}\right)\right)^3$,
it follows that $x$ and $y$ have coprime order.\\
Thus $K:=\erz{x}\times \erz{y}$ is cyclic, which means that it is a batten group.

We see that $x$ and $y$ induce non-trivial power automorphisms on $H$. Thus every batten of $K$ acts on $H=O_{19}(N)$ avoiding $L_{19}$ of type (std).
In addition $x^2$ and $y^3$ induce power automorphisms on $J$. Since $x$ and $y$ act irreducibly on $J$, it follows that every batten of $K$ acts on $J=O_{5}(N)$ avoiding $L_{19}$ of type (std), too.

Altogether $G$ satisfies the conclusion of Proposition \ref{prop:Hinein}.

\medskip On the other hand we observe that $\pi(K)=\{2,3\}=\pi(\erz{x,y}/\erz{x^2})=\pi(K/C_K(H))$ and that $C_H(C_K(J))=C_H(\erz{y^3})=1$.
\\If $g\in HJ$ centralises $\erz{x}$ or $\erz{y}$, then $g=1$. Thus the following lemma yields that $G$ is not $L_9$-free.
\end{ex}

\begin{lemma}\label{lem:merlonOhneStar}
Let $H$ be a non-trivial abelian group where all non-trivial  Sylow subgroups are non-cyclic elementary abelian, let $L$ be a cyclic  group inducing power automorphism on $H$ such that $\pi(L)=\pi(L/C_L(H))$  and let $1\neq J$ be an abelian group admitting $L$ as a group of automorphisms such that the action of $L$ on $O_p(J)$ avoids $L_9$ for every $p\in \pi(J)$ and such that $(|H|, |J|)=1$.

Let $\pi:=\{q\in \pi(L)\mid  C_{O_q(L)}(H)<C_{O_q(L)}(J)\}$.
Suppose that  $C_H(C_L(J))=1$ and, for all $g\in (H\times J)^\#$, suppose that $g$ centralises neither $O_\pi(L)$ nor $O_{\pi'}(L)$.

Then $(H\times J)\rtimes L$ is not $L_{9}$-free.
\end{lemma}

\begin{proof}We first set $L_1:=O_{\pi}(L)$ and   $L_2:=O_{\pi'}(L)$. Since none of the groups $L_1$ nor $L_2$ is centralised by any element of $HJ\setminus\{1\}\neq \varnothing$, it follows that $L_1$ and $L_2$ are both non-trivial.

For every odd prime $p\in\pi(H)$ there is an elementary abelian subgroup $H_p$ of $H$ that has order $p^2$.
In particular there are elements $a_p$ and $b_p$ of $H$ such that $H_p=\erz{a_p}\times \erz{b_p}$.\\
We set $a:=\prod\limits_{p\in\pi(H)} a_p$ and $b:=\prod\limits_{p\in\pi(H)} b_p$.\\ These elements are well-defined (in the sense that the ordering of the primes does not matter) because $H$ is abelian.
For every $p\in \pi(J)$, we further see that $L$ acts on $O_p(J)$ avoiding $L_9$.
Since $J$ is abelian, we deduce from Lemma \ref{lem:Schmidt}~(b) that $[O_p(J), L]$ is elementary abelian. In addition $L$ acts irreducibly on $[O_p(J), L]$ or it induces power automorphisms on $O_p(L)$, by Lemma \ref{lem:laction}.
We choose $x_p\in [O_p(J),L]^\#$. Then $L$ acts irreducibly on $P:=\erz{x_p^L}$.
Next we choose $l\in L$ such that $L=\langle l \rangle$. Then $1\neq x_p^l\in P\leq HJ$ and $1\neq [l,x_p]\in P\leq HJ$. Thus our hypothesis implies that $1\neq [[l,x_p],L_1]\leq P$ and $1\neq [x_p^l,L_2]\leq P$. Altogether we have that $P=\erz{x_p^L}=\erz{(x_p^l)^L}=\erz{[[l,x_p],L_1]^{L}}=\erz{[x_p^l,L_2]^{L}}$.
Moreover, Lemma \ref{lem:avoidL9Cen} yields that $C_L(x_p)=C_L(P)=C_L(O_p(J))$.\\
For every $p\in\pi(L)$ we choose $1\neq x_p\in [O_p(J),L]$, and then we set $x:=\prod\limits_{p\in\pi(J)}x_p$.\\
Since $J$ is abelian, it follows that $C_L(J)=\bigcap_{p\in\pi(J)}C_L(O_p(J))=\bigcap_{p\in\pi(J)}C_L(x_p)=C_L(x)$.
Next we set $y:=x^l$. Then our previous arguments show that $J_0:=\erz{x^L}=\erz{y^L}=\erz{[[l,x],L_1]^{L}}=\erz{[y,L_2]^{L}}$ ($\ast$).

\medskip We will construct a subgroup lattice $L_9$ using Lemma \ref{lem:L9Char}.
For this we set $E:=C_L(HJ)\unlhd HJL$ and $D:=C_{L}(J)$.
\\For every $q\in \pi'$ we have $C_{O_q(L)}(H) \geq  C_{O_q(L)}(J)$ and so $C_{O_q(L)}(J)\leq E$.
This implies that  $C_{L_2}(J)=C_{L_2}(HJ)\leq E$ ($\ast \ast$) and that $D=C_L(J)=C_{L_1}(J)$. We conclude that $C_{L_1}(x)E=(L_1\cap C_L(x))E=(L_1\cap C_L(J))E=C_{L_1}(J)E=DE=D$.
\\If $q\in \pi=\pi(L_1)$, then  $C_{O_q(L)}(H)< C_{O_q(L)}(J)\leq C_{L_1}(x)$ and therefore $C_{L_1}(H)=C_{L_1}(HJ)\leq E$ ($\ast \ast \ast$). Then it follows that  $E\cap L_1=C_{L_1}(HJ)=C_{L_1}(H)< C_{L_1}(x)\leq D\cap L_1$.
In particular $E\neq D$  and  hence (L9(i)) of Lemma \ref{lem:L9Char} holds.

\smallskip Next we set $A:=\erz{a}L_1^xE$, $S:=L_1^{ax}E$ and $T:=L_1^{a^{-1}x}E$.
Then we have that $A$ contains the subgroups $S,T$ and $D$ ($=C_{L_1}(x)E$).
In addition, if $c\in\{a,a^{-1}, a^2\}$, then we see that $\erz{c}=\erz{a}$, because $o(a)$ is odd by construction.
Since $C_{\erz{a}}(D)\leq C_H(D)=C_H(C_L(J))=1$, it follows that $\erz{a}=\erz{c}=[c,D]\times C_{\erz{a}}(D)=[c,D]$, by Lemma \ref{coprime}.
The group  $L_1$ induces power automorphisms on $H$, which means that it normalises $[c,D]$.
Together with Part (d) of Lemma \ref{lem:Hilfslemma}  we conclude that $$\erz{D,L_1^cE}=\erz{[c,D]^{L_1^cE}}L_1^cE=[c,D]L_1^cE=\erz{c}L_1E=A^{x^{-1}}.$$
In particular, since $D$ centralises $x$, it follows that $\erz{D,T}=\erz{D,S}=A$. Moreover, we have that
$A\geq \erz{T,S}=\erz{L_1,L_1^{a^2}E}^{a^{-1}x}\geq \erz{D,L_1^{a^2}E}^{a^{-1}x}=(A^{x^{-1}})^{a^{-1}x}=A$ and therefore we conclude that $\erz{S,T}=A$
as well.
Next Lemma \ref{lem:avoidL9Cen} gives that $C_{L_1}(a)=C_{L_1}(H)\leq E$ by ($\ast \ast \ast$). Furthermore $L_1^c$
is a $\pi$-group  for all $c\in\{a,a^{-1},a^2\}$, whence
$$L_1E\cap L_1^c\leq O_{\pi}(L_1E)\cap L_1^c=L_1\cap L_1^c= C_{L_1}(c)= C_{L_1}(a)\leq E$$ by
Part (b) of Lemma \ref{lem:Hilfslemma}.
Altogether Dedekind's modular law gives that $L_1E\cap L_1^cE=(L_1E\cap L_1^c)E\leq E$  for all $c\in\{a,a^{-1},a^2\}$.
We conclude that $T\cap S=L_1^{ax}E\cap L_1^{a^{-1}x}E\leq E^x=E$, that $D\cap T\leq (L_1E\cap L_1^{a^{-1}}E)^x\leq E^x=E$ and that $D\cap S\leq (L_1E\cap L_1^{a}E)^x\leq E$. With all these properties,
we see that (L9(ii)) of Lemma \ref{lem:L9Char} is true.

\smallskip Now we set $C:=\erz{b}DL_2$ and $U:=\erz{b}L_2E$. Then $C=\erz{D,U}$ and $D\cap U=C_L(J)\cap \erz{b}L_2E=(C_L(J)\cap \erz{b}L_2)E=C_{L_2}(J)E$ by Dedekind's modular law and by Part (b) of Lemma \ref{lem:Hilfslemma}. Hence ($\ast \ast$) implies that (L9(iii)) of Lemma \ref{lem:L9Char} holds.

\smallskip Let $c\in\{a,a^{-1}\}$ and let $X:=\erz{L^{cx}_1,L_2}$. Then $X$ contains a $\pi$-Hall subgroup as well as a $\pi(L_2)$-Hall subgroup of $HJL$.
Since $HJL$ is soluble, there is a $\pi(L)$-Hall subgroup  $K$ of $X$ such that $L_2\leq K$ and  some $g\in HJ$ such that $L^g=K$.
It follows that $L_2\leq K\cap L=L^g\cap L=C_L(g)$ by Lemma \ref{lem:Hilfslemma}~(b).
The hypothesis of our lemma yields that $g=1$ and therefore $L=K\leq X$.
From there we obtain some $h\in X$ such that $L_1^h=L_1^{cx}$ and hence $L_1=L_1\cap L_1^{hx^{-1}c^{-1}}\leq C_{L}(hx^{-1}c^{-1})$ by Lemma \ref{lem:Hilfslemma}~(b).
This forces $cx=h\in X$, and then $c,x\in  X$, because $H$ and $J$ have coprime order and centralise each other.
Altogether $\erz{c}=\erz{a}$ and $\erz{x^L}=J_0$ are subgroups of $X$, and  we conclude that  $X=\erz{a}J_0L$.
Thus
$$\erz{U,T}=\erz{L_1^{a^{-1}x}, \erz{b}L_2E}=\erz{b}\erz{L_1^{a^{-1}x}, L_2,E}=\erz{b}XE=\erz{a,b}J_0L=\erz{b}\erz{L_1^{ax}, L_2,E}=\erz{U,L}.$$

We set $F:=\erz{a,b}J_0L$ in order to obtain Part (L9(iv)) of Lemma \ref{lem:L9Char}.
Moreover, Dedekind's law and Part (a) of Lemma \ref{lem:Hilfslemma} gives that
$$A\cap C=\erz{a}L_1^xE\cap \erz{b}DL_2=(\erz{a}L_1^xE\cap \erz{b}L_2)D$$
$$=(\erz{a}(L_1E\cap L_2)^x\cap \erz{b}(L_1E\cap L_2)) D=
(\erz{a}C_{L_2}(HJ)^x\cap \erz{b}C_{L_2}(HJ))D$$
$$=(\erz{a}C_{L_2}(HJ)\cap \erz{b}C_{L_2}(HJ))D=
(\erz{a}\cap \erz{b}C_{L_2}(HJ))C_{L_2}(HJ)D=C_{L_2}(HJ)D=D.$$
We set $B:=\erz{ab}L^y$. Then
$$A\cap B=(\erz{a}L_1E\cap \erz{ab}L^{yx^{-1}})^x\leq (\erz{a}(L_1E\cap HL^{yx^{-1}}))^x\leq(\erz{a}C_{L_1E}(yx^{-1}))^x\leq (\erz{a}C_L(J))^x=\erz{a}D$$
by Lemma \ref{lem:Hilfslemma}~(b), because $\erz{(yx^{-1})^L}\cap H\leq J\cap H=1$.

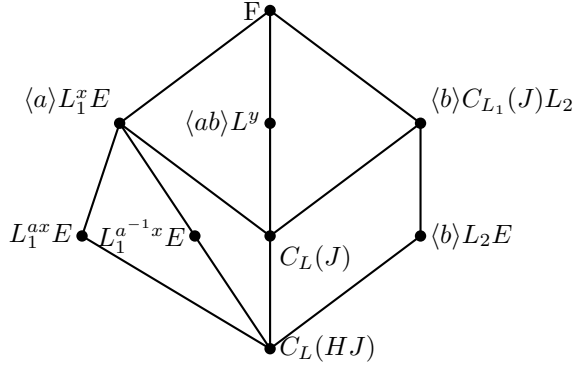
\begin{minipage}{7.6cm}
				\begin{tikzpicture}
					\coordinate[label=right:{$C_L(HJ)$}]		(E) at (0,0);
					\coordinate[label=left:{F}]		(F) at (0,4.5);
				\coordinate[label=above left:{$\erz{a}L_1^xE$}]	(A) at (-2,3);
					\coordinate[label=left:{$\erz{ab}L^y$}]		(B) at (0,3);			
					\coordinate[label=above right:{$\erz{b}C_{L_1}(J)L_2$}]	(C) at (2,3);
					\coordinate[label=left:{$L_1^{ax}E$}]			(S) at (-2.5,1.5);
					\coordinate[label=left:{$L_1^{a^{-1}x}E$}]			(T) at (-1,1.5);
					\coordinate[label=below right:{$C_{L}(J)$}]	(D) at (0,1.5);
					\coordinate[label=right:{$\erz{b}L_2E$}]		(U) at (2,1.5);
	
			\foreach \x in {E,F,A,B,C,D,S,T,U} \fill (\x) circle (2.15pt);
					
					\draw [thick]	(F) -- (A) -- (S) -- (E) -- (T) -- (A) -- (D) -- (E) -- (U) -- (C) -- (D) -- (B) -- (F) -- (C) ;
				\end{tikzpicture}

\end{minipage}\hfill
\begin{minipage}{7.9cm}
In a similar way we obtain that

$A\cap B\leq \erz{ab}D$ and therefore

$D\leq A\cap B\leq \erz{a}D\cap \erz{ab}D=(\erz{a}\cap \erz{ab}D)D=D$.
\end{minipage}

\medskip
We further calculate that
$$B\cap C=\erz{ab}L^y\cap \erz{b}DL_2\leq \erz{ab}(L^y\cap HDL_2)\leq\erz{ab}C_{DL_2}(y)\leq \erz{ab}C_L(J)=\erz{ab}D$$ and similarly
$B\cap C\leq \erz{b}D$. Therefore $D\leq B\cap C\leq \erz{ab}D\cap \erz{b}D=(\erz{ab}\cap \erz{b}D)D=D$.

Finally ($\ast$) and Part (d) of Lemma \ref{lem:Hilfslemma} yield that
$$\erz{A,B}=\erz{\erz{a}L_1^xE,\erz{ab}L^y}=\erz{a,b}\erz{L_1^xE,L^y}
=\erz{a,b}\erz{[yx^{-1},L_1]^{L}} L
=\erz{a,b}\erz{[[l,x],L_1]^{L}} L$$
$$=\erz{a,b}J_0L=F=
\erz{a,b}\erz{[y,L_2]^{L}}L
=\erz{a,b}\erz{L^y, DL_2}
=\erz{\erz{ab}L^y, \erz{b}DL_2}
=\erz{B,C}.$$

Altogether $\{A,B,C,D,E,F,S,T,U\}$ satisfies every condition of Lemma \ref{lem:L9Char}, which means that it is isomorphic to $L_9$.
    \end{proof}

The previous lemma and Lemma \ref{lem:exam} motivate the following definition:

\begin{definition}
			\label{defi:Merlongroup}
Here we define a class $\FL$ of finite groups, and each group in $\FL$ has a type.

We say that $G \in \FL$ has type $(N,K)$ if and only if
the following hold:
\begin{enumerate}
\item[(\FL 1)]$G=N\rtimes K$, where $N$ is a normal nilpotent Hall subgroup of $G$ with modular Sylow subgroups and $K$ is a batten group.
\item[(\FL 2)] If $p\in\pi(N)$, then every batten of $K$ centralises $O_p(N)$ or it acts on it avoiding $L_9$.
\item[(\FL 3)] For all Sylow subgroups $Q$ of $K$ and all distinct Sylow subgroups $P$ and $R$ of $N$ that are not centralised by $Q$, we have that $C_Q(P)\neq C_Q(R)$.
\item[(\FL 4)]
Suppose that $H\leq N$ is abelian, that its non-trivial Sylow subgroups are not cyclic and that $L\leq \mathrm{Pot}_K(H)$ is cyclic and such that $\pi(L)=\pi(L/C_L(H))$.
Let $1\neq J \le N$ be
$L$-invariant and abelian, suppose that $(|H|, |J|)=1$, $[H,J]=1$
and $C_H(C_L(J))=1$, and set\\
$\pi:=\{q\in \pi(L)\mid  C_{O_q(L)}(H)<C_{O_q(L)}(J)\}$.

Then
there is some $g\in (HJ)^\#$ that centralises  $O_\pi(L)$ or $O_{\pi'}(L)$.
\end{enumerate}
\end{definition}

\begin{thm}\label{main}
Let $G$ be a finite $L_9$-free group. Then $G \in \FL$.
\end{thm}

\begin{proof}From Lemma \ref{prop:Hinein} we see that $G=NK$ and that (\FL 1) and (\FL 2) are satisfied.

For (\FL 3) we let $Q$ be a Sylow subgroup of $K$ and we let $P$ and $R$ be distinct Sylow subgroups of $N$ that are not centralised by $Q$.
Since $N$ is a nilpotent normal Hall subgroup of $G$, it follows that $[P,R]=1$ and that $Q$ normalises $P$ and $R$. Then $(P\times R)Q$ is directly indecomposable and $L_9$-free, whence Corollary \ref{kor:3primuntenil} gives that $C_Q(P)\neq C_Q(R)$.

Finally, we look at (\FL 4) and we assume that it is not true.
Then there is an abelian subgroup  $H$ of $N$ such that the non-trivial Sylow subgroups are not cyclic, and we find a cyclic group  $L\leq \mathrm{Pot}_K(H)$ such that $\pi(L)=\pi(L/C_L(H))$ and a non-trivial  $L$-invariant abelian subgroup $J$ of $N$ such that $(|H|, |J|)=1$, $[H,J]=1$ and $C_H(C_L(J))=1$.
Let $\pi:=\{q\in \pi(L)\mid C_{O_q(L)}(H)<C_{O_q(L)}(J)\}$. Then we have, for all $g\in (HJ)^\#$, that $g$ centralises neither $O_\pi(L)$ nor $O_{\pi'}(L)$ for $\pi:=\{q\in \pi(L)\mid C_{O_q(L)}(H)<C_{O_q(L)}(J)\}$.

We note that $P$ does not centralise $O_{\pi}(L)\leq L$.
Then we find a prime $q\in\pi(L)$ such that a Sylow $q$-subgroup $Q$ of $L$ does not centralise $P$.
Using Lemma \ref{avoidbattensubgroup}, we see that $Q$ acts on $O_p(N)$ avoiding $L_9$ and then Lemma \ref{lem:avoidsubP} yields that $Q$ acts non-trivially on $P$, and avoiding $L_9$.
Now we may apply Corollary \ref{cor:p=2nK}: Since $L$ induces power automorphisms on $P$, Part (c) shows that $P$ is elementary abelian.
Then the hypotheses of Lemma \ref{lem:merlonOhneStar} are satisfied.
It says that
 $(H\times J)\rtimes L$ is not $L_{9}$-free, which is false.
We conclude that (\FL 4) holds.
\end{proof}

 \section{The Class $\FL$}

\begin{lemma}\label{merlonsol}
All groups in the class $\FL$ are soluble.
\end{lemma}

\begin{proof}Let $G \in \FL$ be of type $(N,K)$.
Then $N$ is nilpotent normal Hall subgroup of $G$ and $G/N\cong K$ is a direct product of $p$-groups or of groups whose order is divisible by exactly two primes.
Thus $G/N$ is soluble as well, and it follows that $G$ is soluble.
\end{proof}

\begin{lemma}\label{lem:Nhami}
Let $G \in \FL$ be of type $(N,K)$ and $\pi:=\pi([N,K])$. Then every subgroup of $O_\pi(N)$ is normal in $N$. \end{lemma}

\begin{proof}
Let $U$ be a subgroup of $O_{\pi}(N)$. Then $U\unlhd N$ if and only if $O_p(U)\unlhd O_p(N)$ for all $p\in \pi$, since $N$ is nilpotent.
Let $p \in \pi$. We note that this implies that $O_p(N)$ is not centralised by $K$. In particular there is a batten of $K$ that
acts non-trivially on $O_p(N)$ and avoiding $L_9$.
If $O_p(N)$ is abelian, then $O_p(U)\unlhd O_p(N)$.
If $O_p(N)$ is not abelian, then
we apply
Lemma \ref{lem:avoid1prop}~(b) to a batten $B$ of $K$ that acts non-trivially on $O_p(N)$.
The first possibility described there implies that $O_p(N)$ is abelian, which is not the case here. Thus the second possibility holds, and then $O_p(N)\cong Q_8\times I$, where $I$ is cyclic of order at most $2$.
We conclude that $O_p(N)$ is hamiltonian and it follows that $O_p(U)\unlhd O_p(N)$.
\end{proof}

\begin{lemma}\label{lem:merlonsub}
Let $G \in \FL$ be of type $(N,K)$ and $X\leq G$.
Then there is some $x\in [N,K]$ such that $X=(N\cap X)(K^x\cap X)$.
\end{lemma}

\begin{proof}We set $M:=N\cap X$. Then $M$ is a normal Hall subgroup of $X$, because $N$ is one of $G$.
Then the Schur-Zassenhaus Theorem provides a complement $C$ of $M$ in $X$, and we notice that $C$ and $M$ have coprime orders.
Therefore $\pi(C)=\pi(X)\setminus\pi(M)\subseteq\pi(G)\setminus\pi(N)=\pi(K)$.
It follows that $C$ is contained in a complement for $N$ in $G$. Since $G$ is soluble by Lemma \ref{merlonsol}, such a complement is conjugate to $K$, and thus we find $g\in G$ such that $C\leq K^g$.
The coprime action of $K$ on $N$ yields, together with Lemma \ref{coprime}, that
$N=C_N(K) [N,K]$, and therefore
$G=KN=KC_N(K)[N,K]$. We notice that $[N,K] \unlhd G$ and we
let $x\in [N,K]$ and $y\in KC_N(K)\leq N_G(K)$ be such that $g=yx$.
Then $C=K^g\cap X=K^x\cap X$ and hence
$X=MC=(N \cap C)(K^x \cap X)$.
\end{proof}

\begin{lemma}\label{lem:uG}
Let $G \in \FL$ be of type $(N,K)$ and suppose that $U\leq G$.
Then $U$ is a group in class $\FL$ of type $(U\cap N,U\cap K^g)$ for some $g\in [N,K]$.
\end{lemma}
\begin{proof}Lemma \ref{lem:merlonsub} provides some $g\in [N,K]$ such that $U=(U\cap N)\cdot (U\cap K^g)$. By conjugation we may suppose that $g=1$ and we set $K_1:=U\cap K$.
Then Lemma \ref{lem:battensub} yields that  $K_1=U\cap K^g\leq K^g\cong K$ is a batten group.
Moreover, $M:=U\cap N\leq N$ is a normal nilpotent Hall subgroup of $U$ with modular Sylow subgroups, by (\FL 1).
This means that (\FL 1) holds for $U$, and now we turn to (\FL 2) and let $p\in\pi(M)$.
Suppose that $B$ is a batten of $K_1$ that does not centralise $O_p(M)$. Then it does not centralise $O_p(N)$ and therefore Lemma  \ref{avoidbattensubgroup} implies that $B\cong B/1$ acts on $O_p(N)$ avoiding $L_9$.
Then we may apply Lemma \ref{lem:avoidsubP} to see that $B$ also acts on $O_p(M)/1\cong O_p(M)$ avoiding $L_9$.

This gives property $(\FL2)$ of Definition \ref{defi:Merlongroup} for $U$, and $(\FL4)$
follows because $M\leq N$ and $K_1\leq K$.

For (\FL 3)
we let $Q_1$ be a Sylow subgroup of $K_1$ and we let $p,r\in\pi(M)$ be different primes such that $[O_p(M),Q_1]\neq 1\neq[O_r(M),Q_1]$.
We need to prove that
$C_{Q_1}(O_p(M)) \neq C_{Q_1}(O_r(M))$.

First we let $Q$ be a Sylow subgroup of $K$
that contains $Q_1$. Then $[O_p(N),Q]\neq 1\neq[O_r(N),Q]$ and therefore
$C_{Q}(O_p(N))\neq C_{Q}(O_r(N))$, using Property (\FL 3) for $G$.
In particular, these centralisers cannot both be trivial, and
we may suppose that $C_{Q}(O_p(N))\neq 1$. Then $Q$ does not act faithfully on $O_p(N)$, but the action of $Q$ on $O_p(N)$ avoids $L_9$. Definition \ref{defi:avoidQ8} immediately gives that $Q \not \cong Q_8$.
Then
it follows that $Q$ is cyclic, which means that the subgroup lattice $L(Q)$ of $Q$ is a chain, and $Q_1$ is also cyclic.

We assume for a contradiction that
$C_{Q_1}(O_p(M)) = C_{Q_1}(O_r(M))$.
Then
Lemma \ref{lem:avoidL9Cen}, with $Q_1$ in the role of $H$, $O_p(M)$ in the role of $R$ and $O_p(N)$ in the role of $P$, gives that
$C_{Q_1}(O_p(M))=C_{Q_1}(O_p(N))$.
Similarly
$C_{Q_1}(O_r(M))=C_{Q_1}(O_r(N))$, and then
it follows that
$C_{Q_1}(O_p(N)) = C_{Q_1}(O_r(N))$.
We recall that $C_{Q}(O_p(N))\neq C_{Q}(O_r(N))$ and that $Q$ is cyclic, and now we may suppose that $C_{Q}(O_p(N)) \lneq C_{Q}(O_r(N))$. This forces $C_{Q_1}(O_p(N))\lneq C_{Q}(O_p(N))$, which is impossible.
Thus $C_{Q_1}(O_p(M)) \neq C_{Q_1}(O_r(M))$
and  (\FL 3) holds for $U$ as well.
\end{proof}

\begin{lemma}\label{lem:Nmaxgewaehlt}
Let $G \in \FL$ be of type $(N,K)$ and suppose that $S$ is a normal Sylow $q$-subgroup of $K$ that centralises $N$ for some prime $q$.
Let $K_1$ be a Hall $q'$-subgroup of $K$.
\\Then $G$ is also of type $(N\times S,K_1)$.
In particular, if we choose $(N,K)$ such that $|N|$ is as large as possible, then $\pi(K)=\pi(K/C_K(N))$.
\end{lemma}

\begin{proof}
We show that $G=(N \times S)K_1$ satisfies ($\FL1$)-($\FL4$) of Definition \ref{defi:Merlongroup}, and we first note that $K_1$ is a batten group by
Lemma \ref{lem:battensub}.
The structure of $K$ forces all Sylow subgroups of $K$ to be cyclic or quaternion, more specifically $S$ is cyclic or isomorphic to $Q_8$. This means that $S$ is modular.
Since $S$ is a normal Sylow $q$-subgroup of $K$ and $N$ is a Hall subgroup of $G$, by hypothesis, it follows that $N \times S$ is a Hall subgroup of $G$ where all Sylow subgroups are modular.

By hypothesis $[N,S]=1$ and $N$ is nilpotent, hence $N\times S$ is nilpotent, too. This is ($\FL1$).
\\For ($\FL2$) we let $B$ be a batten of $K_1$ and $p\in\pi(N \times S)$.
We keep in mind that $B$ is not necessarily a batten of $K$ -- if it is, then it centralises $O_p(N \times S)$ or it acts on it avoiding $L_9$, because of
($\FL2$) for $G$.

Now we suppose that $B$ is not a batten of $K$ and that $[O_p(N \times S),B]\neq 1$. Then $SB$ is a non-nilpotent batten of $K$.  If $p\neq q$, then $SB$ acts on $O_p(N)$ avoiding $L_9$, and $[O_p(N),S]=1$. Then Definition \ref{defi:avoidbatten} implies that $SB$ acts of type (Cy) and it follows that $B$ acts on $O_p(N)$ avoiding $L_9$. Finally suppose that $q=p$. Then $\Phi(B)$ centralises $S=O_p(NS)$, while $B$ induces power automorphisms on the cyclic group $S$ of order $p$. Thus $SB$ satisfies (std) of Definition \ref{defi:avoidcyclic}, and we deduce that $B$ acts on $S$ avoiding $L_9$.

We turn to ($\FL3$). Let $Q$ be a Sylow subgroup of $K_1$ and let $P$ and $R$ be distinct Sylow subgroups of $N \times S$ that are not centralised by $Q$.
First we note that $Q$ is a Sylow subgroup of $K$ because $K_1$ is a Hall subgroup of $K$ by hypothesis.
Therefore, if $PR\leq N$, then we immediately have that $C_Q(R)\neq C_Q(P)$, by ($\FL3$) in $G$.

Without loss suppose that $R \nleq N$, i.e. $R=S$.
Then we recall that $Q$ was chosen not to centralise $P$ and $R=S$, which means that
$Q$ and $S$ cannot come from distinct battens of $K$, but their product must be a non-nilpotent batten of $K$.
Moreover, $[P,QS]\neq 1$.
We obtain from Lemma \ref{lem:batten} and Definition \ref{defi:avoidbatten} that $\Phi(Q)=Z(Q)=C_Q(S)$ does not centralise $P$ and therefore $C_Q(R)\neq C_Q(P)$. This is ($\FL3$).
\\Finally, let $H\leq N \times S$ be such that its non-trivial  Sylow subgroups are not cyclic, let $L\leq \mathrm{Pot}_{K_1}(H)$ be cyclic and such that $\pi(L)=\pi(L/C_L(H))$  and let $1\neq J$ be an $L$-invariant abelian subgroup of $M$ such that $(|H|, |J|)=1$, $[H,J]=1$ and $C_H(C_L(J))=1$.

As $K$ is a batten group, the set  $\pi(K/C_K(S))$ contains at most one element.
We recall that $L\leq K_1\leq K$, and then it follows that,
for every set of primes $\pi$, the group $S$ centralises $O_\pi(L)$ or $O_{\pi'}(L)$.
In order to prove ($\FL4$) of Definition \ref{defi:Merlongroup}, we may thus suppose that $H$ and $J$ are subgroups of $N$.

Then $L\leq K_1\leq K$ shows that $L\leq \mathrm{Pot}_{K}(H)$.
Hence we apply ($\FL4$) of Definition \ref{defi:Merlongroup} to $G$, i.e. to the type $(N,K)$.
If $\pi:=\{q\in \pi(L)\mid  C_{O_q(L)}(H)<C_{O_q(L)}(J)\}$, then
we find some $g\in (HJ)^\#$ that centralises $O_\pi(L)$ or $O_{\pi'}(L)$.

Altogether, $G=(N\times S) K_1$ satisfies Definition \ref{defi:Merlongroup}.

We now suppose that $|N|$ is as large as possible and we assume for a contradiction that $S\leq C_K(N)$ is a Sylow subgroup of $K$. Then $S$ is not normal in $K$, hence there is a non-nilpotent batten $B$ of $K$ such that $B=\mathcal{B}(B)S$.
For all $p\in\pi(N)$ it follows that $[O_p(N),Z(B)]\leq [O_p(N),S]=1$, by Lemma \ref{lem:batten}. Then Definition \ref{defi:avoidbatten} yields that $B$ does not act on $O_p(N)$ avoiding $L_9$, whence $B$ centralises $O_p(N)$. But now $\mathcal{B}(B)\leq C_K(N)$ and $\mathcal B(B)$ is a normal Sylow subgroup of $K$. This contradicts the maximal choice of $N$.
\end{proof}

\begin{lemma}\label{lem:faktorStar}
Let $G \in \FL$ and suppose that $M\unlhd G$.
Then $G/M \in \FL$.
\end{lemma}
\begin{proof}
By induction we may suppose that $M$ is a minimal normal subgroup of $G$.
We recall that $G$ is soluble by Lemma \ref{merlonsol}, and we let $r$ be prime such that $M$ is an elementary abelian $r$-group.
If $M$ has a complement $C$ in $G$, then $G/M\cong C$ and Lemma \ref{lem:uG} gives that $C\in \FL$ and hence $G/M\in \FL$.

Consequently we may suppose that $M$ does not have a complement in $G$.
We choose $N,K\leq G$ such that $G$ has type $(N,K)$ and such that $|N|$ is as large as possible. Then $\pi(K)=\pi(K/C_K(N))$ by Lemma \ref{lem:Nmaxgewaehlt}.

First suppose that $r\in \pi(K)$. Then $M\leq K$ and we see that $[N,M]\leq N\cap M\leq N\cap K=1$, because $N$ is a Hall subgroup of $G$. Hence $M\leq C_K(N)$.
Next we let $B$ be a batten of $K$ that contains $M$.
Then $M\leq C_B(N)$, which means that
for all $p\in\pi(N)$,
$B$ does not act faithfully on $O_p(N)$.

Since there is some $p\in \pi(N)$ such that $B$ acts on $O_p(N)$ avoiding $L_9$, it follows from Definition \ref{defi:avoidQ8} that $B$ is not isomorphic to $Q_8$.
In addition $M\lneq Z(B)$, if $B$ is not nilpotent, by Definition \ref{defi:avoidbattengroup}. Therefore, in this case, the section $B/M$ is a non-nilpotent batten as well.
We conclude that $K/M$ is a batten group.

Assume for a contradiction that $r\in \pi(N)$, but that $M\nleq C_N(K)$. Then we note that $C_M(K)\unlhd G$ by Lemma \ref{lem:Nhami}, and we deduce that $C_M(K)=1$, because $M$ is a minimal normal subgroup of $G$.
This forces $M\leq [O_r(N),K]$.
Then $[O_r(N),K]$ is not elementary abelian, because otherwise $M$ would have a complement in this commutator and hence in $G$.
But we are working under the hypothesis that it does not.

Now Lemma \ref{lem:laction} and Lemma \ref{lem:avoid1prop}~(b) imply that $M\leq [O_r(N),K]\cong Q_8$. But now $Z([O_r(N),K])$ is the unique subgroup of order $2$ of $[O_r(N),K]$, which means that it must be centralised by $K$ and contained in $M$. But this is a contradiction.
Thus $M\leq C_N(K)$ and Corollary \ref{cor:p=2nK} implies that $p=2$.
We summarise that $M\leq C_K(N)$ and that $K/M$ is a batten group or $M\leq C_N(K)$ and $p=2$.

Let $-:G\to G/M$ be the natural homomorphism.
Then $\bar G=\bar N\cdot \bar K$, where $\bar N$ is a normal nilpotent Hall subgroup of $G$ with modular Sylow subgroups, since sections of modular $p$-subgroups are modular.
Moreover $\bar K$ is a batten group. Thus $\bar G$ satisfies $(\FL 1)$ of Definition \ref{defi:Merlongroup}.
We further deduce $(\FL2)$ from Lemma \ref{avoidbattensubgroup} and  Lemma \ref{lem:avoidsubP}.

For $(\FL 3)$ we let  $Q$ be a Sylow subgroup of $K$ and we let $p,s\in \pi(\bar N)$ be distinct primes such that $[O_s(\bar N),\bar Q]\neq [O_p(\bar N),\bar Q]$.
Then $[O_s(N),Q]\neq [O_p(N),Q]$ and therefore $C_Q(O_s(N))\neq C_Q(O_p(N))$. Since $M\leq C_N(Q)$ or $M\leq C_K(O_s(N))\cap C_K(O_p(N))$, it follows that $C_{\bar Q}(O_s(\bar N))\neq C_{\bar Q}(O_p(\bar N))$.

We finally let $1\neq \bar H\leq \bar N$ be abelian with non-cyclic Sylow subgroups and $\bar  L\leq \mathrm{Pot}_{\bar K}(\bar H)$ be cyclic with $\pi(\bar L)=\pi(\bar L/C_{\bar L}(\bar H))$  and we let $\bar J$ be an abelian $\bar L$-invariant subgroup of $\bar N$ such that $(|\bar H|, |\bar J|)=1$
and $C_{\bar H}(C_{\bar L}(\bar J))=1$.
We set $\bar \pi:=\{q\in \pi(\bar L)\mid \forall\, \bar Q\in \mathrm{Syl}_q(\bar L): C_{\bar Q}(\bar H)<C_{\bar Q}(\bar J)\}$.

Then we assume for a contradiction that every non-trivial element $\bar g\in \bar H\bar J$ centralises neither $O_{\pi}(\bar L)$ nor $O_{\pi'}(\bar L)$. Then Lemma \ref{coprime} yields that $[\bar H\bar J,\bar L]=\bar H\bar J$.
We choose pre-images $H$, $L$ and $J$ in $G$ of smallest possible order.
Then $HJ=[HJ,L]$ and $\pi(\bar X)=\pi(X)$ for all $X\in \{H,\ L,\ J\}$, because $G$ is soluble by Lemma \ref{merlonsol}. In particular we have that $(|H|, |J|)=1$.

If $r\in\pi(X)$ for some $X\in \{H,J\}$,
then $r=2$. Then our assumption implies that $C_{O_2(X)}(L)\leq M$.
It follows that  $X\cong \bar X$ or that $M\leq \Phi(X)=\Phi([X,L])$.
In the second case, we apply Lemma \ref{lem:laction} and Lemma \ref{lem:avoid1prop}. Together they show that $[O_2(N),K]=[O_2(N),L]\cong Q_8$ and thus $\pi(L/C_{O_2(N)}(L))=\{3\}$. This means that $O_2(N)$ centralises $O_\pi(L)$ or $O_{\pi'}(L)$. But then we also have that $[O_2(\bar X), O_\sigma(\bar L)]=1$ for some $\sigma\in\{\pi,\pi'\}$, which is a contradiction.
We deduce that $\bar H\cong H$ and $\bar J\cong J$.
In particular $H\leq N$ is abelian, with non-cyclic Sylow subgroups, and $J\neq 1$ is an abelian $L$-invariant subgroup of $N$.
Since $\bar L$ is cyclic, it follows from our arguments above that $L$ is also cyclic. Moreover $\pi(L)=\pi(\bar L)=\pi(\bar L/C_{\bar L}(\bar H))=\pi(L/C_{L}(H))$, because $M\leq C_K(N)$ or $M\cap L=1$.
\\We now investigate the action of $L$ on $H$. Since $\bar H\cong H$ and $M\cap L=1$ or $M\leq C_L(H)$, we see that $L$ induces power automorphisms on $H$.
In addition Lemma \ref{coprime} shows that $$\overline{C_H(C_L(J))}\cong C_{\bar H}(\overline{C_{L}(J)})\cong C_{\bar H}(C_{\bar L}(\bar J))=1$$ and then the fact that $H\cap M=1$ yields that $C_H(C_L(J))=1$.
Altogether we obtain, by applying ($\FL4$) to $G$, some $g\in HJ^\#$ such that $g$ centralises $O_\pi(L)$ or $O_{\pi'}(L)$, where $\pi:=\{q\in \pi( L)\mid \forall\,  Q\in \mathrm{Syl}_q(L): C_{Q}( H)<C_{ Q}(J)\}$.
Since $\bar H\bar J\cong HJ$, it follows that $\bar g\neq 1$ and $[O_\pi(\bar L),\bar g]=1$ or $[O_{\pi'}(\bar L),\bar g]=1$ .
Again we use that $\bar L$ acts on $\bar H\cong H$ and $\bar J\cong L$ equivalently to $L$, because $M\cap L\leq C_L(N)$. Then we see that
$$\bar \pi:=\{q\in \pi(\bar L)\mid \forall\, \bar Q\in \mathrm{Syl}_q(\bar L): C_{\bar Q}(\bar H)<C_{\bar Q}(\bar J)\}
%=\{q\in \pi( L)\mid \forall\,  Q\in \mathrm{Syl}_q(L): C_{Q}( H)<C_{ Q}(J)\}
=\pi.$$ This is a contradiction.
\end{proof}

\begin{lemma}\label{lem:newunique}
Let $G \in \FL$ be of type $(N,K)$ such that $C_K(N)=1$, let $q \in \pi(K)$ and let $Q\in\syl_q(K)$.
Then  $1\neq [N,\Omega_1(Q)]$ has prime power order.
\end{lemma}
\begin{proof} We apply Lemma \ref{lem:uG} and we see that $N\Omega_1(Q)\in \FL$ has type $(N,\Omega_1(Q))$.
Since $\Omega_1(Q)$ does not centralise $N$, there is a prime $p\in \pi(N)$ such that $[O_p(N),\Omega_1(Q)]\neq 1$.
It follows that $C_{\Omega_1(Q)}(O_p(N))=1$.
Now $(\FL3)$ implies that $\Omega_1(Q)$ centralises $O_r(N)$ for every $r\in \pi(N)\setminus\{p\}$, and this shows that $1 \neq [N,\Omega_1(Q)]\leq [O_p(N),\Omega_1(Q)]\leq O_p(N)$.
\end{proof}

\begin{lemma}\label{lem:newPropertiesZinN}
Let $G \in \FL$ be of type $(N,K)$ such that $C_K(N)=1$, let $q\in \pi(K)$ and let $Q\in\syl_q(K)$.
Let $p\in\pi(N)$ be such that $\Omega_1(Q)$ does not centralise $P:=O_{p}(N)$.
Then the following hold:
\begin{enumerate}
\item $N_G(\Omega_1(Q))=(O_{p'}(N)K)\times C_{P}(K)$.
\item $G=[P,\Omega_1(Q)] N_G(\Omega_1(Q))$.
\item $[P,\Omega_1(Q)]=[P,K]$ acts transitively on $\Omega_1(Q)^G=\{Q_0\leq G\mid |Q_0|=q\}$.
\item If $X \le G$, then there is some $x\in [P,\Omega_1(Q)]$ such that $X=(X\cap P)N_X(\Omega_1(Q)^x)$.
%\item If $x,y\in [P,Q]$ are such that $X\leq N_G(Q^x)\cap N_G(Q^y)$, then $[P,X]=1$ or $xy^{-1}\in C_P(K)$.
\item Suppose that $X \le G$, that $q$ divides $|X|$ and that $x\in P$. Then $X=(X\cap P)N_X(\Omega_1(Q)^x)$ if and only if $\Omega_1(Q)^x\leq X$.
\end{enumerate}
\end{lemma}

\begin{proof}We set $P_0:=[P,\Omega_1(Q)]$.
Then  Lemma \ref{lem:newunique} implies that $P_0=[N,\Omega_1(Q)]$ and so $O_{p'}(N)\leq C_G(\Omega_1(Q))\leq N_G(\Omega_1(Q))$.
Furthermore $K$ acts on $P$ avoiding $L_9$ and then we have that $P_0=[P,Q]=[P,K]$ by Lemma \ref{lem:laction}.
Next $K\leq N_G(\Omega_1(Q))$ from Lemma \ref{battenremark}.
\\Since $N$ is nilpotent, we conclude that $N_G(\Omega_1(Q))=(O_{p'}(N)\times N_P(\Omega_1(Q)))K$.
But we also have that $[N_{P}(\Omega_1(Q)),\Omega_1(Q)]\leq P\cap \Omega_1(Q)=1$, whence $N_{P}(\Omega_1(Q))\leq C_P(\Omega_1(Q))= C_P(K)$ by Lemma \ref{lem:avoidL9Cen}. Consequently $N_P(\Omega_1(Q))=C_P(\Omega_1(Q))$ and it follows that $N_G(\Omega_1(Q))=(O_{p'}(N)K)\times C_P(K)$, as stated in (a).

For (b) we
recall that,
by (a), the subgroups $K$ and $O_{p'}(N)$ normalise $\Omega_1(Q)$.
Then $G=P O_{p'}(N)K \le P N_G(Q)\le G$. Moreover $P \unlhd G$
and Lemma \ref{coprime} implies that $P=C_P(\Omega_1(Q))P_0$, where $C_P(\Omega_1(Q)) \le N_G(\Omega_1(Q))$ and therefore
$G=P_0N_G(\Omega_1(Q))$ as stated in (b).

 From there we deduce that $P_0$ acts transitively on $\Omega_1(Q)^G$ by conjugation. The second statement of (c) follows because $G$ is soluble (Lemma \ref{merlonsol}), together with the fact that $\Omega_1(Q)$ is the unique subgroup of its order in the Hall subgroup $K$ of $G$ (Lemma \ref{battenremark}). This means that
 every subgroup of order $q$ of $G$ is conjugate to $\Omega_1(Q)$, completing (c).

 \smallskip
 For (d) and (e) we let $X \le G$.
Lemma \ref{lem:merlonsub} provides some $g\in G$ such that $X=(N\cap X)(K^g\cap X)$.
Moreover, (a) implies that
$K^g$ and $O_{p'}(N) (=O_{p'}(N)^g)$ normalise $\Omega_1(Q)^g$, and then we summarise:
$$X=(N\cap X)(K^g\cap X)
\leq (P\cap X) (O_{p'}(N)\cap X)(K^g\cap X)\leq (P\cap X)N_X(\Omega_1(Q)^g)\leq X.$$
Using (b) we see that $G=N_G(\Omega_1(Q))P_0$, and then we take $y\in N_G(\Omega_1(Q))$ and $x\in P_0$ such that $g=yx$.
Now we deduce that $X=(P\cap X)N_X(\Omega_1(Q)^{yx})=(P\cap X)N_X(\Omega_1(Q)^x)$, as stated in (d).

Finally, suppose that $q$ divides $|X|$ and that $x\in P$.
Suppose first that $X=(X\cap P)N_X(Q^x)$.
 Then $q$ divides $|N_X(\Omega_1(Q)^x)|$, which provides a subgroup $Q_0$ of order $q$ in $N_X(\Omega_1(Q)^x)\leq N_G(\Omega_1(Q)^x)$ and, by (d), there is some $y\in P_0$ such that $\Omega_1(Q)^y=Q_0\leq N_G(\Omega_1(Q)^x)= N_G(\Omega_1(Q))^x$. Then $\Omega_1(Q)$ and $\Omega_1(Q)^{yx^{-1}}$ are subgroups of $N_G(\Omega_1(Q))$ and therefore $$[yx^{-1},\Omega_1(Q)]=[[yx^{-1},\Omega_1(Q)],\Omega_1(Q)]\leq [[P,\Omega_1(Q)]\cap N_G(\Omega_1(Q)),\Omega_1(Q)]\leq [C_P(\Omega_1(Q)),\Omega_1(Q)]=1$$ by Lemma \ref{coprime}.
 We conclude that $\Omega_1(Q)^x=\Omega_1(Q)^y=Q_0\leq N_X(\Omega_1(Q)^x)\leq X$.

Now, conversely, suppose that $\Omega_1(Q)^x\leq X$. Then (d) provides some $z \in P_0$ such that $X=(P\cap X)(N_X(\Omega_1(Q)^z))$.
In the paragraph above we have shown that $\Omega_1(Q)^z\leq X$.
We apply (c) to $X=(X\cap N)(X\cap K^g)$, which is a group in $\FL$ by Lemma \ref{lem:uG}, and we obtain some $y\in P\cap X$ such that $\Omega_1(Q)^{zy}=\Omega_1(Q)^x$.
We conclude that
$$X=X^y=(P\cap X)^y(N_{X^y}(\Omega_1(Q)^{zy}))=(P\cap X)(N_X(\Omega_1(Q)^x)),$$
because $P\cap X\unlhd X$.\end{proof}

\begin{lemma}\label{lem:powerneu}
Let $G \in \FL$ be of type $(N,K)$ and let $p\in\pi(N)$ such that $K$ induces non-trivial power automorphisms on $P:=O_{p}(N)$.

Then  for all $X,\ Y\leq G$,  there is some $i\in\{0,1\}$ such that $|\erz{X,Y}\cap P|=|(P\cap X)(P\cap Y)|\cdot p^i$.

In addition $i=0$ if and only if there is some $g\in P$ such that for both $Z\in\{X,Y\}$ we have $Z\leq (Z\cap P)O_{p'}(N)K^g$.
\end{lemma}

\begin{proof} We first remark that Lemma \ref{lem:Nhami} gives that every subgroup of $P$ is normal in $N$, and hence in $G$, because $K$ normalises every subgroup of $P$ as well. In addition $P=[P,K]$ is elementary abelian by Corollary \ref{cor:p=2nK} (c).

Let  $X,\ Y\leq G$.
Then Lemma \ref{lem:merlonsub} provides $x,y\in P$ such that $X\leq (X\cap P)O_{p'}(N)K^x$
and
$Y\leq (Y\cap P)O_{p'}(N)K^y$.

This implies that $\erz{X,Y}\leq (X\cap P)(Y\cap P)\erz{x^{-1}y}O_{p'}(N)K^x$, bearing in mind that $X \cap P$ and $Y \cap P$ are normal subgroups of $G$, and therefore $\erz{X,Y}\cap P=(P\cap X)(P\cap Y)\erz{x^{-1}y}.$

Since $P$ is elementary abelian, we
see that $o(xy^{-1}) \in \{1,p\}$ and we
deduce the first assertion.

If it is possible to choose $x=y$, then $\erz{X,Y}\cap P=(P\cap X)(P\cap Y)\erz{x^{-1}y}=(P\cap X)(P\cap Y)$ and in particular $i=0$ in the statement of the lemma.

For the converse we suppose that $i=0$, i.e. $|\erz{X,Y}\cap P|=|(P\cap X)(P\cap Y)|$. Then $x^{-1}y\in \erz{X,Y}\cap P=(P\cap X)(P\cap Y)$ and thus we find $x_0\in X\cap P$ and $y_0\in Y\cap P$ such that $x^{-1}y=x_0y_0$.
Then $g:=yy_0^{-1}=xx_0\in P \cap X \cap Y$.
We note that $x_0$ normalises $X$, centralises $P \cap X$ and normalises $O_{p'}(N)$, which implies that
$X=X^{x_0}\leq ((X\cap P)O_{p'}(N)K^x)^{x_0}
=(X\cap P)O_{p'}(N)K^{xx_0}$.
Similarly $Y\leq(Y\cap P)O_{p'}(N)K^{yy_0^{-1}}$.
\end{proof}

\begin{lemma}\label{lem:neuproperties}
Let $G \in \FL$ be of type $(N,K)$.
Suppose that $X$ and $Y$ are subgroups of $G$ such that $\erz{X,Y}=G$ and let $B$ is a batten of $K$. Suppose that $K$ has a normal $q$-complement $H$. Then one of the following hold:
\begin{enumerate}
\item  $q\nmid |G:X|$,
\item $q\nmid |G:Y|$, or
\item $q=2$, $K$ has a section isomorphic to $Q_8$ and $4$ divides $(|X|,|Y|)$.
\end{enumerate}
\end{lemma}
\begin{proof}Let $Q\in \syl_q(K)$ and
suppose that $q$ divides neither $|G:X|$ nor $q\mid |G:Y|$.
Then Lemma \ref{lem:merlonsub} gives maximal subgroups $M_X$ and $M_Y$ of $Q$ such that $X\leq NHM_X$ and $Y\leq NHM_Y$.

If $Q$ is cyclic, then $M_Y=M_X=\Phi(Q)$ and it follows that $G=\erz{X,Y}\leq NH\Phi(Q)\neq NHQ=G$. This is impossible.
We conclude that $Q$ is not cyclic and then, since $K$ is a batten group, it follows that $Q\cong Q_8$.
We assume for a contradiction that $4\nmid |X|$. Then $X\leq NH\Phi(Q)$ and hence $G=\erz{X,Y}\leq NHM_Y\neq NHQ=G$, which is again a  contradiction.
\end{proof}

\begin{lemma}\label{lem:irredu}
Let $G\in \FL$ be  of type $(N,K)$ such that $K$ acts irreducibly on $[O_p(N),K]/\Phi([O_p(N),K])$ for some prime $p\in\pi(N)$.
If $X,Y\leq G$ are such that $\erz{X,Y}=G$, then $X$ or $Y$ acts irreducibly on $[O_p(N),K]/\Phi([O_p(N),K])$.
\end{lemma}
\begin{proof}Let $X, Y\leq G$ be such that $\erz{X,Y}=G$ and let $P:=O_p(N)$. We assume for a contradiction that neither $X$ nor $Y$ act irreducibly on $[P,K]/\Phi([P,K])$.
Lemma \ref{lem:Nhami} implies that  $N$ normalises every subgroup of $P$ ($\ast$).
Moreover, by Lemma \ref{lem:merlonsub}, there are $x,y\in N$ such that $X=(X\cap N)(X\cap K^x)$ and $Y=(X\cap N)(Y\cap K^y)$ and, by assumption, neither $X\cap K^x$ nor $Y\cap K^y$ act irreducibly on $[P,K]/\Phi([P,K])$.
It follows from Lemma \ref{lem:laction} that $X\cap K^x$ and $Y\cap K^y$ both induce power automorphisms on $P$ and that $|P|\neq p$. Thus ($\ast$) yields that $X$ and $Y$ normalise every subgroup of $P$. Then also
 $G=\erz{X,Y}$ normalises every subgroup of $P$, which contradicts the irreducible action of $K$.
\end{proof}

\section{The main result}

\begin{thm}\label{thm:Rueck}
If $G \in \FL$, then $G$ is $L_9$-free.
\end{thm}
\begin{proof}
Assume for a contradiction that the statement is false and let $G$ be a minimal counterexample.
Then there is a sublattice $\mathcal L = \{E, S, T, D, U, A, B, C, F\}$  of $L(G)$ isomorphic to $L_9$, and in particular $\mathcal L$ satisfies the relations in Definition \ref{defi:L10}.

We let $G$ be of type $(N,K)$ where, among the minimal counterexamples, we choose $G$ such that $|N|$ is as large as possible and we set
$$\pi(K)^*:=\pi(K)\setminus\{|\mathcal B(H)|\mid H \text{ is a non-nilpotent batten of }K\}.$$
Then $K$ has a normal $q$-complement for every $q\in \pi(K)^*$ and Lemma \ref{lem:neuproperties} is applicable.

We will first analyse how $\mathcal L$ fits into the subgroup lattice of $G$.

\smallskip (1) $F=G$, $C_K(N)=1$ and every subgroup of $N$ is normal in $N$.

\begin{subproof}The group $F$ is a  subgroup of $G$
that is not $L_9$-free, and
Lemma \ref{lem:uG} yields that $F \in \FL$.
Hence the minimal choice of $G$ implies that $F=G$.
\\Similarly, it follows from Lemma \ref{lem:L9GdirProdTfOrd} that $G$ is not a direct product of two non-trivial groups of coprime order.
Let $p\in\pi(N)$. Then $N=O_{p'}(N)\times O_p(N)$ because $N$ is nilpotent. If $K$ centralises $O_p(N)$, then $G=O_{p'}(N)K\times O_p(N)$,
where the direct factors have coprime order by ($\mathcal L 1$).
But we just saw above that such a direct decomposition of $G$ is not possible.
Therefore $[O_p(N),K] \neq 1$ and
 $p$ divides $|[O_p(N),K]|$, which divides  $|[N,K]|$.
We conclude from Lemma \ref{lem:Nhami} that every subgroup of $O_p(N)$ is normal in $N$.
This implies that every subgroup of $N$ is a normal subgroup of $N$, because $N$ is nilpotent.

Since we have chosen $N$ as large as possible, Lemma \ref{lem:Nmaxgewaehlt} implies that $\pi(K)=\pi(K/C_K(N))$.
Let $q\in \pi(C_K(N))$. Then
the previous equation forces $q \in \pi(K/C_K(N))$, and therefore
a Sylow $q$-subgroup of $K$ has order at least $q^2$. In particular, for all non-nilpotent battens $V$ of $K$, we have that $\mathcal B(V)\nleq C_K(N)$.
It follows that $q\in \pi(K)^*$.
Now Lemma \ref{lem:neuproperties} provides $X,Y\in\{T,U,B\}\subseteq \mathcal L$ such that $X\neq Y$ and $O_q(C_K(N))\leq X\cap Y=E$.
Since $O_q(C_K(N))$ is characteristic in $C_K(N) \unlhd NK=G$,
it follows that $G/O_q(C_K(N))$ is not $L_{10}$-free. Moreover, $G/O_q(C_K(N))\in \FL$ by
Lemma \ref{lem:faktorStar}.
Since $G$ is a minimal counterexample, we conclude that $O_q(C_K(N))=1$.
We recall that $G$ is soluble, by Lemma \ref{merlonsol}, hence $C_K(N)$ is soluble, and then there must exist a prime $q \in \pi(C_K(N))$ such that
$O_q(C_K(N)) \neq 1$. This gives a contradiction, and therefore
$C_K(N)=1$.
\end{subproof}

We remark that, by (1), we may apply Lemmas \ref{lem:newunique} and \ref{lem:newPropertiesZinN}.

\medskip (2) For all $p \in \pi(N)$ we have that $O_p(N)\cap D\unlhd G$. In particular $N \cap D \unlhd G$ and $N \cap E=1$.

\begin{subproof}
Let $H\in\{E,D\}$, let $p\in \pi(N)$ and set $P:=O_{p}(N)$.
If $H=E$, then we set $\mathcal M:=\{U, T, B\}$ and otherwise we set $\mathcal M:=\{A, B, C\}$.
Then for all distinct $X,Y\in \mathcal M$, we have that $X \cap Y=H$ and $\erz{X,Y}=F$.

Assume for a first contradiction that $H\cap P$ is not a normal subgroup of $G$.
Since every subgroup of $N$ is normal in $N$ by  (1), it follows that  $K$ does not induce power automorphism on $P$.
Then Lemma \ref{lem:laction} implies that $K$ acts irreducibly on $\tilde P:=[P,K]/\Phi([P,K])$.
We apply Lemma \ref{lem:irredu} twice to find some $X\in\mathcal M$ and some $Y\in \mathcal M\setminus\{X\}$  such that $X$ and $Y$ act irreducibly on $\tilde P$.
In particular $[P,K]=[H\cap P,Y]=[H\cap P,X]\leq X\cap Y=H$.
It follows that $[P,K]\leq H\cap P\leq [P,K] C_P(K)$, which yields  that $H\cap P$ is normalised by $K$ and hence $H\cap P\unlhd NK=G$. This is a contradiction.
We deduce that $P \cap D \unlhd G$ as stated,
in particular $N \cap D \unlhd G$ and also
$N \cap E \unlhd G$.

For the final statement in (2) we use that $G/(N\cap E)$ is not $L_9$-free. Then the minimality of $G$ and Lemma \ref{lem:faktorStar} give that $N\cap E=1$.
\end{subproof}

(3) For every $q \in \pi(K)$ such that $1  \neq Q \in \syl_q(D)$, one of the following holds:

$[N,\Omega_1(Q)] \le D \le A$
or $K$ induces power automorphisms on  $[N,\Omega_1(Q)]$ and $[N,\Omega_1(Q)] \cap A \neq 1$.

Moreover $E=1$.

\begin{subproof}
We adopt the same notation as in the previous step, which means that
$H\in\{E,D\}$, $p\in \pi(N)$ and $P:=O_{p}(N)$.
If $H=E$, then $\mathcal M:=\{U, T, B\}$, and otherwise $\mathcal M:=\{A, B, C\}$.
Whenever $X,Y\in \mathcal M$ are distinct, then $X \cap Y=H$ and $\erz{X,Y}=F$.

Let $q \in\pi(K) \cap \pi(H)$ and  $Q\in\syl_q(D)$.
By conjugation we may suppose that $Q\leq K$.
Then $\Omega_1(Q)=\Omega_1(Q_0)$ for some Sylow $q$-subgroup $Q_0$ of $K$ by Lemma \ref{battenremark} and Lemma \ref{lem:newunique} provides some $p\in \pi(N)$ such that $1\neq [N,\Omega_1(Q)]$ is a $p$-group.
Now Lemma \ref{lem:newPropertiesZinN}~(e) states that $X=  (X\cap P)N_X(\Omega_1(Q))$ for all
 subgroups $X\in \mathcal M$ ($\ast$).
 %$X\in [G/H]=\{Y\in L(G)\mid H\leq Y\}$.

Let $X$ and $Y$ be distinct elements of $\mathcal M$ and assume for a contradiction that $X\cap P$ and $Y\cap P$ are subgroups of $C_P(K)$.
Then
$$G\stackrel{(1)}{=}F=\erz{X,Y} \stackrel{(\ast)}{\le} \erz{C_P(K), N_X(\Omega_1(Q)),N_Y(\Omega_1(Q))}\leq C_P(K) N_G(\Omega_1(Q))=N_G(\Omega_1(Q)),$$
which contradicts (1).

For the remainder of this proof we let $X$ and $Y$ in $\mathcal M$ be such that
their intersection with $P$ is not contained in $C_P(K)$.
We note that $C_P(K)=C_P(\Omega_1(Q))$ by Lemma \ref{lem:avoidL9Cen}.
Then it follows that $1\neq [P\cap X,\Omega_1(Q)]=[[P\cap X,\Omega_1(Q)],\Omega_1(Q)]$ and hence $[P,K]\cap X\nleq C_P(K)$. In a similar way we observe that $[P,K]\cap Y\nleq C_P(K)$.

If $K$  acts irreducibly on $[P,K]/\Phi([P,K])$, then Lemma \ref{lem:irredu} yields that $X$ or $Y$, say $X$, acts irreducibly on  $[P,K]/\Phi([P,K])$.
It follows that $[P,K]\leq X$ and  $[P,K]\cap Y\leq X\cap Y=H$.
Let $Z\in \mathcal M\setminus\{X,Y\}$. Then Lemma \ref{lem:irredu} yields that $Z$ or $Y$, say $Y$, acts irreducibly on $[P,K]/\Phi([P,K])$ and contains $[P,K]\cap Y$. Since $[P,K]\cap Y\nleq C_P(K)$, it follows that $[P,K]\leq V\cap X=H$.
By (1), this is only possible if $H=D$, in other words
$\pi(K) \cap \pi(E)=\varnothing$.
Then the fact that $E \cap N=1$ (see (1) once more) forces $E=1$.

The previous paragraph also gives that $[N,\Omega_1(Q)]=[P,\Omega_1(Q)]=[P,K]\leq D\leq A$, by Lemma \ref{lem:laction}, in the case where $K$ acts irreducibly on $[P,K]/\Phi([P,K])$.

Otherwise
Lemma \ref{lem:laction} gives that $K$ and hence $\Omega_1(Q)$ induce power automorphisms on $P$.
Together with (1) this means that every subgroup of $P$ is normal in $G$.
%and from Corollary \ref{cor:p=2nK}~(c) we see that $P=[P,K]$ is elementary abelian.
%In addition Lemma \ref{lem:avoidL9Cen} and Corollary \ref{cor:p=2nK}~(c) give  that $P\cap N_G(\Omega_1(Q))=C_P(\Omega_1(Q))=C_P(K)=1$.
Now, if $V,W\in \mathcal M$
are distinct, then
$$PN_G(P)=G=\erz{V,W}\stackrel{(\ast)}{\leq} (V\cap P)(W\cap P)N_G(\Omega_1(Q))$$ and it follows that $P=(V\cap P)(W\cap P)$.
We deduce that $|P|=\frac{|V\cap P|\cdot |W\cap P|}{|H\cap P|}$ for all $W,V\in \mathcal M$. In particular $|W\cap P|=|X\cap P|\neq 1$ for all $W\in \mathcal M$.
This implies all our claims about $D$, because $P=[P,K]=[P,\Omega_1(Q)]=[N,\Omega_1(Q)]$ by Lemma \ref{lem:laction} and $A\in \mathcal M$.

If, still in the power automorphism case, we have that $H=E$, then we recall that $E\cap P=1$ by (1).
Therefore
$$|(U\cap P)|\cdot |(T\cap P)|=\frac{|(U\cap P)|\cdot |(T\cap P)|}{|E\cap P|}=|(U\cap P)(T\cap P)|=|P|=|(U\cap P)(A\cap P)|$$
$$=\frac{|U\cap P|\cdot |A\cap P|}{|P\cap E|}=|U\cap P|\cdot |A\cap P|.$$
This implies that $A\cap P=T\cap P\neq 1$. But in this case we may interchange $T$ by $S$ in $\mathcal M$. Since $1\neq T\cap P=A\cap P=S\cap P$, we arrive at the contradiction that $1\neq S\cap T\cap P=E\cap P \le E \cap N=1$ (by (1)). Thus, we have that $\pi(K) \cap \pi(E)= \varnothing$ in this case as well. Again it follows that $E=1$.
\end{subproof}

\medskip
The remainder of the proof is dedicated to constructing a  subgroup of $G$ that violates Property ($\FL 4$).
\medskip

(4) If $p\in \pi(N)$, then $A\cap K$ does not centralise $O_p(N)$.
In particular $A\cap K\neq 1$.

\begin{subproof}
We assume for a contradiction that $A\cap K$ centralises $P:=O_p(N)$ for some $p\in \pi(N)$.

It follows from Lemma  \ref{lem:merlonsub} that, for every subgroup $X$ of $G=PO_{p'}(N)K$, there is some $x\in P$ such that $X\leq (X\cap P)O_{p'}(N)K^x$.
Since $[P,A\cap K]=1$, we further have THAT $X\leq (X\cap P)O_{p'}(N)K^y$ for all $X\leq A$ and $y\in P$ ($\ast$).

Assume that $A\cap P=1$. Then for both $X\in \{U,B\}$,
we see that $P\leq G\stackrel{(1)}{=}F=\erz{U,A}\stackrel{(\ast)}{\leq} (P\cap X)O_{p'}(N)K^x$ and therefore $P=P\cap U=P\cap B\leq U\cap B=E$. But this contradicts (3).

Thus $A\cap P\neq 1$ and then (1), together with our assumption at the beginning of the proof, imply that every subgroup of $A\cap P$ is normal in $A$.
Suppose that $X,Y\in\{S,T,D\}$ are distinct. We recall that $A\cap P\leq A=\erz{X,Y}\leq (X\cap P)(Y\cap P)O_{p'}(N)K$, and then it follows that $A\cap P=(X\cap P)(Y\cap P)$.
Since $X\cap Y=E\stackrel{(2)}=1$, we know more: $T\cap P\cong S\cap P\cong D\cap P$ and $|T\cap P|^2=|A\cap P|$.

If $K$ induces power automorphisms on $P$ and if $X\in \{A,T,S\}$, then $(X\cap P)$ and $(U\cap P)$ are normal subgroups of $G$ by (1).
In addition there is some $u\in P$ such that $U\leq (U\cap P)O_{p'}(N)K^u$ and then $X\leq (X\cap P)O_{p'}(N)K^u$  by ($\ast$).
We apply Lemma \ref{lem:powerneu} to see that $|P\cap A|=|P:P\cap U|=|P\cap S|=\sqrt{|P\cap A|}$. Now it follows that $P\cap A=1$, which is a contradiction.

We conclude that $K$ does not induce power automorphisms on $P$. Then  Lemma \ref{lem:laction} gives that $K$ acts irreducibly on $[P,K]/\Phi([P,K])$.
Since $P\cap D\unlhd G$ by (2), we see that either $[P,K]\leq D$ or that $1\neq P\cap D\leq  C_P(K)$.

In the first case $T\cap P\cong D\cap P\geq [P,K]$ and $T\cap D=E\stackrel{(2)}{=}1$. This implies, together with Lemma \ref{coprime}, that $C_P(K)$ has a subgroup isomorphic to $[P,K]$.
We apply Lemma \ref{lem:laction}, in combination with Part (b) of Lemma \ref{lem:avoid1prop}, and we deduce that $[P,K]$ is cyclic of order $2$ and that, therefore, $K$ centralises it. This is impossible.

It follows that the second case holds, i.e. $1\neq P\cap D\leq  C_P(K)$. Then $p=2$ by Corollary \ref{cor:p=2nK}~(a).
If $[P,K]$ is not elementary abelian, then Part (b) of Lemma \ref{lem:avoid1prop} and  Lemma \ref{lem:laction} give that $[P,K]\cong Q_8$ and $P=[P,Q]\times I$, where $I$ is a subgroup of $P$ of order at most $2$.
We note that $T\cap P\cong D\cap P$ and $T\cap D=E\stackrel{(3)}{=}1$, and then we conclude  that $T\cap P$ and $D\cap P$ are cyclic of order $2$. Consequently $A\cap P=(T\cap P)\cdot (D\cap P)=\Omega_1(P)\unlhd G$.
Now $1\stackrel{(3)}{=}E=U\cap A\geq U\cap \Omega_1(P)$ and therefore $U\leq O_{\pi'}(N)K^u$.
We arrive at a contradiction:  $P\leq G=\erz{U,A}\leq (P\cap A) O_{\pi'}(N)K^u$, because $[P,K]\cong Q_8$ is not elementary abelian.

So we finally have that $[P,K]$ is elementary abelian.
Then Lemma \ref{lem:avoid1prop}~(b) gives that $C_P(K)$ is cyclic and Lemma \ref{coprime}  shows that
$T\cap P\cong D\cap P\leq C_P(K)$.

Together with the fact that $T\cap D=E\stackrel{(3)}{=}1$, we obtain that $D\cap P$ is cyclic of order $2$.
It follows that $P\cap D$,  $T\cap P$ and $P\cap S$ have order $2$ and hence $P\cap A$ is elementary abelian of order $4$. Moreover $A\cap [P,K]$ is cyclic of order $2$, and therefore it equals one of the subgroups $T\cap P$, $S\cap P$ or $D\cap P$. The last case is not possible because $D\cap P\leq C_P(K)$.
By symmetry between $S$ and $T$ in the lattice we may suppose that $T\cap P\leq [P,K]$.
Recall that $K$ acts irreducibly on $[P,K]$ while $A\cap K$ centralises $P$.
In particular (1) implies that $A$ does not act irreducibly on $[P,K]$.
Thus Lemma \ref{lem:irredu}, together with the fact that $G\stackrel{(1)}{=}F=\erz{A,U}$, gives that $U$ acts irreducibly on $[P,K]$.
We also know that $T\cap U=E\stackrel{(3)}{=}1$, and this implies that $[P,K]\nleq U$. Then Lemma \ref{lem:MQvar} yields that $U\cap P\leq C_P(K)$. But we recall that $C_P(K)$ is cyclic, and its unique involution is contained in $D$. Then the fact that $U\cap D=E\stackrel{(3)}{=}1$ implies that $U\cap P=1$.
Finally, we see that $G\stackrel{(1)}{=}F=\erz{U,T}\stackrel{(\ast)}{\leq} [P,K]O_{p'}(N)K^u$, which gives a contradiction.\end{subproof}

\medskip
By conjugation and by Lemma \ref{lem:merlonsub} we may suppose that $A=(A\cap N)(A\cap K)$.
\\We set $\pi:=\pi(K)^* \cap \pi(A)$ and we let $L_1$ be a Hall $\pi$-subgroup of $A\cap K$.
\\Then we let
$$\sigma:=\{p\in \pi(N)\mid [O_p(N),Q]\neq 1 \text{ for some }Q\leq L_1,\text{ where }|Q|\in \pi\}.$$

\medskip (5) If $p\in\pi(N)$ and $[O_p(A),L_1]=1$, then $O_p(N)\leq D$.
Moreover $\pi \neq \varnothing \neq \sigma$.

\begin{subproof}
First suppose that $L_1$ centralises $P:=O_p(N)$ for some $p\in \pi(N)$.
Then  (4) implies that $A\cap K\neq L_1$ and then there is a non-nilpotent batten $V$ of $K$ such that $Q:=\mathcal B(V)\leq A$ and $[P,Q]\neq 1$.
Now $|Q|=q$ is a prime and $[N,Q]=[P,Q]$ by Lemma \ref{lem:newunique}.
In addition we see, from Definition \ref{defi:avoidbatten}, that $Q$ does not induce power automorphisms on $P$ and that $P=[P,Q]$.
If $q$ divides $|D|$, then (3) implies that $P=[N,Q]\leq D$, as stated.

Now we suppose that $q$ does not divide $|D|$.
Let $R\leq K$ be such that $V=QR$. Up to conjugation we may suppose that $R\cap L_1$ is a Sylow subgroup of $L_1$.
Then $R$ does not centralise $P$ by Definition \ref{defi:avoidbattengroup} and hence $R\nleq L_1$.
It follows that $R\nleq A$, from the definition of $L_1$.
Since $Q\cdot \Phi(R)=Q\cdot Z(V)$ is nilpotent by Lemma \ref{lem:batten}, we deduce that $A$ has a normal $q$-complement. Moreover $A\in \FL$, by Lemma \ref{lem:uG}, whence we may apply Lemma \ref{lem:neuproperties} to $A$. Then we see that $S$ and $T$ have orders divisible by $q$, because $q\notin \pi(D)$.
Let $t,s\in [P,Q]$ be such that $Q^s\leq S$ and $Q^t\leq T$.
The irreducible action of $Q$ on $P$, together with the fact that $P\cap T\cap S\leq E\stackrel{(3)}{=}1$, implies that $P\cap T=1$ or $P\cap S=1$.
Without loss $P\cap T=1$.
Then $T\leq N_G(Q^t)$ by Lemma \ref{lem:newPropertiesZinN}~(e).
Assume that $A$ normalises $Q^t$. Then $Q^s\leq N_G(Q^t)$, which implies that $[ts^{-1},Q]Q^s=\erz{Q^s,Q^t}\leq N_G(Q^t)$ by Lemma 4.1.1~(b) of \cite{Schmidt}. Then the irreducible action of $Q^t$ on $P$ forces $t=s$, contradicting the fact that $T\cap S=E\stackrel{(3)}{=}1$.

Thus $A$ does not normalise $Q^t$ and $\erz{T,D}=A\nleq N_G(Q^t)$, which yields that $D\nleq N_G(Q^t)=(O_{p'}(N)K)\times C_{P}(K)=O_{p'}(N)K$ by Part (a) of Lemma \ref{lem:newPropertiesZinN}.
It follows that $D\cap P\neq 1$ and therefore $P\leq D$, because $P\cap D\unlhd G$ by (2) and $K$ acts irreducibly on $P$.

\smallskip  We turn to the second statement and assume for a contradiction that $\pi=\pi(A)\cap \pi(K)=\varnothing$.
Then $A \cap K=1$, contrary to (4).
Hence if $\pi=\pi(A)\cap \pi(K)^*=\varnothing$, then
we can draw two conclusions:
First $L_1=1$ and the statement we just proved gives that $N \le D$.
Second, there must be a prime in $\pi(K) \setminus \pi(K)^*$ dividing $|A|$.
By definition of $\pi(K)^*$, such a prime is $|\mathcal B(V)|$ for some non-nilpotent batten $V$ of $K$.
We choose such a non-nilpotent batten $V$ and set $Q:=\mathcal B(V)$.
Then there are some prime $r$ and an $r$-subgroup $R\leq K$ such that $QR=V$, and $r \in \pi(K)^*$
because of the structure of non-nilpotent battens (Definition 2.1).
In particular $r\nmid |A|$ by assumption and Lemma \ref{lem:neuproperties} yields that $B$ and $U$ contain a conjugate of $R$.
We recall that $N\leq D\leq B$
by the first consequence of our assumption and because of the structure of the lattice $L_9$.
Then Lemma \ref{lem:newPropertiesZinN}~(c) gives that $\Omega_1(R)^G\subseteq B$.
Thus $B\cap U=E\stackrel{(3)}{=}1$, which is a contradiction.
This proves that $\pi \neq \varnothing$.

If $\sigma =\varnothing$, then for all
$p \in \pi(N)$, all $q \in \pi$ and all $q$-subgroups $Q$ of $L_1$, we have that $[O_p(N),Q]=1$.
Then $[N,L_1]=1$ by definition of $L_1$, and $L_1 \neq 1$ because $\pi \neq \varnothing$.
But then $1 \neq L_1 \le C_K(N)$, contrary to (1).
\end{subproof}

Next we set $H:=O_\sigma(N)$
and we prove that $H$ is a candidate for the desired properties in ($\FL 4$).

\smallskip
(6) The non-trivial  Sylow subgroups of $H$ are elementary abelian (in particular $N$ is abelian), but not cyclic, and $K\leq  \mathrm{Pot}_K(H)$.

\begin{subproof}By definition $H\leq N$ is nilpotent.
If, for all $p\in\sigma$, the group $K$ does not act irreducibly on $[O_p(N),K]/\Phi([O_p(N),K])$, then
 Lemma \ref{lem:laction} gives that $K\leq  \mathrm{Pot}_K(H)$. In particular $O_p(N)$ is not cyclic.
Moreover, Corollary \ref{cor:p=2nK} yields that $O_p(N)=[O_p(N),K]$ is elementary abelian.
Therefore, our claim is satisfied in this case.

Let us assume for a contradiction that there is some $p\in\sigma$ such that $K$ acts irreducibly on the group $[O_p(N),K]/\Phi([O_p(N),K])$.
We set $P:=O_p(N)$ and we choose $Q\leq L_1$ of order $q\in \pi$ such that $[P,Q]\neq 1$, by the definition of $\sigma$.

\textbf{Case 1:} $[P,Q]\leq A$.

Then Lemma \ref{lem:newPropertiesZinN}~(c) implies that $Q^g\leq A$ for every $g\in G$.
We recall that $U\cap A=E\stackrel{(3)}{=}1$, and then it follows first that $q \notin \pi(U)$ and then that $q\mid |B|$, by Lemma \ref{lem:neuproperties}. Here we use that $G=F=\erz{B,U}$ by (1).
Thus we find some $g\in G$ such that $Q^g\leq D$, because $D=A\cap B$.
We apply (3) to observe that  $[P,Q]=[N,Q]\leq D$ and then $D$ contains every conjugate of $Q$ by Lemma \ref{lem:newPropertiesZinN}~(c).
Recall that $q\notin \pi(U)$, which implies that $q\in \pi(T)$ by Lemma \ref{lem:neuproperties}. Again we use that $G=F=\erz{T,U}$. But this contradicts the fact that $T\cap D=E=1$ by (3).

\textbf{Case 2:} $[P,Q]\nleq A$.

Then $[P,Q]\nleq D$, in particular $P \nleq D$, and $P\cap D\unlhd G$ by (2).
This means that $K$ stabilises the subgroup $[P \cap D,K]/$
of $[P,K]/\Phi([P,K])$, while acting irreducibly. This forces
$[P \cap D,K] \le \Phi([P,K])$, and together with coprime action (Lemma \ref{coprime}) we see that $K$ centralises $P \cap D$.

By Lemma \ref{lem:merlonsub} we know that $A\in \FL$, with type $(A\cap N,A\cap K)$. Additionally, since $q\in \pi(K)^*$, the group $K$ has a normal $q$-complement.
Then $A\cap K$ also has a normal $q$-complement.
We may apply Lemma \ref{lem:neuproperties} to  $A=\erz{T,S}=\erz{T,D}=\erz{S,D}$.
It yields that at least two of the groups $D$, $T$, $S$ have a subgroup of order $|Q|$.
As $|Q|\nmid |E|$ by (2), there is some $g\in G$ such that $Q^g\neq Q$ and $Q^g\leq A$.
Then Part (e) of Lemma \ref{lem:newPropertiesZinN} implies that $P\cap A\nleq C_P(Q)=C_P(K)$ by Lemma \ref{lem:avoidL9Cen}.
In the present case we have that $[P,A]\nleq A$, and then Lemma \ref{lem:MQvar} gives that $A$ does not act irreducibly on $[P,K]/\Phi([P,K])$.
Thus $A\cap K$ induces power automorphism on $P$ by Lemma \ref{lem:laction}, and these automorphisms are not trivial because $Q\leq A\cap K$.
Corollary \ref{cor:p=2nK}~(c) implies that $P=[P,A\cap K]\stackrel{\ref{lem:laction}}{=}[P,K]$ is elementary abelian and in particular that $D\cap P\leq C_P(K)\stackrel{\ref{lem:avoidL9Cen}}{=}C_P(A\cap K)=1$. Hence there is some $d\in [P,K]$ such that $D=N_D(Q^d)$, by Lemma \ref{lem:newPropertiesZinN}~(b).
In addition we see from Lemma \ref{lem:irredu} that $B$ and $C$ act irreducibly on $P$.
If it was true that $P\le B$, then it would follow that $1\neq P\cap A\leq P\cap B\cap A\leq P\cap D\leq C_P(K)=1$, which is a contradiction.
We conclude that $P \nleq B$  and hence $P\cap B=1$ because of the irreducible action of $B$ on $P$, and similarly
%$B\cap P=1$ and similarly we see that $C\cap P=1$.
$P\cap C=1$.

Then Lemma \ref{lem:newPropertiesZinN}~(b) provides $b,c\in P$ such that $B=N_B(Q^b)$ and $C=N_C(Q^c)$.
If $Q^b=Q^c$, then $G=F=\erz{B,C}\leq N_G(Q^b)$, which is false.

Consequently $Q^b \neq Q^c$, and we can use Lemma \ref{lem:newPropertiesZinN}~(a), Dedekind's modular law  and Lemma \ref{lem:Hilfslemma}~(c).
Together this shows that
$$D\leq N_B(Q^b)\cap N_C(Q^c)=O_{p'}(N)K^b\cap O_{p'}(N)K^c=O_{p'}(N)(K^b\cap O_{p'}(N)K^c)$$
$$\leq
O_{p'}(N)C_K(bc^{-1})
\leq C_G(bc^{-1}),$$ because $N$ is nilpotent and because $\erz{(bc^{-1})^K}\cap O_{p'}(N)\leq P\cap O_{p'}(N)=1$.

We recall that $P=[P,K]$ is abelian, hence it is contained in $Z(N)$, and then it follows that
$D$ centralises $1\neq bc^{-1}\in [P,K]$. Here we also use Lemma \ref{lem:laction}, i.e. that $D$ centralises $P$.
We conclude that $D=N_D(Q^h)$ for all $h\in P$.
 Again Lemma \ref{lem:newPropertiesZinN}~(b) provides $t,s\in P$ such that $T=(T\cap P)N_T(Q^t)$ and $S=(T\cap S)N_C(Q^s)$.
Let $X\in \{T,S\}$.
Then  $P\cap A\leq A=\erz{D,X}\leq (X\cap P)N_G(Q^x)$, which implies that $P\cap X=P\cap A$.
But now $P\cap A\leq P\cap S\cap T=P\cap E=1$, by (3), which gives a final contradiction.
\end{subproof}

(7) For all $p\in\sigma$ we have that $O_{p}(N)\cap A\neq 1$.

\begin{subproof} We set $P:=O_p(N)$ for some $p\in \sigma$. Then there is some $Q\leq L_1$ of order $q\in \pi$ such that $[P,Q]\neq 1$, by the definition of $\sigma$.
Then $q\in \pi(K)^*$ and we see that $K$ as well as $K\cap A$ have a normal $q$-complement.
Since $A\in \FL$ by Lemma \ref{lem:uG} ,we may apply Lemma \ref{lem:neuproperties}.
We notice that $A=\erz{D,T}=\erz{D,S}$, and then it follows that $q\in \pi(D)$ or that $q \in \pi(T) \cap \pi(S)$.
In the first case (3) implies our assertion.
In the second case we use Lemma \ref{lem:newPropertiesZinN}~(c). It gives that $Q^s\leq S$ and $Q^t\leq T$ for some $t,s\in [P,K]$.
We deduce that $[t^{-1}s,Q]\leq \erz{Q^s,Q^t}\leq A$ by Lemma 4.1.1 of \cite{Schmidt}.
In conclusion, our assertion is true or $t^{-1}s$ centralises $Q$. But in the second case, we see that $Q^t=Q^s\leq T\cap S=E\stackrel{(3)}{=}1$, and this gives a contradiction.
\end{subproof}

(8)  For all $q\in \pi$ we have that $q$ divides $|S|$, $|T|$ and $|B|$.% but $q\nmid |U|$.
\\Furthermore, if $Q\leq L_1$ and $|Q|=q$, then
$[N,Q]\cap T=[N,Q]\cap S=1$, but $[N,Q] \cap U\neq 1$.

\begin{subproof}Suppose that $Q\leq L_1$ has order $q$. Then $[N,Q]$ is a $p$-group by Lemma \ref{lem:newunique}, for some prime $p\in \sigma$. We set $P=O_p(N)$.
Then for every $X\in \{T,S\}$ there is some $x\in P$ such that  $X=(X\cap P)N_X(Q^x)$, by Lemma \ref{lem:newPropertiesZinN}~(b).
Using (6) we see that $K$ induces power automorphisms on $P$.
Thus $P=[P,K]$ is elementary abelian by Corollary \ref{cor:p=2nK}~(c), and then Lemma \ref{lem:powerneu} is applicable.

Assume for a contradiction that $P\gneq (P\cap A)(P\cap U)$.
Then, for all $X\in\{T,S,A\}$, we have that
$$P\cap \erz{X,U}=P\cap F\stackrel{(1)}{=}P >(P\cap A)(P\cap U)\geq (P\cap X)(P\cap U).$$
Hence Lemma \ref{lem:powerneu} yields that $|P| = |(P\cap X)(P\cap U)| p\stackrel{(2)}{=}|(P\cap X)||(P\cap U)| p$ and, for all $X\in \{T,S,A\}$, we see that $p\cdot |P\cap X|=|P:P\cap U|$.
In particular, we have that $|P\cap A|=|P\cap S|=|P\cap T|$. Therefore, the fact that $T,S\leq A$ implies that $P\cap A=P\cap T=P\cap S\leq P\cap (T\cap S)=P\cap E\stackrel{(3)}{=}1$. This contradicts (7).

We conclude that $P=(P\cap A)(P\cap U)$ and hence Lemma \ref{lem:powerneu} provides some
element $g\in P$ such that,
for both $X\in\{A,U\}$, it is true that
$X\leq (X\cap P)O_{p'}(N)K^g=(X\cap P)N_G(Q)^g$.

 Assume for a contradiction that $q\in \pi(U)$. Then Part (e) of Lemma \ref{lem:newPropertiesZinN}  yields that $Q^g\in A\cap U=E\stackrel{(3)}{=}1$. This is impossible.
Thus $q\notin \pi(U)$  and we deduce
that $q\in \pi(X)$ for all $X\in\{T,S,B\}$. Here we use Lemma \ref{lem:neuproperties} and the fact that $G\stackrel{(1)}{=}F=\erz{U,X}$.
Next we apply Lemma \ref{lem:powerneu}, which gives that $(P\cap A)\neq (P\cap S)(P\cap T)$.
Then the same lemma yields, for both combinations of $\{X,Y\}\in \{T,S\}$, that
$p$ divides
$$|P\cap Y||P\cap X||P\cap U|=|P\cap A||P\cap U|=|P|\leq |P\cap X||P\cap U|\cdot p.$$
It follows that $p\cdot |P\cap Y|\leq p$ and then that
$S\cap P=T\cap P=1$.
Finally (6) and Lemma \ref{lem:powerneu} yield that $p^2\leq |P|\leq |P\cap T||P\cap U|\cdot p=|P\cap U|\cdot p$. We conclude that $P\cap U\neq 1$.
\end{subproof}

(9) For every $p\in \sigma$ there is some  $r\in\pi(K)^*$ such that a Sylow $r$-subgroup of $U$ does not centralise $O_p(N)$.

\begin{subproof}We set $P:=O_{p}(N)$. Since $p\in\sigma$, there is some $Q\leq L_1$ such that $|Q|=q\in \pi$.
In addition (6) implies that $P$ is elementary abelian, and then  $P\leq Z(N)$ because $N$ is nilpotent.
By Lemma \ref{lem:newPropertiesZinN}~(b) and (8) there is some $s\in P$ such that $S=N_S(Q^s)$.

Assume for a contradiction that  $U$ centralises $P$.
Then $U=U^s=(U\cap P)N_U(Q^s)$ by Lemma \ref{lem:newPropertiesZinN}~(d) and hence $G\stackrel{(1)}{=}F=\erz{U,S}\leq (U\cap P)N_G(Q^s)$ implies that $U\cap P=P$. With (7) we obtain the contradiction that $1\neq P\cap A\leq U\cap A=E\stackrel{(3)}{=}1$.

It follows that $U$ does not centralise $P$. We recall that $P\leq Z(N)$ and then we obtain  a prime  $r\in\pi(K)$ such that a Sylow $r$-subgroup of $U$ does not centralise $P$.
Since $K$ induces power automorphisms on $P$ by (6), Definition \ref{defi:avoidbattengroup} gives, for every non-nilpotent batten $V$ of $K$, that $\mathcal B(V)$ centralises $P$.
This implies that $r\in \pi(K)^*$.
\end{subproof}

We are now able to define $L$ and $J$. Let $u\in N$ be such that $U=(U\cap N)(U\cap K^u)$.

We  set $\tilde \pi:=\{r\in \pi(K)^*\mid [H,R]\neq 1 \text{ for some }R\in\syl_r(U)\}$ and we let $L_2$ be a Hall $\tilde \pi$-subgroup of $(U^{u^{-1}}\cap K)$.
Then $L_2\leq K$ and $L=\erz{L_1,L_2}$ is a subgroup of $K$.
In addition (5) and (9) show that $\tilde \pi \neq \varnothing$.
\\Next we set $\rho:=\{p\in \pi(N)\mid [O_p(N),R]\neq 1 \text{ for some }R\leq L_2\text{ with }|R|\in \tilde \pi\}$ and $J:=[O_\rho(N),K]$.
Then $\rho\neq \varnothing$ by Lemma \ref{lem:newunique}, because $\tilde \pi\neq \varnothing$, and hence $J\neq 1$.
Finally, we note that $J$ is $L$-invariant by construction.

\smallskip
(10) $\sigma\cap \rho=\varnothing$, i.e. $|H|$ and $|J|$ are coprime.

\begin{subproof} We assume for a contradiction that the prime $p$ divides $|H|$ and $|J|$.
Then by definition there are $q\in \pi$ and $r\in \tilde \pi$ and subgroups $Q\leq L_1$ and $R\leq L_2$ such that the following hold:

$|Q|=q$, $|R|=r$, $1\neq [N,Q]\leq O_p(N)=:P$ and $1\neq  [N,R]\leq P$.

In particular (6) shows that $K$ induces power automorphisms on $P$.
It follows from Corollary \ref{cor:p=2nK}~(c) and Lemma \ref{lem:laction} that $P=[P,K]=[P,Q]=[N,Q]=[N,R]$, and then (8) shows that $P\cap S=1$.
Moreover $P\cap A\neq 1$ by (7).
We apply Lemma \ref{lem:powerneu} to obtain some $i,j\in\{1,0\}$ such that
$$p^i|P\cap U|=|(P\cap U)(P\cap S)|p^i=|P|=|(P\cap U)(P\cap A)|p^j\stackrel{(2)}{=}|(P\cap U)||(P\cap A)|p^j.$$
This implies that $|P\cap A|p^j=p^i$, and we obtain that $i=1$ and $j=0$.

In particular we see that $P=(U\cap P)(A\cap P)$.
Then Lemma \ref{lem:powerneu} and the fact that
$A\cap U=E\stackrel{(3)}{=}1$
give some element $g\in P$ such that $X\leq (P\cap X)K^g=(P\cap X)N_G(R)^g$, where $X\in\{A,U\}$.
Assume for a contradiction that $r\in \pi(A)$. Then \ref{lem:newPropertiesZinN}~(d) forces $R^g\leq U\cap A=E\stackrel{(3)}{=}1$. This is impossible, hence
$r\notin\pi(A)$ and we apply Lemma \ref{lem:neuproperties} to  $G\stackrel{(1)}{=}F=\erz{A,B}$. Then it follows that $r\in \pi(B)$.

Moreover $q\in \pi(B) \cap \pi(T)$ by (8).
Since $B\cap T=E\stackrel{(3)}{=}1$, this is not possible, hence Lemma \ref{lem:powerneu} and (8) give that $|P|=|P\cap B|\cdot |P\cap T|\cdot p=|P\cap B|\cdot p$.
On the other hand we have that $r\in \pi (B) \cap \pi(U)$, which is also impossible because $B\cap U=E\stackrel{(3)}{=}1$. Now Lemma \ref{lem:powerneu} implies that $|P\cap U|\cdot |P\cap B|\cdot p=|P|=|P\cap B|\cdot p$.
In particular we see that $P\cap U=1$, and this contradicts  (8).
\end{subproof}

We summarize:

The definitions of $\sigma$ and $H$ imply that $H\leq N$, and (6) shows that the non-trivial  Sylow subgroups of $H$ are not cyclic.
In addition $L\leq K$ induces power automorphisms on $H$.

From (10) we deduce that $\pi\cap \tilde\pi=\varnothing$.
A Hall $\pi(K)^*$-subgroup of $K$ is nilpotent by Lemma \ref{lem:battensub}, which means that $L$ is nilpotent. In particular we see that $L=L_1\times L_2$.
Let $r\in \pi(L)=\pi\cup\tilde \pi$ and $R=O_r(L)$.
If $r\in \pi$, then $[H,R]\geq [[N,\Omega_1(R)],R]\neq 1$, and if $r\in \tilde p$, then $[H,R]\neq 1$ by (9).
These arguments show that $\pi(L)=\pi(L/C_L(H))$ ($\ast$).

We assume for a contradiction that $L$ is not cyclic.
Then, since $L$ is nilpotent and a batten group by Lemma \ref{lem:battensub}, we deduce that
$O_2(L)\cong Q_8$.
Definition \ref{defi:avoidQ8} implies that, for all $p\in\pi(N)$, the group $O_2(L)$ does not induce non-trivial power automorphisms on $O_p(N)$.
In particular $O_2(L)$ centralises $H$,
which contradicts ($\ast$).
Thus $L$ is cyclic.

Furthermore, we already saw that $J$ is $L$-invariant and that $1\neq J$. Then (10) gives that
$(|H|, |J|)=1$, and since $N$ is nilpotent, this forces $[H,J]=1$.

Assume for a contradiction that $J$ is not abelian.
Then Lemma \ref{lem:avoid1prop}~(b) and Lemma \ref{lem:laction} show that $O_2(J)=[O_2(N),K]\cong Q_8$ and therefore $2\in\rho$.
We let $r\in \tilde \pi$ be such that a Sylow $r$-subgroup $R$ of $K$ acts faithfully on $O_2(J)$. Then we must have that $|R|=3$ because
$O_2(J)\cong Q_8$.
Moreover Lemma \ref{lem:newPropertiesZinN}~(a) yields that $R$ centralises $O_{2'}(N)\geq H$. This contradicts ($\ast$).

Hence $J$ is abelian.
For every $q\in\pi$ and every subgroup $Q\leq L_1$ of order $q$, the definition of $\sigma$ and Lemma \ref{lem:newunique} provide some $p\in \sigma$ such that $[N,Q]\leq O_p(N)=P$.
Then we see, using (6) and Corollary \ref{cor:p=2nK}, that $C_P(Q)=1$.
Furthermore (10) yields that $Q$ centralises $J=O_{\rho}(N)$, and then
it follows that $1=C_P(Q)\geq C_P(C_L(J))$.
Since $H=O_{\sigma}(N)$ is abelian, we conclude that $C_H(C_L(J))=1$.

The previous argument also yields that $\{q\in \pi(L)\mid  C_{O_q(L)}(H) < C_{O_q(L)}(J)\}\supseteq \pi$.
Let $r\in \tilde \pi$ and suppose that $R\leq L$ has order $r$. We recall the definition of $\rho$ and apply Lemma \ref{lem:newunique}: Then we see that $1\neq [N,R]\leq O_{\rho}(N)=J$.
Thus $r\notin \{q\in \pi(L)\mid  C_{O_q(L)}(H) < C_{O_q(L)}(J)\}$ and it follows that $\{q\in \pi(L)\mid  C_{O_q(L)}(H) < C_{O_q(L)}(J)\}=\pi$.

\medskip Since $G\in \FL$, we can use Property $(\FL 4)$, which gives some $g\in (HJ)^\#$ that centralises  $O_\pi(L)=:L_1$ or $O_{\pi'}(L)=:L_2$.
Let $i \in \{1,2\}$ be such that
$[L_i,g]=1$.
We may suppose that $g$ has prime order $p$. Then $p\in \sigma\cup \rho$ and hence there is a subgroup $Q\leq L$ of prime order such that $[N,Q]\leq O_p(N)=:P$.
%We set $Q_0:=O_q(L)$.
Lemma \ref{lem:avoidL9Cen} gives  that $g\in C_P(L_i)=C_P(K)$ or that $L_i$ centralises $P$.

In the first case Corollary \ref{cor:p=2nK} yields that $p=2$, and then $K$ does not induce power automorphisms on $P$.
Using (6) we deduce that $p\in \rho$ and $g\in O_p(J)\cap C_P(K)=[P,K] \cap C_P(K)\cap J$.
Since $J$ is abelian, we obtain a contradiction in this case.

It follows that the second case above holds, i.e. $[P,L_i]=1$. This means that $i=1$ if $p\in \rho$ and $i=2$ if $p\in \sigma$.
In addition we see, from (9), that $i\neq 2$.
We conclude that $L_i=L_1$, $p\in \rho$ and $q\notin \pi$.
In particular $q\notin \pi(A)$.
Since $G\stackrel{(1)}{=}F=\erz{A,U}=\erz{A,B}$, Lemma \ref{lem:neuproperties} implies that $B$ and $U$ contain a conjugate of $Q$.
In addition (5) shows that $P\leq D\leq B$ and therefore Lemma \ref{lem:newPropertiesZinN}~(c)  gives that $Q^G\subseteq B$.
Finally, we obtain a contradiction, because $B\cap U=E=1$ by (3). This concludes the proof.
\end{proof}

\begin{mainthm}
A finite group is in  $\FL$ if and only if it is $L_9$-free.
\end{mainthm}

\begin{proof}Let $G$ be a finite group.
If $G$ is $L_9$-free, then Theorem \ref{main} shows that $G \in \FL$.

Conversely, if $G \in \FL$, then $G$ is $L_9$-free by Theorem \ref{thm:Rueck}.
\end{proof}

\end{document}